\documentclass[10pt,letterpaper]{article}
\usepackage[left=0.9in,right=0.9in,top=1in,bottom=1in]{geometry}
\usepackage{amsmath}
\usepackage{amsthm,alltt} 
\usepackage{amssymb}

\usepackage{graphicx}
\usepackage[T1]{fontenc} 
\usepackage[utf8]{inputenc}

\usepackage{hyperref}
\usepackage{thm-restate}
\usepackage{xcolor}
\usepackage{graphicx}
\usepackage{subcaption}
\usepackage{enumitem}

\usepackage{xfrac}
\newcommand{\st}{\mathrm{s.t.}}
\newcommand{\ONE}{\mathbf{1}}

\newcommand{\Diff}{\mathrm{d}}
\newcommand{\Real}{\mathbb{R}}
\newcommand{\Ze}{\mathbb{Z}}
\newcommand{\conv}{\mathrm{conv}}
\newcommand{\tr}{^{\intercal}}

\newtheorem{assumption}{Assumption}
\newtheorem{notation}{Notation}
\newtheorem{example}{Example}
\newtheorem{definition}{Definition}
\newtheorem{remark}{Remark}
\newtheorem{proposition}{Proposition}
\newtheorem{lemma}{Lemma}
\newtheorem{claim}{Claim}

\def\i0{i_0}

\begin{document}

\title{Lifting convex inequalities for bipartite bilinear programs}



\author{Xiaoyi Gu\footnote{xiaoyigu@gatech.edu, Georgia Institute of Technology}%\orcidID{0000-0002-1020-1186} 
\and
Santanu S. Dey\footnote{santanu.dey@isye.gatech.edu, Georgia Institute of Technology}%\orcidID{1111-2222-3333-4444} 
\and
Jean-Philippe P. Richard\footnote{jrichar@umn.edu, University of Minnesota, Minneapolis}%\orcidID{2222--3333-4444-5555}
}



\maketitle

\begin{abstract}
The goal of this  paper is to derive new classes of valid convex inequalities for quadratically constrained quadratic programs (QCQPs) through the technique of lifting. 
Our first main result shows that, for sets described by one bipartite bilinear constraint together with bounds, it is always possible to sequentially lift a seed inequality that is valid for a restriction obtained by fixing variables to their bounds, when the lifting is accomplished using affine functions of the fixed variables. In this setting, sequential lifting involves solving a non-convex nonlinear optimization problem each time a variable is lifted, just as in Mixed Integer Linear Programming. To reduce the computational burden associated with this procedure,  we develop a  framework based on subadditive approximations of lifting functions that permits sequence-independent lifting of seed inequalities for separable bipartite bilinear sets.  In particular, this framework permits the derivation of closed-form valid inequalities. We then study a separable bipartite bilinear set where the coefficients form a minimal cover with respect to the right-hand-side. 
For this set, we introduce a \textit{bilinear cover inequality}, which is second-order cone representable. We argue that this bilinear cover inequality is strong by showing that it yields a constant-factor approximation of the convex hull of the original set. 
We study its lifting function and construct a two-slope subadditive upper bound. 
Using this subadditive approximation, we lift fixed variable pairs in closed-form, thus deriving a \textit{lifted bilinear cover inequality} that is valid for general separable bipartite bilinear sets with box constraints.

%\keywords{Lifting \and Separable bipartite bilinear sets \and Subadditivity}
\end{abstract}

%\allowdisplaybreaks

%\input{Sec0_Intro.tex}
\section{Introduction}

\subsection{Generating strong cutting planes through lifting}

\textit{Lifting} is a technique that is used to derive or strengthen classes of cutting planes. 
It was first introduced to optimization in the context of mixed integer linear programming (MILP); see~\cite{richard2010lifting} for a review. 
The lifting process has two steps:
\begin{itemize}
\item \textit{Fixing} and generation of a \textit{seed inequality}: 
In the first step, the set $S$ of interest is restricted by fixing a subset of variables, say $x^F$, to specific values (typically to one of their bounds), say $\tilde{x}^F$. 
A valid inequality $h(x) \geq h_0$, which we call \textit{seed inequality}, is then generated for the restriction $S|_{x^F = \tilde{x}^F}$.  
\item \textit{Lifting} the seed inequality: 
The seed inequality $h(x) \geq h_0$, when viewed with ``zero coefficients" for the fixed variables $h(x) + 0\cdot x^F\geq h_0$ is typically not valid for the original set $S$. 
The task in the lifting step is to generate an inequality $h(x) + g(x^F) \geq h_0 + g_0$, which (i) is valid for $S$ and (ii) satisfies $g(\tilde{x}^F)  = g_0$.
Under condition (ii), inequality $h(x) + g(x^F) \geq h_0 + g_0$ reduces to inequality $h(x) \geq h_0$ when $x^F$ is set to $\tilde{x}^F$. 
The process of lifting is often accomplished by rotating or titling the seed inequality~\cite{espinoza2010lifting}.
    %-- which algebraically amounts to changing the coefficients of fixed variables and/or updating the right-hand-side) so that the resulting inequality is valid for the original set.   
\end{itemize}
Though condition (ii) is not strictly necessary to impose, we require it in the remainder of the paper as otherwise $h(x) + g(x^F) \geq h_0 + g_0$ is weak on the face $x^F = \tilde{x}^F$.
%The condition $g(\tilde{x}^F)  = g_0$ is not necessary, but without this condition the inequality , which is not ideal. As such, we will require the condition  $g(\tilde{x}^F)  = g_0$ as a part of lifting.
%See Padberg~\cite{padberg1973facial,padberg1975note}, Balas~\cite{balas1975facets}, Hammer et al.~\cite{hammer1975facet}, Wolsey~\cite{wolsey1976facets,wolsey1977valid}, Balas and Zemel~\cite{balas1978facets}, Ceria et al.~\cite{ceria1998cutting}, Martin and Weismantel~\cite{martin1998intersection}, Gu et al.~\cite{gu2000sequence,gu1999lifted}, Richard et al.~\cite{richard2003liftedsl,richard2003liftedbta}, Atamt\"urk~\cite{atamturk2003facets,atamturk2004sequence}, Agra and Constantino~\cite{agra2007lifting}, Kaparis and Letchford~\cite{kaparis2008local}, Dey and Richard~\cite{dey2009linear}, Zeng and Richard~\cite{zeng2007framework,zeng2011polyhedral,zeng2011polyhedrala}, and Gomez~\cite{gomez2018submodularity} for examples of use of the lifting technique technique for generation of cutting-planes for mixed integer programs. 

Lifting, as a technique for generating cutting-planes in MILP, has been extensively researched.
Originally devised for node packing and knapsack sets ~\cite{padberg1973facial,padberg1975note,balas1975facets,hammer1975facet,balas1978facets}, lifting was extended to general settings \cite{wolsey1976facets,wolsey1977valid,gu1999lifted,gu2000sequence,richard2003liftedbta,richard2003liftedsl,atamturk2004sequence}
and used to derive families of valid inequalities for many sets including
\cite{ceria1998cutting,martin1998intersection,atamturk2003facets,agra2007lifting,kaparis2008local,dey2009linear,zeng2007framework,zeng2011polyhedral,zeng2011polyhedrala,gomez2018submodularity}.   
Many of the classes of cutting planes that have 
% \red{
yielded
% } 
% \footnote{\red{JP: REPLACED THIS: proven to yield}} 
significant computational gains can be obtained through lifting.
This includes \textit{lifted cover inequalities}~\cite{gu1999lifted}, lifted tableaux cuts~\cite{dey2009linear,narisetty2011lifted}, and even the \textit{Gomory mixed integer cut}~\cite{dey2010two}; see \cite{balas1980strengthening,richard2009valid,dey2010constrained,conforti2011geometric,basu2012unique,basu2013unique,basu2019nonunique,averkov2015lifting,basu2015operations,dey2010composite,basu2019nonunique}
for papers related to lifting in the infinite group problem model.
Similarly, \textit{mixing inequalities}~\cite{gunluk2001mixing} can be viewed as an outcome of lifting~\cite{dey2010composite}. 

Significantly fewer articles have focused on studying how lifting can be applied to nonlinear programs and mixed integer nonlinear programs.
Exceptions include~\cite{richard2010liftingframework}, which develops a general theory for lifting linear inequalities in nonlinear programming, \cite{nguyen2018deriving} which applies lifting to derive the convex hull of a nonlinear set, \cite{gupte2012mixed} which studies lifting for the pooling problem, \cite{atamturk2011lifting} which uses lifting for conic integer programs, and \cite{chung2014lifted} which develops strong inequalities for mixed integer bilinear programs.  

\subsection{Goal of this paper}
%In this paper, we present a new class of valid convex inequalities for a certain type of quadratically constrained quadratic programs (QCQP). 
%\footnote{\red{JP: I think we probably want to be careful about not overselling it since there is in mind mind a big difference between quadratic and bilinear...}}
The goal of this  paper is to derive  new  classes  of  valid  convex  inequalities  for \textit{quadratically constrained quadratic programs} (QCQPs) through the technique of lifting.

Generating valid inequalities for single row relaxations (together with bounds and integrality restrictions), \textit{i.e.}, for knapsack constraints, was the first, and arguably the most important step in the development of computationally useful cutting-planes in MILP. 
Motivated by this observation, various cutting-planes and convexification techniques for sets defined by a single non-convex quadratic constraint together with bounds have recently been investigated;  see ~\cite{chung2014lifted,tawarmalani2010strong,anstreicher2020convex} for classes of valid inequalities for single constraint QCQPs and~\cite{dey2019new,santana2020convex} for convex hull results for such sets. 
The paper~\cite{rahman2019facets} studies a set similar to the one we study, albeit with integer variables. Further, \cite{dey2019new} demonstrates that cuts obtained from one-row relaxations of QCQPs can be useful computationally. 
The paradigm of intersection cuts has also been explored to generate cuts for single-constraint QCQPs~\cite{munoz2020maximal,bienstock2020outer}. 
Due to lack of space, we refrain from describing here the vast literature on convexification techniques for QCQPs  and instead refer interested readers to~\cite{burer2015gentle,santana2020convex} and the references therein.
%presented within these papers.

In this paper, we investigate the lifting of a convex seed inequality for a feasible region defined by a single (non-convex) quadratic constraint together with bound constraints. 
Apart from \cite{atamturk2011lifting}, we are not aware of any paper that attempts to study or employ lifting of convex nonlinear inequalities. 
To the best of our knowledge, this is the first study that derives lifted valid inequalities for general non-convex quadratic constraints with arbitrary number of variables.  
An extended abstract of this paper was accepted for publication in IPCO 2021~\cite{GuIPCO2021}.
\subsection{Main contributions}
\begin{itemize}
\item \textit{Can we always lift?} 
    We present an example in two variables that illustrates that, even when a set is defined by a convex quadratic constraint, 
    it might not always be possible to lift a linear seed inequality, valid  for the restriction obtained by fixing a variable at lower bound, when we assume $g(\cdot) - g_0$ is an affine function of the fixed variable.
    %we fix one variable to its lower bound to obtain a restriction and the seed inequality in a linear inequality -- this seed inequality valid for the restriction, cannot be lifted when we assume $g(\cdot) - g_0$ is an affine function of the fixed variable. 
    Our main result, by contrast, establishes that there exists a large class of sets, described by a single \textit{bipartite bilinear} constraint~\cite{dey2019new} together with bounds, for which it is always possible to lift when variables are fixed at their bounds.
    %a seed inequality that is valid for a restriction obtained by fixing variables to their bounds, where the lifting is accomplished using affine functions of the fixed variables. 
    %We note that any quadratic constraint can be relaxed to produce a bipartite bilinear constraint.
    Note that any quadratic constraint can be relaxed to a bipartite bilinear constraint.
    \item \textit{Sequence-independent lifting.} 
    %Just as in the case of sequential lifting for MILPs, in the case of QCQPs as well, 
    The lifting of a fixed variable requires the solution of a non-convex nonlinear optimization problem. 
    When multiple variables must be lifted one at a time, this process (referred to as \textit{sequential lifting}) can be computationally
    % \footnote{\red{JP: REMOVED: very}}
    prohibitive.
    Further, the form of the lifted inequality obtained will differ depending on the order in which variables are lifted.  
    For MILPs, it was shown in \cite{wolsey1977valid} that when the so-called \textit{lifting function}  is subadditive, lifting is far more computationally tractable in part because the form of the lifted inequality is independent of the order in which variables are lifted. 
    %it is possible to accomplish lifting in a far more computationally tractable fashion, using the theory of \textit{subadditive function} approximation of . 
    We develop a similar general result for sequence-independent lifting of seed inequalities for \textit{separable bipartite bilinear} 
    %\blue{\sout{quadratic}} 
    constraints.
%    \footnote{\red{JP: sometimes we say bilinear quadratic, and sometimes we just say bilinear... Should we just stick to one?}}
    \item \textit{Bilinear covering set and bilinear cover inequality.} 
    We next study a \textit{separable bipartite bilinear} set whose coefficients form a minimal cover with respect to the right-hand-side.
    %(minimal cover defined in an analogous fashion as in integer programming). 
    For this set, we derive a \textit{bilinear cover inequality}. 
    This \textit{second-order cone representable valid inequality} yields a constant-factor approximation of the convex hull of the original set. 
    %We prove various results that show that the strength of this bilinear covering inequality, such as showing the inequality . 
    %(in terms of a constant blow up-factor \red{GU: blow-up factor ?}).
    \item \textit{Sequence-independent lifting of  bilinear cover inequality.} 
    %We study the lifting function corresponding to the bilinear cover inequality and construct a \textit{two-slope} subadditive upper bound. 
    We construct a \textit{two-slope} subadditive upper bound of the lifting function corresponding to the bilinear cover inequality. 
    This function is reminiscent of the two-slope subadditive functions studied in the context of cutting-planes for the infinite group relaxation~\cite{gomory1972some,richard2010group,koppe2015electronic}, although there is no apparent connection.
    %between these sets and the infinite group model. 
    Using this subadditive function, we lift fixed variable pairs in closed-form, thus describing a family of \textit{lifted bilinear cover inequalities}, which are valid for general separable bipartite bilinear 
    %\blue{\sout{quadratic}} 
    constraints. 
\end{itemize}

\paragraph{Notation and organization of the paper}
Given a positive integer $n$, we denote the set $\{1, \dots, n\}$ by $[n]$. 
Given a set $S \subseteq \Real^n$ and $\theta > 0$, we use $\theta\cdot S$ to denote the set $\{\theta x\, |\, x \in S\}$. 
We also use $\textup{conv}(S)$ to denote the convex hull of set $S$. The rest of the paper is organized as follows.
In Section~\ref{sec:main} we present our main results. 
In Section~\ref{sec:future} we discuss some key directions for future research. 
Sections~\ref{section:existence}-\ref{section:lifted} give the proofs of the results described in Section~\ref{sec:main}.
%\footnote{\red{JP: SHOULD WE HAVE MORE SECTIONS?}}
%Due to lack of space we do not present proofs.
% All the remaining sections are dedicated to presenting proofs of the results presented in Section~\ref{sec:main}.

%\input{Sec1_Main.tex}
\section{Main results}\label{sec:main}

%To motivate the assumptions we pose in stating Theorem~\ref{thm:existence}, we first describe examples that show that lifting might not be possible when they are not imposed.  
%In order to introduce our first result, we first present a few examples.

Before we discuss our results, we first present two examples that illustrate how lifting can be performed for a set defined by a quadratic constraint and what challenges can arise during such procedure. 

\begin{example}\label{example:one}
Consider the set
%\begin{eqnarray*}
$S:= \{\ (x_1, x_2, x_3) \in [0,1]^3 \ | \ x_1x_2 + 2x_1x_3 \geq  1\  \}.$
%\end{eqnarray*}
First, we fix $x_3 = 0$ to obtain the restriction  $S|_{x_3 = 0}: =\{ (x_1, x_2) \in [0,1]^2 \, |\, x_1x_2 \geq 1\}$. The seed inequality
$\sqrt{x_1x_2} \geq 1,$
is a valid convex inequality for $S|_{x_3 = 0}$.
We next show how it can be lifted into a valid inequality for $S$. Observe that, although valid for $S|_{x_3 = 0}$, the seed inequality is not valid for $S$, since $(x_1, x_2, x_3) = (1,0,\sfrac{1}{2})$ violates it while belonging to $S$. 
We therefore must introduce variable $x_3$ into the seed inequality so as to make it valid. 
In particular we seek 
%to determine whether there exists 
$\alpha \in \Real$ for which
\begin{eqnarray}
\sqrt{x_1x_2} + \alpha x_3 \geq 1, \label{ex:liftedsimple}
\end{eqnarray}
is valid for $S$. 
This question can be answered by solving the problem
    % \begin{align}\label{example:lifting}
    % \begin{array}{rcl}
    %     \alpha^*:= &\textup{sup}& \frac{ 1 - \sqrt{x_1x_2}}{x_3} \\
    %     &\textup{s.t.}& x_1x_2 + 2x_1x_3 \geq  1 \\
    %     && x_3 \in (0, 1], x_1, x_2 \in [0, 1],
    %     \end{array}
    % \end{align}
\begin{equation}\label{example:lifting}
\begin{split}
\alpha^*:= \textup{sup}\ & \frac{ 1 - \sqrt{x_1x_2}}{x_3} \\
\textup{s.t.}\ & x_1x_2 + 2x_1x_3 \geq  1 ,\
x_3 \in (0, 1],\ (x_1, x_2) \in [0, 1]^2,            
\end{split}
\end{equation}
where a key challenge is to first ascertain that the supremum is finite.
When $\alpha^*$ is finite, it is clear that choosing any $\alpha \ge \alpha^*$ in (\ref{ex:liftedsimple}) yields a valid inequality for $S$.
Problem (\ref{example:lifting}) can be analyzed using the following facts: 
(1) for any fixed value of $x_3$, we can always assume that an extreme point is the optimal solution, as the objective is to maximize a convex function, and
(2) the extreme points of the set where $x_3$ is fixed to a value within its bounds are well-understood~\cite{santana2020convex}. 
This suggests that one can inspect all different values of $x_3$ to establish that the supremum is finite. 
We illustrate these calculations next.
    
Observe that $\alpha^*$ can be obtained by computing the supremum $\alpha_1^*$ of (\ref{example:lifting}) for $x_3 \in [\sfrac{1}{2},1]$ and then computing the supremum $\alpha_2^*$ of (\ref{example:lifting}) for $x_3 \in (0,\sfrac{1}{2}]$. 
When $x_3 \in [\sfrac{1}{2},1]$, one optimal solution is $x_1 = \frac{1}{2x_3}$ and $x_2 = 0$, thus $\alpha^*_1 = \textup{sup}_{x_3 \in [\sfrac{1}{2}, 1] }\frac{1}{x_3} = 2.$
When $x_3 \in (0,\sfrac{1}{2}]$, one optimal solution is $x_1 = 1$ and $x_2 = 1 - 2x_3$, thus 
    % \red{GU: change formation}
    % \begin{align*}
    % \begin{array}{rcl}
    % \alpha^*_2 &=& \textup{sup}_{x_3 \in (0, \sfrac{1}{2}] }\frac{1 - \sqrt{1 - 2x_3}}{x_3}  
    %     %&=& \textup{sup}_{x_3 \in (0, 1/2] }\frac{(1 - \sqrt{1 - 2x_3})(1 +\sqrt{1 - 2x_3})}{(x_3)(1 + \sqrt{1 - 2x_3})} \\
    %     = \textup{sup}_{x_3 \in (0, \sfrac{1}{2}] }\frac{2}{1 + \sqrt{1 - 2x_3}} = 2.
    % \end{array}
    % \end{align*}
$$
\alpha^*_2 = \sup_{x_3 \in (0, \sfrac{1}{2}] }\frac{1 - \sqrt{1 - 2x_3}}{x_3}  
%&=& \textup{sup}_{x_3 \in (0, 1/2] }\frac{(1 - \sqrt{1 - 2x_3})(1 +\sqrt{1 - 2x_3})}{(x_3)(1 + \sqrt{1 - 2x_3})} \\
= \sup_{x_3 \in (0, \sfrac{1}{2}] }\frac{2}{1 + \sqrt{1 - 2x_3}} = 2.
$$
Choosing any $\alpha \ge \alpha^*=\max\{\alpha_1^*,\alpha_2^*\}=2$ yields a valid inequality for $S$. 
The strongest such valid inequality is 
%Thus $\alpha =2$ gives a valid inequality for $S$, \textit{i.e.}, 
$\sqrt{x_1x_2} + 2 x_3 \geq 1.$
%is a valid inequality for $S$.
% \hfill \qed
%    \qed
\end{example}
Example~\ref{example:one} might suggest that lifting can always be performed when seeking to derive a linear valid inequality. 
Example~\ref{example:nolifting} shows that it is not so. 
%even when fixing variables at bounds.  
\begin{example}\label{example:nolifting}
Consider the set $
S=\left\{ \ (x_1,x_2)\in[0,1]^2 \ \bigm| \ -x_1^2 - (x_2-0.5)^2\geq -0.5^2 \ \right\}.$

The inequality $-x_1 \geq 0$ is valid for the set $S|_{x_2=0}$ obtained from $S$ by fixing $x_2=0$. 
By
%, for example,  
setting up an optimization problem similar to (\ref{example:lifting}), 
it is 
%then
easy to verify that there is no $\alpha\in \Real$ for which $-x_1+\alpha x_2\geq 0$ is valid for $S$. 
%\hfill \qed
\end{example}
In Example~\ref{example:nolifting}, ($i$) set $S$ is convex, ($ii$) we are trying to lift a linear inequality, and ($iii$) $x_2$ is fixed to a bound. 
Even then, it is not possible to lift the seed inequality when we insist that lifting should be accomplished using an affine function of the fixed variable; see Example~\ref{ex:twocont} in Section~\ref{sec:future} for further discussion.
% \footnote{\red{JP: REMOVE: on this}}

\subsection{Sufficient conditions under which seed inequalities can be lifted}

In Theorem~\ref{thm:existence}, we identify a large class of single row QCQPs where lifting can be accomplished using affine functions of the fixed variables. 

%\footnote{\red{JP: Should we introduce a definition for biparite bilinear set here, so as to mirror what we do in Definition~\ref{def:separable}}?}

\begin{definition} \label{defn:bbs}
A set $Q$ is a \textit{bipartite bilinear set}\footnote{We use the term bipartite, perhaps redundantly, to highlight that variables can be divided into two groups, such that any degree two term comes from product of variables one each from these two groups~\cite{dey2019new}.}  if it is of the form
% \label{eq:bilincoverset}
$$
S = \left\{ \ (x,y) \in [0,1]^m \times [0,1]^n \ \Bigm| \ x\tr Q y + a \tr x + b\tr y \geq c \ \right\},
$$
where $Q\in\Real^{m\times n}$, $a\in \Real^m$, $b\in \Real^n$, and $c\in\Real$.
\end{definition}

\begin{restatable}{theorem}{ThmExistence} \label{thm:existence}
%Consider a set described by one bipartite

%\footnote{\red{JP: ISN"T THAT THE DEFINITION OF BILINEAR THOUGH?}}
%bilinear constraint and bounds on variables
%$$
%S = \left\{ \ (x,y) \in [0,1]^m \times [0,1]^n \ \Bigm| \ x\tr Q y %+ a\tr x + b\tr y \geq c \ \right\},
%$$
%where $Q\in\Real^{m\times n}$, $a\in \Real^m$, $b\in \Real^n$, and $c\in\Real$.
Let $S$ be a bipartite bilinear set. Given $C\times D\subset [m]\times[n]$ and
% $(\tilde{x}_i,\tilde{y}_j)\in\{0,1\}^2$
{$\tilde{x}_i,\tilde{y}_j\in\{0,1\}$} for $i\in [m]\backslash C$, $j\in [n]\backslash D$, 
assume that inequality 
$
h(x_C, y_D) \geq r
$
is valid for $\{ (x,y)\in S \ |\ x_{[m]\backslash C}=\tilde{x}_{[m]\backslash C}, \ y_{[n]\backslash D}=\tilde{y}_{[n]\backslash D}  \}\neq\emptyset$, where $h$ is a concave function defined on {$[0,1]^{|C|+|D|}$}. 
Then, for any $k\in [m]\backslash C$, there exists a finite $f_k\in(-\infty,\infty)$ for which
%\begin{align*}
$h(x_C, y_D) + f_kx_k \geq r + f_k\tilde{x}_k$
%\end{align*}
is valid for {$\{ (x,y)\in S \ |\ x_{([m]\backslash C)\backslash\{k\}}=\tilde{x}_{([m]\backslash C)\backslash\{k\}}, \ y_{[n]\backslash D}=\tilde{y}_{[n]\backslash D} \}$}.
\label{THM_EXISTENCE}
\end{restatable}

\begin{remark}
The result of Theorem~\ref{THM_EXISTENCE} can be applied iteratively to all the fixed variables one at a time to obtain a valid inequality for $S$.
Theorem~\ref{THM_EXISTENCE} holds even when the bounds on variables are not $[0,1]$, since we can always rescale and translate variables.
%so that the bounds are $[0,1]$. 
%Also, we note that any quadratic constraint can be relaxed to produce a bipartite bilinear constraint. 
%\footnote{\red{JP: Probably would be better to move that last comment where it is more visible.}}
\end{remark}

The proof of Theorem~\ref{THM_EXISTENCE} is presented in Section~\ref{section:existence} and uses calculations similar to those presented in Example~\ref{example:one}.
In particular, using a characterization of extreme points of the bipartite bilinear set $S$~\cite{dey2019new}, the proof reduces to establishing the result for three-variable problems where one of the variables is fixed. 
For a three-variable problem, a number of cases have to be analyzed to verify that the optimal  value of an optimization problem similar to (\ref{example:lifting}) is finite. 
The proof can be turned into an algorithm to compute the lifting coefficients, although not necessarily an efficient or practical one.
%\footnote{\red{JP: SHOULD WE ALSO SAY THAT THE PROOF CAN ALSO BE TURNED INTO AN ALGORITHM TO COMPUTE THESE COEFFICIENTS, ALTHOUGH MAYBE NOT AN EFFICIENT ONE?}}

Theorem~\ref{THM_EXISTENCE} assumes that, when variables $x$ and $y$ are fixed, they are fixed at their bounds (either $0$ or $1$.) 
When this assumption is not imposed, we show next through an example that lifting may not be possible.
\begin{restatable}{example}{EgNonexists}
Consider the bipartite bilinear set $
S= \{ (x,y,\hat{x})\in[0,1]^3 | \left(x-\sfrac{1}{4}\right)\left(y-\sfrac{1}{2}\right)$ $\geq \sfrac{\hat{x}}{4}+\sfrac{1}{8} \}.$ First, we argue that the seed inequality $x \ge \sfrac{3}{4}$ is valid for the restriction of $S$ where $\hat{x}=\sfrac{1}{2}$. 
This is clear as $|y-\sfrac{1}{2}|\le \sfrac{1}{2}$ when $y \in [0,1]$ and $|x-\sfrac{1}{4}|<\sfrac{1}{2}$ when $x < \sfrac{3}{4}$.  
Next, we claim that there is no $\alpha\in \Real$ such that $x+\alpha(\hat{x}-\sfrac{1}{2}) \geq \sfrac{3}{4}$ is valid for $S$.
Assume by contradiction that $x+\alpha (\hat{x}-\sfrac{1}{2}) \ge \sfrac{3}{4}$ is valid for $S$ for some $\alpha \in \Real$.
Since $(x,y,\hat{x})=(0,0,0) \in S$, we must have $-\sfrac{\alpha}{2} \ge \sfrac{3}{4}$.
Since $(x,y,\hat{x})=(1,1,1) \in S$, we must have $1+\sfrac{\alpha}{2} \ge \sfrac{3}{4}$.
This is the desired contradiction as the former expression requires that $\alpha \le -\sfrac{3}{2}$ while the later requires that $\alpha \ge -\sfrac{1}{2}$.
\end{restatable}

%We also note that the lifting result of Theorem~\ref{THM_EXISTENCE} does not necessarily extend to larger families of quadratic sets, even when variables are fixed at their bounds.
%\begin{restatable}{example}{EgSquare}
%Inequality $x\leq 0$ is valid for the quadratic problem obtained by fixing $y=0$ in 
%\begin{eqnarray*}    
%S=\left\{x,y\in[0,1]^2 \Bigm| x^2 + (y-0.5)^2\leq 0.5^2 \right\}.
%\end{eqnarray*}    
%Further, there is no $\alpha\in \Real$ such that $x+\alpha y\geq 0$ is valid for $S$.
%\end{restatable}
\subsection{A framework for sequence-independent lifting}

Given a set of variables fixed at their bounds and a seed inequality for the corresponding restriction, a valid inequality for the original problem can be obtained by lifting each fixed variable one at the time.
This computationally demanding process requires the solution of a non-convex nonlinear optimization problem, similar to (\ref{example:lifting}), to lift each variable. 
It results in a lifted inequality whose form depends on the order in which variables are lifted. Next, we study situations where the lifting inequality obtained does not depend on the order in which variables are lifted. 
In particular,  we develop a subadditive theory for lifting in QCQPs that is inspired by that originally developed in MILP in \cite{wolsey1977valid}.
%Just as in the case of sequential lifting for MILPs, in the case of QCQPs as well, sequential  Similar to MILPs, we explore a subadditive theory for lifting in a sequence independent fashion.% \red{GU: for lifting in a sequence-independent manner. ?}
We consider the special case of separable bipartite bilinear constraints.
%\footnote{\red{JP:Technically, I think we could select $a_i$ to be vectors in $\Real^p$ and everything would still go through...}}
\begin{restatable}{definition}{DefSeparable}
\label{def:separable}
%A set $Q$ is said to be a \textit{separable bipartite bilinear set} if there exists parameters $d$ and $a_i \in \Re$ for $i \in [n]$ such that
%\begin{eqnarray}
% \label{eq:bilincoverset}
%Q:= \left\{(x , y) \in [0, 1]^n \times [0, 1]^n \,\Bigm|\,  \sum_{i = 1}^n a_i x_iy_i \geq d \right\}.
%\end{eqnarray}
A set $Q$ is a \textit{separable bipartite bilinear set} if it is of the form
\begin{eqnarray*}
% \label{eq:bilincoverset}
Q:= \left\{ \ (x , y) \in [0, 1]^n \times [0, 1]^n \ \Bigm|\  \sum_{i = 1}^n a_i x_iy_i \geq d \ \right\},
\end{eqnarray*}
for some $d$ and $a_i \in \Real$ for $i \in [n]$, \textit{i.e.}, variables $x_i$ and $y_i$, for $i \in [n]$, appear in only one term. \end{restatable}

In the separable case, it is natural to lift each pair of variables $x_i$ and $y_i$ together.
Next, we derive conditions 
%on the lifting function of a seed  inequality 
that guarantee that the form of the lifted inequality obtained is independent of the order in which these pairs are lifted. 
This result is obtained, as is common in MILP, by deriving a subadditive upper bound on the lifting function of the seed inequality, from which all lifting coefficients can be derived.
%\footnote{\red{JP: I think two ideas are conflated here. The first is that subadditive gives sequence independence, the second is that approximations of the lifting function still give valid inequalities... Do we want to be more specific?}} 
 
 %thereby resulting in a sequence-independent   
%while the pairs are lifted in a sequence-independent fashion. 
%The main idea is to study the lifting function just as in the case of MILPs and to obtain a subadditive upper bound to the lifting function. 
%The next result presents the key ideas on how to sequence-independently lift a seed inequality.  

%{\color{red}
%\footnote{JP: I have split the original proposition into a %definition and a proposition as follows: }
\begin{restatable}{definition}{liftingfunction}\label{def:liftingfunction}
Let $Q$ be a separable bipartite bilinear set. 
Assume that $\Lambda = \{I, J_0, J_1\}$ is a partition of $[n]$ (\textit{i.e.}, $I\cup J_0\cup J_1 = [n]$ with $I\cap J_0=I\cap J_1=J_0\cap J_1=\emptyset$) and that
%\begin{eqnarray}
$h(x_I, y_I) \geq r$, %\label{ineq:seed}
%\end{eqnarray}
is a valid inequality for $\{(x,y)\in Q \ |\  x_{J_0}=y_{J_0}=0,\ x_{J_1}=y_{J_1}=1\}$. 
For $\delta \in \Real$, we define the lifting function of the seed inequality
%(\ref{ineq:seed}) 
as 
$$ \textstyle
\phi(\delta):= \max\left\{ \ r-h(x_I, y_I) \ \Bigm| \ \sum_{i\in I} a_ix_iy_i\geq \left(d-\sum_{i\in J_1}a_i\right) -\delta, \ (x_I,y_I) \in [0,1]^{2|I|}
%(x,y)\in[0,1]^{2n} 
\ \right\}.
$$
%\footnote{\red{JP: Should we have $(x,y) \in [0,1]^{2|I|}$ instead %of $(x,y) \in [0,1]^{2n}$?}}
%\footnote{\orange{XY: I believe it should be either $(x_I,y_I) \in %[0,1]^{2|I|}$ or $(x,y)\in[0,1]^{2n}$.}}
\end{restatable}

Structured approximations of lifting functions allow for simple lifting of inequalities as described next in Proposition~\ref{lm:seqind}, whose proof can be found in Section~\ref{section:seqind}. 

\begin{restatable}{proposition}{LmSeqInd}
\label{lm:seqind}
Let $Q$ be a separable bipartite bilinear set and let $\Lambda = \{I, J_0, J_1\}$ be a partition of $[n]$. 
Let $\phi$ be the lifting function of seed inequality $h(x_I, y_I) \geq r$
%(\ref{ineq:seed}) 
for $\{(x,y)\in Q \ |\  x_{J_0}=y_{J_0}=0,\ x_{J_1}=y_{J_1}=1\}$ where $h$ is a concave function. 
Assume there exists $\psi: \Real \mapsto \Real$ and concave functions $\gamma_i:\Real^2 \mapsto \Real$ for $i \in J_0 \cup J_1$ such that 
\begin{enumerate}[label={(\roman*)},align=left]
\item \label{item:seqind:1}
$\psi(\delta)\geq \phi(\delta)$, $\forall \delta \in \Real$;
\item \label{item:seqind:2}
$\psi$ subadditive, (\textit{i.e.}, $\psi(\delta_1)+\psi(\delta_2)\geq \psi(\delta_1+\delta_2)$, $\forall \delta_1,\delta_2\in\Real$) with $\psi(0) = 0$;
\item \label{item:seqind:3}
for $i\in J_0$, 
$\gamma_i(x,y)\geq \psi(a_ixy), \forall (x,y)\in [0,1]^2,$
\item \label{item:seqind:4}
for $i\in J_1$, 
$\gamma_i(x,y)\geq \psi(a_ixy-a_i), \forall (x,y)\in [0,1]^2.$
\end{enumerate}
Then, the lifted inequality 
$h(x_I, y_I)+\sum_{i\in J_0\cup J_1} \gamma_i(x_i,y_i)\geq r$
is a valid convex inequality for $Q$.
\end{restatable}

%{\color{blue}
%\begin{restatable}{proposition}{LmSeqInd}
%\label{lm:seqind}
%5Let $Q$ be a separable bipartite bilinear set. 
%Assume that $\Lambda = \{I, J_0, J_1\}$ is a partition of $[n]$, \textit{i.e.}, $I\cup J_0\cup J_1 = [n]$ with $I\cap J_0=I\cap J_1=J_0\cap J_1=\emptyset$, and that
%$$
%h(x_I, y_I) \geq r
%$$
%is a valid inequality for $\{(x,y)\in Q \,|\, x_{J_0}=y_{J_0}=0,x_{J_1}=y_{J_1}=1\}$ where $h$ is a concave function. 
%For $\delta \in \Real$, compute the lifting function
%$$
%\phi(\delta):= \max\left\{r-h(x_I, y_I) \Bigm| \sum_{i\in I} a_ix_iy_i\geq \left(d-\sum_{i\in J_1}a_i\right) -\delta, x,y\in[0,1]^n \right\}.
%$$
%Let $\psi: \Real \mapsto \Real$ be such that 
%\begin{enumerate}
%\item 
%(i) $\psi(\delta)\geq \phi(\delta)$, $\forall \delta \in \Real$;
%\item 
%(ii) $\psi$ subadditive, (\textit{i.e.}, $\psi(\delta_1)+\psi(\delta_2)\geq \psi(\delta_1+\delta_2)$, $\forall \delta_1,\delta_2\in\Real$) with $\psi(0) = 0$;
%\item 
%(iii) for any $i\in J_0$, there exists a concave function $\gamma_i(x,y)$ such that
%    $
%    \gamma_i(x,y)\geq \psi(a_ixy), \forall x,y\in [0,1],
%    $
%    and for any $i\in J_1$, there exists a concave function $\gamma_i(x,y)$ such that
%    $
%    \gamma_i(x,y)\geq \psi(a_ixy-a_i), \forall x,y\in [0,1].
%    $
%\end{enumerate}
%Then, the lifted inequality 
%$$
%h(x_I, y_I)+\sum_{i\in J_0\cup J_1} \gamma_i(x_i,y_i)\geq r
%$$
%is a valid convex inequality for $Q$.
%\end{restatable}}
The statement of Proposition~\ref{lm:seqind} does not specify the type of functional forms $\gamma_i(x_i,y_i)$ to use in ensuring that conditions \ref{item:seqind:3} and \ref{item:seqind:4} are satisfied.  
It is however clear from the definition that choosing $\gamma_i(x_i,y_i)$ to be the concave envelope of $\psi(a_ix_iy_i)$ over $[0,1]^2$ when $i \in J_0$, and the concave envelope of $\psi(a_ix_iy_i - a_i)$ over $[0,1]^2$ when $i \in J_1$ is the preferred choice for $\gamma_i$.

\begin{remark}
While we state the result of Proposition~\ref{lm:seqind} for a set $Q$ defined by a single separable bipartite bilinear constraint, a similar result would also hold for sets defined by  multiple separable bipartite bilinear constraints.
\end{remark}

\subsection{A seed inequality from a minimal covering set}

To generate lifted inequalities for separable bipartite bilinear sets, we focus next on a family of restrictions we refer to as \textit{minimal covering sets}. 
For such minimal covering sets, we introduce a provably strong convex, second-order cone representable valid inequality. 
We use this inequality as the seed in our lifting procedures. 

%We will be interested in beginning from a so-called `minimal cover' for deriving a seed inequality.
\begin{restatable}{definition}{DefMinimal}
\label{def:minimal}
%For a separable bipartite bilinear set $Q$, we define a \textit{minimal cover} to be a partition $\Lambda = \{I, J_0, J_1\}$ of $[n]$, where $I\neq \emptyset$, 
%\footnote{\red{JP: I PROPOSE REMOVE THE FOLLOWING SINCE WE DESCRIBED WHAT A PARTITION IS IN AN EARLIER RESULT:} $I\cup J_0\cup J_1 = [n]$ with $I\cap J_0 = I\cap J_1 = J_0\cap J_1 = \emptyset$,
%}
%such that
%\begin{enumerate}
%    \item $\sum_{i\in I} a_i > d - \sum_{i\in J_1} a_i:= d^\Lambda>0$;
%    \item $\sum_{i\in K\subsetneq I} a_i \leq d^\Lambda$.
%\end{enumerate}
%Let $J_0^+ := \{i\in J_0|a_i>0\}$, $J_0^- := \{i\in J_0|a_i < 0\}$ and similarly define $J_1^+$, $J_1^-$.
Let $k \in \Ze_{+}$ be a positive integer. 
We say that $a_i\in \Real$ for $i \in [k]$ form a \textit{minimal cover} of $d\in \Real$,
%\footnote{The textbook minimal cover is that $\sum_{i\in I} a_i\geq d$ and $\sum_{i\in J\subsetneq I} a_i <d$, a little different from our definition.} 
if (i) $a_i > 0$ for all $i \in [k]$, $d >0$, (ii) $\sum_{i = 1}^k a_i > d$, (iii) $\sum_{i \in K} a_i \leq d$, $\forall K \subsetneq [k]$.
% \begin{enumerate}[label={(\roman*)},align=left]
% \item $a_i > 0$ for all $i \in [k]$, $d >0$, \label{cond:1}
% \item $\sum_{i = 1}^k a_i > d$ \label{cond:2}
% \item $\sum_{i \in K} a_i \leq d$, $\forall K \subsetneq [k]$. \label{cond:3}
% \end{enumerate}
For a separable bipartite bilinear set $Q$, we say that a partition $\Lambda = \{I, J_0, J_1\}$ of $[n]$, where $I\neq \emptyset$, is a \textit{minimal cover yielding partition} if: $a_i$ for $i \in I$ form a minimal cover of $d^\Lambda:= d - \sum_{i\in J_1} a_i$. 
For a minimal cover yielding partition, we let $J_0^+ := \{ i\in J_0\ |\ a_i>0 \}$, $J_0^- := \{ i\in J_0\ |\ a_i < 0 \}$; we define $J_1^+$ and $J_1^-$ similarly.
\end{restatable}

\begin{remark}
When $k \geq 2$, conditions (ii)
%\ref{cond:2} 
and (iii)
%\ref{cond:3} 
in the definition of minimal cover imply condition (i) 
%\ref{cond:1}. 
For example, 
%that $a_i > 0$, $i \in [k]$ (\textit{i.e.}, they imply part of condition (1.)) since  
if $a_i \leq 0$ for some $i\in [k]$, then (ii)
%\ref{cond:2} 
implies $\sum_{j \in [k] \setminus \{i\}} a_j > d$, contradicting 
(iii).
%\ref{cond:3}. 
Now (iii)
%\ref{cond:3} 
together with $a_i >0$ for $i \in [k]$ implies $d >0$.
%If $\Lambda$ represents a minimal cover, then $a_i > 0$ for $i \in I$.
%Otherwise, $\sum_{i \in I \setminus \{i\}} a_i > d^{\Lambda}$. 
%Also note that here $\Lambda$ (or $I$) is indeed a minimal cover in the `integer programming sense' for $d^\Lambda+\epsilon$ with $\epsilon\downarrow 0$.
%, requiring $a_i>0$ for any $i\in I$.
%\footnote{\red{JP: I AM NOT SURE WHAT THIS LAST NOTATION MEANS... NOTE ALSO THAT IN IP, A COVER IS A COLLECTION OF VARIABLES WHOSE COEFFICIENTS ADD TO SOMETHING HIGHER THAN THE RIGHT-HAND-SIDE, IT IS NOT REALLY A PARTITION OF THE SET OF VARIABLES...}}
%\footnote{\color{blue}{GU: Here actually $I$ is the `minimal cover' but RHS changes with different selection of $J_0$ and $J_1$. The textbook minimal cover is that $\sum_{i\in I} a_i\geq d$ and $\sum_{i\in J\subsetneq I} a_i <d$, a little different from our definition.}}
\end{remark}

%\blue{
%\footnote{JP: Should we add this, and then not reintroduce all notations later? Also, do we really need $d_i$? seems like $a-\Delta$ should do...}
\begin{notation}\label{not:1}
Assuming that $a_i$ for $i \in [n]$ form a minimal cover of $d$, we use 
(i) $\Delta:= \sum_{i=1}^n a_i - d$, 
(ii) $d_i:= d - \sum_{j\in [n]\backslash \{i\}} a_j$, 
(iii) $I^> := \{i \in [n] \ |\  a_i > \Delta\}$, 
(iv) when $I^> \neq \emptyset$, $\i0$ to be any index in $I^>$ such that $a_{\i0} = \min \{ a_i \ |\  i \in I^> \}$.
% \begin{enumerate}[label={(\roman*)},align=left]
% \item $\Delta:= \sum_{i=1}^n a_i - d$.
% \item $d_i:= d - \sum_{j\in [n]\backslash \{i\}} a_j$.
% \item $I^> := \{i \in [n] \ |\  a_i > \Delta\}$.
% \item when $I^> \neq \emptyset$, \red{$\i0$ to be any index in $I^>$ such that $a_{\i0} = \min \{ a_i \ |\  i \in I^> \}$ .}
% \end{enumerate}
\end{notation}
For a minimal cover, conditions (ii)
%\ref{cond:2} 
and 
(iii)
%\ref{cond:3}
in Definition~\ref{def:minimal} imply that $\Delta>0$ and $a_i \ge \Delta$ for all $i \in [n]$, respectively.
Simple computations show that $d_i = a_i - \Delta$.
%}

Our overall plan is the following. 
We will fix $x_i = y_i=0$ for $i\in J_0$ and $x_iy_i=1$ for $i\in J_1$.
Then, we will find a valid seed inequality for the set where the coefficients form a minimal cover. 
Finally, we will lift this seed inequality. 
One key reason to generate cuts from a seed inequality corresponding to a minimal cover is the following result.
%\footnote{\red{JP: Should we say more $x_i=y_i=0$ and $x_i=y_i=1$ since in the case where $x_iy_i=0$ we have cases $(0,1)$ and $(1,0)$, which we do not treat...}}

\begin{restatable}{theorem}{ThmMinimal}
\label{thm:minimal}
For a nonempty separable bilinear set $Q$, either there exists at least one minimal cover yielding partition or $\conv(Q)$ is polyhedral.
%one partition of minimal cover $\Lambda = \{I, J_0, J_1\}$ or we have that $Q = \emptyset$ or $\conv(Q)$ is polyhedral.
\end{restatable}

Loosely speaking, the proof of Theorem~\ref{thm:minimal}, which is given in Section~\ref{section:minimal}, is based on showing that if there is no minimal cover yielding partition, then $Q$ is ``almost" a packing-type set, \textit{i.e.}, a set of the form $\{ (x,y) \in [0, 1]^{2n} \ |\ \sum_{i = 1}^n a_ix_iy_i \leq d\}$ where $a_i$s are non-negative.
%\footnote{\red{JP: SHOULD IT BE $[0,1]^2$ IN THE ABOVE FOOTNOTE?}}
%\footnote{\blue{XG: $(x,y)\in[0,1]^{2n}$ or $x,y\in[0,1]^n$, maybe the former expression?}}
For packing sets $Q$, \cite{richard2010liftingframework} shows that $\textup{conv}(Q)=\textup{proj}_{x,y}(G)$ where
\begin{eqnarray*}
G  = \left\{(x,y,w) \in [0, 1]^{3n} \ \Bigm|\  \sum_{i = 1}^n a_iw_i \leq d, \ x_i + y_i -1 \leq w_i,\  \forall i \in [n]\right\}.
\end{eqnarray*}
%\footnote{\red{JP: I removed the $\max\{,0\}$ since $w$ is assumed to belong to $[0,1]$}}
We say ``almost", since there are non-packing sets such as $S:= \{ (x,y) \in [0, 1]^4\ |\  x_1y_1 - 100x_2y_2 \geq -98\}$, where there is no partition that yields a minimal cover. 
Such sets are ``overwhelmingly" like a packing set; in the case of the example, it is a perturbation of the packing set $\{ (x_2, y_2) \in [0, 1]^2\ |\  100x_2y_2\leq 98\}$. 
For such sets it is not difficult to show that $\conv(S)$ is polyhedral.

Since the main focus of this paper is the study of lifted convex (nonlinear) inequalities and since in the packing case the convex hull is trivially obtained using McCormick inequalities~\cite{mccormick1976computability}, the remainder of the paper will concentrate on the case where there exists a minimal cover yielding partition.

Associated with a minimal cover is a convex valid inequality that we present next.

\begin{restatable}{theorem}{ThmValid}
\label{thm:valid}
Consider the separable bipartite bilinear minimal covering set 
as presented in Definition~\ref{def:separable}
%\begin{eqnarray}
%\label{eq:bilincoverset}
%Q:= \left\{\ (x,y) \in [0, 1]^{2n}  \ \Bigm|\  \sum_{i = 1}^n a_i %x_iy_i \geq d \ \right\},
%\end{eqnarray}
where $a_i$, $i \in [n]$ form a minimal cover of $d$.
%($i$) $a_i > 0$ for $i \in [n]$, ($ii$) $d>0$, ($iii$) 
%$\sum_{i = 1}^n a_i > d$, and ($iv$) $\sum_{i \in I} a_i \leq d$ for $I \subsetneq [n]$.
Then, the \textit{bilinear cover inequality} is valid for $Q$:
%{\color{blue}{GU: PROPOSED CHANGE: 
%Then the following (seed) minimal cover inequality
%}}
%\footnote{\red{JP: SHOULD WE GIVE THIS INEQUALITY A NAME? MINIMAL BILINEAR COVER INEQUALITY?}}
%\begin{eqnarray}
%\sum_{i = 1}^n \frac{1}{ 1 - \sqrt{1 - \frac{\Delta}{a_i}}}\left( \sqrt{x_iy_i} - 1\right) \geq -1, \label{COVER_INEQUALITY}
%\end{eqnarray}
%where $\Delta :=  \sum_{i = 1}^n a_i - d$ is valid for $Q$. 
%Using notation $d_i:= a_i - \Delta$, we can write (\ref{COVER_INEQUALITY}) equivalently as
\begin{eqnarray} \label{eq:bilincoverineq}
\sum_{i = 1}^n \frac{\sqrt{a_i}}{ \sqrt{a_i}-\sqrt{d_i}}\left( \sqrt{x_iy_i} - 1\right) \geq -1.
\end{eqnarray}
%where $d_i = d - \sum_{j \in [n]\setminus\{i\}}a_j$, is valid for $Q$.
%\footnote{\red{JP: SHOULD IT BE j INSTEAD OF i IN THE ABOVE SUM?} \blue{fixed}}
% Using the notation $\Delta :=  \sum_{i = 1}^n a_i - d$, we can rewrite the inequality as
% \begin{eqnarray}
% \sum_{i = 1}^n \frac{1}{ 1 - \sqrt{1 - \frac{\Delta}{a_i}}}\left( \sqrt{x_iy_i} - 1\right) \geq -1. \label{COVER_INEQUALITY}
% \end{eqnarray}
\end{restatable}
%\footnote{\red{JP: Should we standardize notation $(x,y) \in [0,1]^{2n}$ vs $(x,y) \in [0,1]^n \times [0,1]^n$?}}
%\footnote{\red{JP: Should we just introduce a definition for this %set?}}

Our proof of Theorem~\ref{thm:valid}, which is presented in Section~\ref{section:valid}, uses techniques from disjunctive programming~\cite{balas1998disjunctive} and an ``approximate version"  of Fourier-Motzkin projection.  In particular, using the minimal covering property of the coefficients %in~(\ref{eq:bilincoverset}) 
and a characterization of the extreme points of bipartite bilinear sets~\cite{dey2019new}, we obtain $n$ second-order cone representable sets whose union contains all the extreme points 
separable bipartite bilinear set. 
%of~(\ref{eq:bilincoverset}). 
Next we write an extended formulation~\cite{balas1998disjunctive,ben2001lectures} of the convex hull of the union of these sets. 
Finally, we use the Fourier-Motzkin procedure to project out the auxiliary variables of the extended formulation one at a time. 
This procedure works to project out most of the variables.
The last step however requires a relaxation to be constructed so that projection can be carried in closed-form. 
Finally we obtain an inequality that is in fact stronger than (\ref{eq:bilincoverineq}).

%\footnote{\red{JP: We actually obtain something stronger than (5). Should we just add a comment?}}

Inequality (\ref{eq:bilincoverineq}) can be viewed as a strengthening of an inequality presented in \cite{tawarmalani2010strong} for the set $Q^{\textup{relax}}$ obtained from $Q$ by relaxing upper bounds on variables, \textit{i.e.},
%$$Q^{\textup{relax}} := \left\{ \ (x,y) \in \Real^{2n}_+ \  \Bigm| \ \sum_{i =1}^n a_ix_iy_i \geq d\ \right\},$$
$Q^{\textup{relax}} := \{  (x,y) \in \Real^{2n}_+ \  | \ \sum_{i =1}^n a_ix_iy_i \geq d \},$
where $a_i>0$ for $i \in [n]$ and $d>0$. 
The convex hull of $Q^{\textup{relax}}$ is shown in \cite{tawarmalani2010strong} to be described by nonnegativity constraints together with
\begin{eqnarray}\label{eq:oldieq}
\sum_{i = 1}^n \frac{\sqrt{a_i}}{\sqrt{d}}\sqrt{x_iy_i}\geq 1.
\end{eqnarray}

%\red{The ensuing proposition, whose proof can be found in Section~\ref{section:comparedtoCRT}, shows that (\ref{eq:bilincoverineq}) improves on (\ref{eq:oldieq}). It essentially proceeds by comparing the coefficients of variable pairs $\sqrt{x_iy_i}$ inside of the inequalities.The proof also shows that, if $n\geq 2$ and there exists $i \in [n]$ such that $d_i > 0$, then (\ref{eq:bilincoverineq}) strictly dominates (\ref{eq:oldieq}).
%}

The ensuing proposition, whose proof we skip due to lack of space shows that (\ref{eq:bilincoverineq}) improves on (\ref{eq:oldieq}). 
It essentially proceeds by comparing the coefficients of variable pairs $\sqrt{x_iy_i}$ inside of the inequalities. 
Moreover, if $n\geq 2$ and there exists $i \in [n]$ such that $d_i > 0$, then (\ref{eq:bilincoverineq}) strictly dominates (\ref{eq:oldieq}).

% is valid (in fact gives the convex hull) of $Q^{\textup{relax}}$.
\begin{proposition}\label{prop:comparedtoCRT}
Inequality (\ref{eq:oldieq}) is dominated by  (\ref{eq:bilincoverineq}) over the $0-1$ box, \textit{i.e.}, 
$$\{\ (x,y) \in [0, 1]^{2n}  \ |\  (\ref{eq:oldieq})\ \} \supseteq \{\ (x,y) \in [0, 1]^{2n}  \ |\  (\ref{eq:bilincoverineq}) \ \}.$$
\end{proposition}
%\footnote{\red{JP: I moved this discussion from the back to here, since it was not mentioned at all here and was getting lost in the paper...}}

Even though Proposition~\ref{prop:comparedtoCRT} hints at the strength of the bilinear cover inequality, it can be easily verified that (\ref{eq:bilincoverineq}) does not produce the convex hull of $Q$. 
However there are a number of reasons to use this inequality as a seed for lifting. 
The first reason is that, not only is inequality (\ref{eq:bilincoverineq}) second-order cone representable, we only need to introduce one extra variable representing $\sqrt{x_iy_i}$ for each $i \in [n]$, to write it as a second-order cone representable set. 
%\footnote{red{JP: Do we want to use "second order" or "second-order"?}}
Apart from the convenience of using this inequality within modern conic solvers, the main reason for considering it as a seed inequality is its strength. 
In particular, we prove next that (\ref{eq:bilincoverineq}) provides a constant factor approximation of the convex hull of the original set. 

%Lifted inequality (\ref{eq:bilincoverineq}) is a convex inequality that can be used to create a convex relaxation of $Q$. 
%We argue in the following theorem that this convex relaxation is indeed strong.
%\begin{restatable}{theorem}{ThmStrength}
%\label{thm:strcovercon}
%Assume that $a_i$s form a minimal cover of $d$.
%For $p, q \in \mathbb{R}^n_{+}$, define 
%\begin{eqnarray*}
%z^* := \textup{min} \left\{ \sum_{i =1}^n (p_i x_i + q_i y_i) \right.\left| 
%\begin{array}{l}
%\sum_{i = 1}^n a_i x_iy_i \geq d \\
%x, y\in [0,1]^n
%\end{array}
%\right\}
%\end{eqnarray*}
%and 
%\begin{eqnarray*}
%z^l := \textup{min} \left\{  \sum_{i =1}^n \left(p_i x_i + q_i y_i\right) \right.\left| 
%\begin{array}{l}
%\sum_{i = 1}^n \frac{\sqrt{a_i}}{\sqrt{a_i} - \sqrt{d_i} } (\sqrt{x_iy_i} - 1 )\geq -1 \\
%x, y\in [0,1]^n
%\end{array}
% \right\}.
%\end{eqnarray*}
%Then $z^l \leq z^* \leq 4z^l$.
%\end{restatable}
\begin{restatable}{theorem}{ThmStrength}
\label{thm:strcovercon}
%Assume that $a_i$s form a minimal cover of $d$, \textit{i.e.}, ($i$) $a_i > 0$ for $i \in [n]$, ($ii$) $d>0$, ($iii$)  $\sum_{i = 1}^n a_i > d$, and ($iv$) $\sum_{i \in I} a_i \leq d$ for $I \subsetneq [n]$. 
%\footnote{\red{JP: SHOULD WE SAY THE A'S FORM A COVER OR SHOULD WE SAY A SET OF INDICES FORM A COVER?}}
Let $Q$ be a bipartite bilinear minimal covering set 
%as described in (\ref{eq:bilincoverset}). 
%\footnote{\red{JP: Do we want to go with $\Real^{2n}_+$?}}
Let $R:= \{ (x,y) \in \Real^{2n}_{+} \ |\  (\ref{eq:bilincoverineq})\}$. 
Then 
$ (4\cdot R) \cap [0, \ 1]^{2n} \subseteq \textup{conv}(Q) \subseteq R \cap [0, \ 1]^{2n}.$
\end{restatable}
%\blue{ADD EXPLANATION. One should note that due to the form of (\ref{eq:bilincoverineq}), $(4\cdot R)\subseteq R$.}
Since $R$ is a set of the covering type (that is, its recession cone is the non-negative orthant), we have that $4\cdot R \subseteq R$. 
%To give intuition as to why Theorem~\ref{thm:strcovercon} holds, we note that in any feasible solution of (\ref{eq:bilincoverset}), $x_iy_i \geq \sfrac{d_i}{a_i}$ for $i \in [n]$. This condition is also enforced by the bilinear cover inequality (\ref{eq:bilincoverineq}).  Therefore, there is no point in the set $R$ that is very close to the origin, so that it must be substantially rescaled to be in convex hull of  (\ref{eq:bilincoverset}). 
The proof of Theorem~\ref{thm:strcovercon}, which is given in Section~\ref{section:strcovercon}, is based on optimizing linear functions with non-negative coefficients on $R$ and $Q$ and proving a bound of $4$ on the ratio of their optimal objective function values.
% \footnote{\red{JP: IS IT CLEAR WHAT IS MEANT BY A NON-NEGATIVE LINEAR FUNCTION?}}

%\footnote{\red{JP: I removed this because of the add that was done before}
%A set similar to (\ref{eq:bilincoverset}) is studied in~\cite{tawarmalani2010strong}, except that there are no upper bounds on the variables. 
%In this case, the convex hull can be described in closed form using a single nonlinear inequality that is similar in structure to (\ref{eq:bilincoverineq}). 
%We can formally verify, however, that inequality (\ref{eq:bilincoverineq}), which also uses the information of upper bounds, dominates the inequality presented in~\cite{tawarmalani2010strong}. 
%Moreover, if $n\geq 2$ and there exists $i \in [n]$ such that $d_i > 0$, then (\ref{eq:bilincoverineq}) strictly dominates the inequality presented in~\cite{tawarmalani2010strong}.}

%\footnote{\red{JP: CAN WE SAY STRICTLY DOMINATES?}}

\subsection{Lifting the bilinear cover inequality}
%\subsection{Lifting seed inequality (\ref{eq:bilincoverineq}).}
%{\color{blue}{GU: PROPOSED CHANGE TITLE: Lifting (seed) minimal cover inequality (\ref{eq:bilincoverineq}).}}

We now follow the framework of Proposition~\ref{lm:seqind} to perform sequence-independent lifting of the bilinear cover inequality. 
%\footnote{\red{JP: IS THAT THE NAME WE WANT TO USE?}}
The first step is to study the lifting function.  

%Although evaluating the lifting function for an arbitrary valid inequality, in general, appears to be a difficult task, the bilinear cover inequality (\ref{eq:bilincoverineq}) 
%\footnote{\red{JP: IS THAT THE NAME WE WANT TO USE?}}
%has sufficient structure that we can derive a strong subadditive upper bound. 
%We will use this approximation next to lift the seed inequality in closed-form sequence-independent fashion.

\begin{restatable}{theorem}{ThmBound}
\label{thm:upperbound}
%Consider the lifting function for valid inequality~(\ref{eq:bilincoverineq}):
%\begin{align*}
%\phi(\delta) := \max\ &\sum_{i = 1}^n \frac{\sqrt{a_i}}{\sqrt{a_i} - \sqrt{d_i}}\left( 1- \sqrt{x_iy_i}\right)-1\\
%\st\ &\sum_{i = 1}^n a_ix_iy_i \geq d - \delta, \\
%& (x, y) \in [0, 1]^{2n}. 
%\end{align*}
%Let $\Delta :=  \sum_{i = 1}^n a_i - d$ and let $a_{i_0} = \min \{a_i\ | \ a_i>\Delta\}$ if it exists. 
Let $\phi$ be the lifting function for valid inequality~(\ref{eq:bilincoverineq}). Define
\begin{equation}
\psi(\delta) := \left\{
\begin{array}{lrrll} 
l_+(\delta+\Delta) - 1\ & &\delta&\leq -\Delta \\
l_-\delta & -\Delta\leq & \delta & \leq 0 \\
l_+\delta & 0 \leq & \delta,& & 
\end{array}\right. 
\label{eq:upperbound}
%&:=& \psi(\delta) 
\end{equation}
where $l_- = \frac{1}{\Delta}$ and where
$l_+=\frac{\sqrt{a_{i_0}}+\sqrt{d_{i_0}}}{\Delta \sqrt{d_{i_0}}}$
if $a_{i_0}$ exists and $l_+=\frac{1}{\Delta}$ otherwise.
Then 
\begin{enumerate}[label={(\roman*)},align=left]
\item $l_+\geq l_-> 0$, 
\item $\psi(\delta)$ is subadditive over $\Real$ with $\psi(0) = 0$, and 
\item $\phi(\delta) \le \psi(\delta)$ for $\delta \in \Real$.
\end{enumerate}
\end{restatable}

Although computing the lifting function for an arbitrary valid inequality, in general, appears to be a difficult task, the bilinear cover inequality (\ref{eq:bilincoverineq}) has sufficient structure that we can derive a strong subadditive upper bound in Theorem~\ref{thm:upperbound}. 
The key to proving Theorem~\ref{thm:upperbound}, as we show in Section~\ref{section:upperbound}, is to first obtain the lifting function exactly in a region around the origin, and to argue that the linear upper bound of the lifting function for this region upper bounds the lifting function globally.
Figure~\ref{fig:sub} presents examples of the lifting function $\phi$, and the upper bound $\psi$ we derived in Theorem~\ref{thm:upperbound} for the cases when $a_{i_0}$ exists and for the case when it  does not. 
%\footnote{\red{JP: The red function in Figure~1 seems to suggest that it takes a constant value when $\delta$ is sufficiently negative when it is in fact equal to $-\infty$. Thinking of the function as constant, the blue function does not look like an upper bound anymore...}}
\begin{figure}
    \centering
         \begin{subfigure}[b]{0.45\textwidth}
         \centering
         \includegraphics[width=\textwidth]{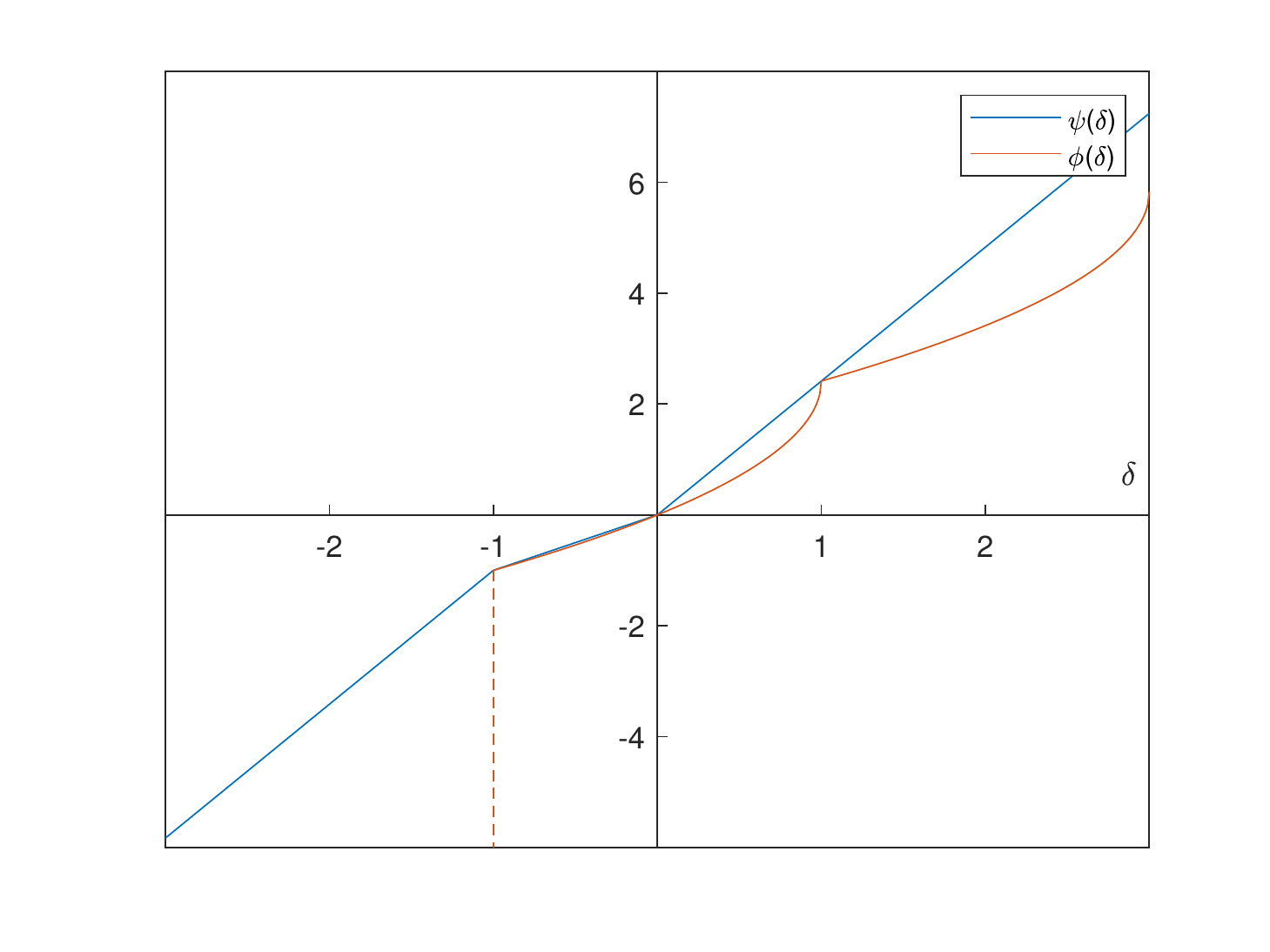}
         \caption{$a_i=2$, $\Delta = 1$}
        %  \label{fig:y equals x}
     \end{subfigure}
        \begin{subfigure}[b]{0.45\textwidth}
        \centering
        \includegraphics[width=\textwidth]{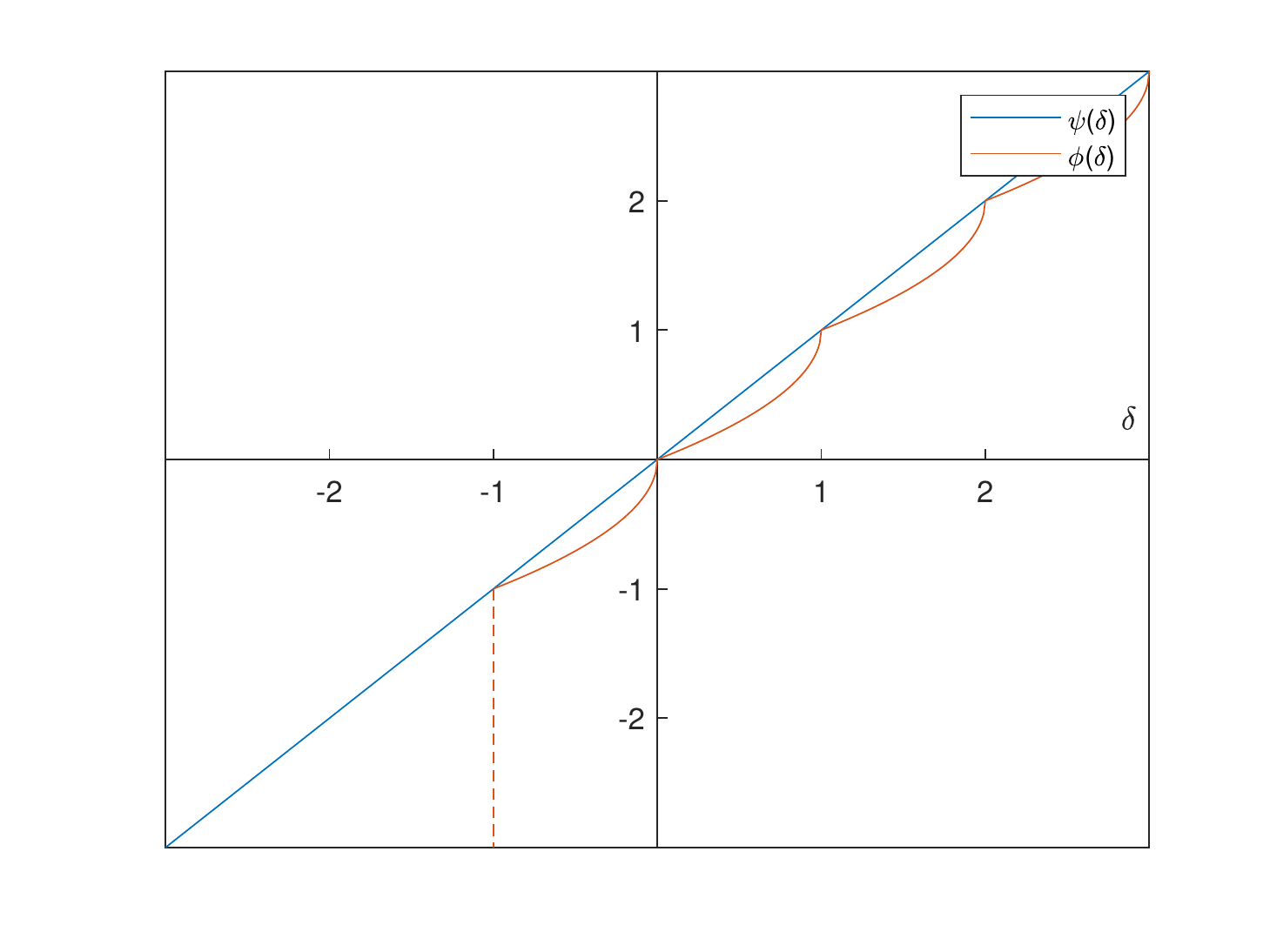}
        \caption{$a_i=1$, $\Delta = 1$}
        % \label{fig:y equals x}
        \end{subfigure}
        \caption{Lifting function $\phi(\delta)$ in red and subadditive upper bound $\psi(\delta)$ in blue}
        \label{fig:sub}
\end{figure}
%   \begin{center}\label{fig:sub}
%     \includegraphics[width=\textwidth]{liftfunplot3.eps}
%   \end{center}

% \red{
We observe in Figure~\ref{fig:sub} that the lifting function is not subadditive since it is convex in a neighborhood of the origin. 
Therefore,  building a subadditive approximation is required to achieve sequence-independent lifting.
% }
%\blue{POSSIBLY ADD EXAMPLE ABOUT NON-SUBADDITIVITY.}
%\blue{A LITTLE MORE DETAILS ABOUT FIGURES.}

Building on the subadditive upper bound obtained in Theorem~\ref{thm:upperbound}, we are now able to lift the bilinear cover inequality in a sequence-independent manner.
\begin{restatable}{theorem}{ThmLifted}
\label{thm:lifted}
%{\color{blue}{GU: PROPOSED CHANGE: 
%Consider the separable bipartite bilinear set $Q$ as in definition~\ref{def:separable}.
%}}

%Consider the separable bipartite bilinear set
%\begin{eqnarray}
% \label{eq:bilincoverset}
%Q:= \left\{(x , y) \in [0, 1]^n \times [0, 1]^n \,\Bigm|\,  \sum_{i = 1}^n a_i x_iy_i \geq d \right\}.
%\end{eqnarray}
%\footnote{\red{JP: DO WE NEED TO KEEPON REINTRODUCING THIS SET?}}
Consider the separable bipartite bilinear set presented in Definition~\ref{def:separable}. 
Let $\Lambda = \{I, J_0, J_1\}$ be a minimal cover yielding partition and let $\Delta, a_{\i0}, d_i, l_{+}, l_{-}$ be defined as in Theorems~\ref{thm:valid} and~\ref{thm:upperbound} (We clarify that they are calculated using $d^\Lambda$ instead of $d$). 
Let $J_0^+$, $J_0^-$, $J_1^+$, and, $J_1^-$ be as in Definition~\ref{def:minimal}. 
Then inequality 
\begin{eqnarray}
\label{eq:liftedbilinearcoverinequality}
\sum_{i \in I} \frac{\sqrt{a_i}}{ \sqrt{a_i}-\sqrt{d_i}}\left( \sqrt{x_iy_i} - 1\right) + \sum_{i\notin I} \gamma_i (x_i, y_i) \geq -1,
\end{eqnarray}
is valid for $Q$ where 
%\begin{align*}
%\Delta &=\sum_{i\in I}a_i - d^\Lambda  = \sum_{i\in I\cup J_1} a_i - d \\
%a_{i_0} &= \min\{a_i|a_i>\Delta, i\in I\}\\
%l_+ &= \left\{\begin{array}{ll}
%     \frac{\sqrt{a_{i_0}}+\sqrt{a_{i_0}-\Delta}}{\Delta\sqrt{a_{i_0}-\Delta}}\ & \textup{if } a_{i_0} \textup{ exists} \\
%     \frac{1}{\Delta} & \textup{otherwise}
%\end{array}\right.\\
%l_- &= 1/\Delta
%\end{align*}
%and 
$\gamma_i: \Real^2 \rightarrow \Real$ for $i \in [n]\setminus I$ are the concave functions:
\begin{enumerate}[label={(\roman*)},align=left]
    \item $\gamma_i(x,y) = l_+a_i\min\{x,y\}$ for $i\in J_0^+$;
    \item $\gamma_i(x,y) = -l_+a_i\min\{2-x-y, 1\}$ for $i\in J_1^-$;
    \item $\gamma_i(x,y) = \min\{l_-a_i(x+y-1), l_+a_i(x+y-1)+l_+\Delta-1, 0\}$ for $i\in J_0^-$;
    \item 
    $\gamma_i(x,y) = \min\{\tilde{g}_i(x,y),\tilde{h}_i(x,y), g_i(x,y), h_i(x,y)\},$
    for $i\in J_1^+$ with $a_i\geq a_{\i0}$ when $I^>\neq\emptyset$, 
    and
    $\gamma_i(x,y) = \min\{\tilde{g}_i(x,y),\tilde{h}_i(x,y)\}$ in all other cases where 
    $i\in J_1^+$, 
    with
    \begin{align*}
        \tilde{g}_i(x,y)&=l_+a_i(\min\{x, y\}-1)+l_+\Delta-1\\
        \tilde{h}_i(x,y)&=l_-a_i(\min\{x, y\}-1)\\
        g_i(x,y)&= \sqrt{a_i - \Delta}\sqrt{a_i}l_+\sqrt{xy} -l_+(a_i-\Delta) -1 \\
        h_i(x,y)&= \frac{\sqrt{a_i}}{\sqrt{a_i} - \sqrt{d_i}}(\sqrt{xy} - 1).
    \end{align*}
\end{enumerate}
\end{restatable}

We refer to inequality (\ref{eq:liftedbilinearcoverinequality}) as  \textit{lifted bilinear cover inequality}. 
{This inequality  is second-order cone representable.}
The proof of Theorem~\ref{thm:lifted} can be found in Section~\ref{section:lifted}. 
%The key is to obtain good approximations of the concave envelopes of the functions $\psi(a_ixy)$ and $\psi(a_ixy-a_i)$ over the unit box. 
% We observe that not all functions $\gamma_i$ described in Theorem~\ref{thm:lifted} are affine. 
%In particular, in the case of $J^+_1$, is is possible to obtain a term of the type $\sqrt{xy}$, which appears to be 
%{incomparable with affine terms.}
%the best function that upper bounds the subadditive function $\psi$.
%\footnote{\red{JP: IS IT REALLY THE BEST?}}
%\input{Sec_End.tex}
\section{Future directions}\label{sec:future}

The results presented in this paper open up new avenues for generating cutting-planes for QCQPs. 
They also raise new theoretical and computational questions that can be investigated. 
To illustrate this assertion, we revisit Example~\ref{example:nolifting}  next.

\begin{example}\label{ex:twocont}
Consider $S:= \{ (x_1,x_2) \in [0,1]^2 \,|\, - x_1^2 - (x_2-0.5)^2\geq -0.5^2\}$ with the same fixing as in Example~\ref{example:nolifting}, \textit{i.e.}, $x_2 =0$.
For the associated restriction $S|_{x_2=0}$, consider the seed inequality $-x_1 \geq 0$. 

In contrast to our earlier discussion,  consider now the problem of lifting this seed inequality into an inequality of the form $-x_1 + \alpha \sqrt{x_2} \geq 0$. 
Finding the values of $\alpha$ that generate a valid inequality is equivalent to solving the problem
% \begin{align*}
% \begin{array}{rcl}
%     \alpha:= &\textup{sup}& \frac{x_1}{\sqrt{x_2}} \\
%     &\textup{s.t.} & -x_1^2 - (x_2-0.5)^2\geq -0.5^2\\
%     && x_1 \in [0,1], x_2 \in (0, 1].
% \end{array}
% \end{align*}
\begin{align*}
    \alpha^*:= \sup\ \left\{ \frac{x_1}{\sqrt{x_2}} \ \Bigm| \
     -x_1^2 - (x_2-0.5)^2\geq -0.5^2, \ 
     x_1 \in [0,1],\  x_2 \in (0, 1] \right\}.
\end{align*}

Using constraint $-x_1^2 - (x_2-0.5)^2\geq -0.5^2$ we can bound the objective function as:
%\footnote{\red{JP, SHOULD WE SAY "Using constraint ** we can bound the objective as follows"? SHOULD WE HAVE SAID SOMETHING MORE CAREFUL ABOUT $\alpha^*=1$ SINCE WE ONLY DESCRIBE $\alpha^* \le 1$? }. }
$
\frac{x_1}{\sqrt{x_2}} \leq \frac{\sqrt{0.5^2 - (x_2-0.5)^2}}{\sqrt{x_2}} 
=  \frac{\sqrt{(1 - x_2)(x_2)}}{\sqrt{x_2}} = \sqrt{1 - x_2}.
$
It follows that selecting $\alpha \ge \alpha^* = 1$ yields a valid inequality for $S$. Note first that $\alpha <0$ leads to an invalid inequality since $(x_1,x_2)=(0, 0.5)$ is a feasible point. 
Moreover, any $\alpha \in [0,1)$
% for any value of $\alpha$ satisfying $0 \leq \alpha<1$ 
yields an invalid inequality, since the point $(x_1,x_2)$ where $x_1 = \sqrt{x_2(1-x_2)}$ and $x_2 = 1-(\sfrac{(1+\alpha)}{2})^2$ is feasible.
% , invalidating the inequality.
Therefore, the inequality
$-x_1 + \sqrt{x_2} \geq 0$ 
is the strongest such lifted inequality.
% \hfill \qed
\end{example}
The above example raises the question of obtaining a complete characterization of when one can accomplish lifting, \textit{i.e.}, of generalizing Theorem~\ref{THM_EXISTENCE} to situations where the functional form of the lifted variable is not necessarily linear. 
It would also be valuable to develop a theory to accomplish sequence-independent lifting in the more general case of bipartite bilinear programs, instead of just the separable case. 
%Even for the separable case, one could also explore the possibility of constructing other 
%subadditive upper bounds to the lifting function. 
%\footnote{\red{JP: WHAT DOES NON-DOMINATING MEAN HERE?}}
On the computational side, one key question is to understand the complexity of separating the lifted bilinear cover inequality presented in Theorem~\ref{thm:lifted} and to design efficient computational schemes to perform separation. 
%If needed, just like in the case of MILPs, approximate procedures for separation may be needed to be devised. 
Finally, extensive numerical experiments should be conducted to understand the practical strength of these inequalities and to determine how useful they can be in the solution of QCQPs. 
Given the strength of the seed inequality, we are hopeful that these lifted inequalities could yield nontrivial dual bound improvements. 
\section{Proof of Theorem~\ref{thm:existence}}
\label{section:existence}
%\footnote{\red{JP: Title of this section used to be ``Existence of lifting function"}}
%\ThmExistence*
\begin{proof}[Theorem~\ref{thm:existence}]
Without loss of generality, we assume that we lift a component of the variable $x$, say $x_k$ with $k\in [m]\backslash C$. In addition, we assume $\tilde{x}_k=0$; if not we may perform the operation $x_k\leftarrow 1-x_k$ and $f_k\leftarrow -f_k$. 

In order to find a lifting coefficient, We examine the following optimization problem
\begin{eqnarray*}
u_k(x_k):=\frac{1}{x_k}&\max &r - h(x_C, y_D)\\
&\st & x\tr Q y + a\tr x + b\tr y \geq c,\\
&& x_C, y_D\in [0,1],\ x_{[m]\backslash C\backslash\{k\}}=\tilde{x}_{[m]\backslash C\backslash\{k\}}, y_{[n]\backslash D}=\tilde{y}_{[n]\backslash D}.
\end{eqnarray*}
Now note that $u_k^* = \sup_{x_k\in(0,1]} u_k(x_k)$, assuming it exists, is a valid the coefficient for lifting, \textit{i.e.}, 
% , while noting that $u_k(\tilde{x}_k)\leq 0.$
% In case that $u_k^*$ exists, 
$
h(x_C, y_D) + u_k^* x_k \geq r
$
is a valid lifted inequality. 
Any coefficient larger than $u_k^*$ is also valid for lifting.

From the concavity of $h$ (\textit{i.e.}, convexity of $r-h$), for any specific $x_k$ the optimal solution must be an extreme point. 
According to~\cite{dey2019new}, all extreme points satisfy the following property: except one pair of $(x_i, y_j)$, all other $x_{i'}, y_{j'}$ pairs will be equal to either $0$ or $1$. 
Thus, for any pair of partitions $\{i\}\cup I_0\cup I_1 = C$ (denoted by $I$) and $\{j\}\cup J_0\cup J_1 = D$ (denoted by $J$), define
\begin{eqnarray*}
u_{I,J}(x_k):=\frac{1}{x_k}&\max &r - h(x_C, y_D)\\
&\st & x\tr Q y + a\tr x + b\tr y \geq c,\\
&& x_i,y_j\in [0,1],\ x_{I_0} = 0, x_{I_1} = 1, y_{J_0} = 0, y_{J_1} =1,\\
&& x_{[m]\backslash C\backslash\{k\}}=\tilde{x}_{[m]\backslash C\backslash\{k\}}, y_{[n]\backslash D}=\tilde{y}_{[n]\backslash D}.
\end{eqnarray*}
We clearly have $u_k(x_k)=\max_{I,J} u_{I,J}(x_k)$. 
In addition, observe that $u_k^*=\max_{I,J} u_{I,J}^*$ where $u_{I,J}^* = \sup_{x_k\in(0,1]} u_{I,J}(x_k)$. 
Therefore in order to prove that $u_k^*<\infty$, it is sufficient to show that for any partition $I,J$, $u_{I,J}^*<\infty$. 
Therefore, we now focus on one instance of such partitions.

{
We define $\tilde{x} \in \mathbb{R}^m$ and $\tilde{y} \in \mathbb{R}^n$ as: $(\tilde{x}_{I,J})_{I_0\cup\{i\}} = 0$, $
(\tilde{x}_{I,J})_{I_1}=1$, $ (\tilde{y}_{I,J})_{J_0\cup\{j\}} = 0$, $(\tilde{y}_{I,J})_{J_1} = 1.$
% \begin{align*}
% (\tilde{x}_{I,J})_{I_0\cup\{i\}} &= 0,
% (\tilde{x}_{I,J})_{I_1}=1,\\
% (\tilde{y}_{I,J})_{J_0\cup\{j\}} &= 0,
% (\tilde{y}_{I,J})_{J_1} = 1.
% \end{align*}
In addition, define $r_{I,J} := r-p_{I_1}\tr \ONE_{I_1}-q_{J_1}\tr \ONE_{J_1}$, $c_{I,J}:= c-a\tr \tilde{x}_{I,J}-b\tr \tilde{y}_{I,J}-\tilde{x}_{I,J}\tr Q \tilde{y}_{I,J}$, $a_{I,J}:=a_i+ Q_{i, *} \tilde{y}_{I,J}
$, $b_{I,J} := b_j + \tilde{x}_{I,J}\tr Q_{*, j}$, and $a_{I,J,k} := a_k + Q_{k, *}\tilde{y}_{I,J}$
% \begin{eqnarray*}
% r_{I,J}&:=&r-p_{I_1}\tr \ONE_{I_1}-q_{J_1}\tr \ONE_{J_1}\\ 
% c_{I,J}&:=&c-a\tr \tilde{x}_{I,J}-b\tr \tilde{y}_{I,J}-\tilde{x}_{I,J}\tr Q \tilde{y}_{I,J},\\ 
% a_{I,J}&:=&a_i+ Q_{i, *} \tilde{y}_{I,J},\\ 
% b_{I,J} &:=& b_j + \tilde{x}_{I,J}\tr Q_{*, j},\\
% a_{I,J,k} &:=& a_k + Q_{k, *}\tilde{y}_{I,J},
% \end{eqnarray*}
so that we have equivalently}
%\footnote{\red{JP: Are these computations useful? I would propose removing what is in red and going directly to the problem}}
\begin{eqnarray*}
u_{I,J}(x_k)=\frac{1}{x_k}&\max &r_{I,J} - h_{I,J}(x_i, y_j)\\
&\st & q_{ij}x_iy_j + a_{I,J} x_i + b_{I,J}y_j + a_{I,J,k}x_k + q_{kj} x_k y_j \geq c_{I,J},\\
&& (x_i,y_j) \in [0,1]^2,
\end{eqnarray*}
where $h_{I,J}$ is $h$ after the appropriate restriction. Note that $h_{I,J}$ is concave. As we are focusing on the pair of partitions $I,J$, for simplicity we rewrite the problem as
\begin{eqnarray*}
u(\hat{x}):=\frac{1}{\hat{x}}&\max_{x,y} &r - h(x, y)\\
&\st & qxy + ax+by+\hat{a}\hat{x} + \hat{q} \hat{x}y \geq c,\ (x,y)\in[0,1]^2,
\end{eqnarray*}
and $u^*:=\sup_{\hat{x}\in(0,1]} u(\hat{x})$. 
It remains to prove $u^* < \infty$.

For any $\epsilon \in (0,1]$ and $\hat{x}\in [\epsilon,1]$, we have
\begin{eqnarray*}
u(\hat{x})= &\max_{x,y} &\{ \ \frac{1}{\hat{x}}(r - h(x, y))\
\bigm| \ qxy + ax+by+\hat{a}\hat{x} + \hat{q} \hat{x}y \geq c,
\ (x,y)\in[0,1]^2 \ \},\\
\leq &\max_{x,y} &\{ \ \frac{1}{\hat{x}}(r - h(x, y)) \ \bigm| \
 (x,y)\in[0,1]^2 \ \} \\
\leq &\max_{x,y} &\left\{\ \max\{\frac{1}{\epsilon}(r - h(x, y)),(r - h(x, y))\}\ \bigm| \  (x,y)\in [0,1]^2 \ \right\}\\
:=& w <  \infty.
\end{eqnarray*}
It is clear that $u(\hat{x})\leq w<\infty$ for any $\hat{x}\in [\epsilon,1]$. Therefore, to show that $u^*<\infty$, it is sufficient to show that $\limsup_{\hat{x}\downarrow 0}u(\hat{x})<\infty$.
We define
\begin{subequations}\label{sub*}
\renewcommand{\theequation}{\arabic{parentequation}.\arabic{equation}}
\begin{align}
    v(\hat{x}) = \max\ & r - h(x, y) \nonumber \\
\st\ &  x\geq 0, \label{sub1*} \\
& x\leq 1, \label{sub2*}\\
& y\geq 0, \label{sub3*}\\
& y\leq 1, \label{sub4*}\\
&qxy + ax+by+\hat{a}\hat{x} + \hat{q} \hat{x}y \geq c. \label{sub5*}
\end{align}
\end{subequations}
Denote the feasible region of \eqref{sub*} as $S(\hat{x})$.
Since $v(0) \leq 0$ {(because the seed inequality is assumed to be valid for the restriction)}, one can prove that $\limsup_{\hat{x}\downarrow 0}u(\hat{x}) <\infty$ by showing that there exists $l<\infty$ such that $$v(\hat{x})-v(0)\leq l \hat{x}+o(\hat{x}) \textup{ for }\hat{x}\downarrow 0.
\footnote{This is equivalent to saying $\limsup_{\hat{x}\downarrow 0} \frac{v(\hat{x}) - v(0)}{\hat{x}} \leq l$ }$$ We denote the feasible region of the above problem as $S(\hat{x})$.

For $i \in \{1,\ldots,5\}$, we define $v_i(\hat{x})$ to be the optimal value of \eqref{sub*} where constraint $(\ref{sub*}.i)$ is at to equality. 
We use $S_i(\hat{x})$ to denote the corresponding feasible region.  
For example, 
\begin{eqnarray*}
v_a(\hat{x}) = &\max & \{ \ r - h(x, y) \ \bigm| \ (x,y)\in S_1(\hat{x}) \ \} \\
 = &\max & \{ \ r - h(x, y) \ \bigm|  \  x=0,\ (x,y) \in S(\hat{x}) \ \} \\
= &\max & \{ \ r - h(0, y) \ \bigm| \ by+\hat{a}\hat{x} + \hat{q} \hat{x}y \geq c,\  y\in[0,1] \ \}.
\end{eqnarray*}
%and similarly $v_b$, $v_c$, $v_d$ and $v_e$. 

Note that $v(\hat{x}) = \textup{max}_{i \in \{1, \dots, 5\}}\{v_i(\hat{x})\}$, 
%\footnote{\red{JP: Why not equal?}}
%\footnote{\orange{XY: changed to equal}}
since the objective function in computing $v(\hat{x})$ is maximizing a convex function, implying that there exists an optimal solution where at least one of the constraints \eqref{sub1*}-\eqref{sub5*} is active.

Thus, to prove that $\limsup_{\hat{x}\downarrow 0}u(\hat{x}) <\infty$ it suffices to show that there exists $l < \infty$ such that
\begin{eqnarray}\label{eq:toproveforthem1}
v_i(\hat{x})-v(0)\leq l \hat{x}+o(\hat{x}) \textup{ for } \hat{x}\downarrow 0 \ \textup{ for all } i \in \{1, \dots, 5\}.
\end{eqnarray}

\paragraph{The case of $v_1, v_2, v_3, v_4$:} 
We present a proof of (\ref{eq:toproveforthem1}) for the case of $v_1$. 
The proof is similar for the cases of $v_2$, $v_3$, and $v_4$.

First it is straightforward to verify that there exists a sufficiently small $\hat{x}_0>0$ such that for any $0<\hat{x}<\hat{x}_0$, we have one of the following two cases:
% 1.) feasible region gets larger from $0$ to $0_+$;
\ref{whatwascase2} $S_1(\hat{x})\backslash S_1(0) = \emptyset$ (including the case $S_1(\hat{x}) = \emptyset$) and 
\ref{whatwascase1} $S_1(\hat{x})\backslash S_1(0)\neq \emptyset$,
% 2.) feasible region remains the same, gets smaller or becomes empty (i.e. infeasible) from $0$ to $0_+$;
% 3.) feasible at neither $0$ nor $0_+$.
% For each $v_i$, in both cases of 2.) and 3.), we have 
as, for sufficiently small $\hat{x}_0$, we may assume that  it is impossible that $S_1(\hat{x})\neq\emptyset=S_1(0)$. 
%for $\hat{x}_0\downarrow 0$.
%For each $v_i$, 
\begin{enumerate}[label={(\roman*)}]
\item \label{whatwascase2}
We have $v_1(\hat{x})\leq v_1(0)\leq v(0)\leq 0$ (this holds even in the case when $S_1(\hat{x}) = \emptyset$ or $S_1(\hat{x}) = S_1(0) = \emptyset$), \textit{i.e.}, $v_1(\hat{x})-v(0)\leq 0\cdot \hat{x}$.

%For $v_1=v_{x=0}$ in case 1.), note that 
\item \label{whatwascase1}
We consider two sub-cases: 
\begin{enumerate}[label={(\alph*)}]
\item if $b=0$, the feasibility of $\hat{x}=0$ yields $c\leq 0$. Thus, we have $S_1(0) = [0,1]\supseteq S_1(\hat{x})$ for $\hat{x}\in (0,\hat{x}_0)$ (actually in case \ref{whatwascase2}.).

\item if $b\neq 0$, assume first that $b<0$.
Then $S_1(\hat{x}) = \{y\in [0,1]: y \leq (c-\hat{a}\hat{x})/(b+\hat{q}\hat{x})\}$. We also denote $\Delta(\hat{x}):= (c-\hat{a}\hat{x})/(b+\hat{q}\hat{x}) - c/b$  and since $b < 0$, we have that $|\Diff\Delta(\hat{x})/\Diff \hat{x}| < \infty$ for $\hat{x}=0$.

Since $S_1(\hat{x})\backslash S_1(0)\neq \emptyset$, we have $0 \leq c/b<1$ as well as $\Delta(\hat{x})\geq 0$.

Utilizing the fact that $\Delta (\hat{x})\in[0,1-c/b]$ (the upper bound from the fact that $\hat{x}$ is assumed to be sufficiently small) and the concavity of $h$, we obtain
\begin{equation}
    \label{eq:concavelinear}
h\left(0,\frac{c}{b}+\Delta(\hat{x})\right)\geq \frac{\Delta(\hat{x})}{1-c/b} h(0,1)+\left(1-\frac{\Delta(\hat{x})}{1-c/b}\right) h\left(0, \frac{c}{b}\right).
\end{equation}
We now have %for $\hat{x}\downarrow 0$,
\begin{eqnarray*}
v_1(\hat{x})-v(0)&\leq& v_1(\hat{x})-v_1(0)\\
& \leq& \max\left\{(r - h(0, 0)) - ( r - h(0,0)),\right.\\
&&\left.(r - h(0, \Delta(\hat{x})+c/b) - (r - h(0, c/b) )\right\}\\
& \leq& \max\left\{0,\frac{h(0,c/b)-h(0,1)}{1-c/b} \Delta(\hat{x}) \right\}\\
& =& \max\left\{0,\left.\frac{h(0,c/b)-h(0,1)}{1-c/b}\frac{\Diff\Delta(\hat{x})}{\Diff \hat{x}}\right|_{\hat{x}= 0}\hat{x}+o(\hat{x})\right\},\ \ \ \  (\hat{x}\downarrow 0)
\end{eqnarray*}

where the second inequality comes from the fact that $v_1(\hat{x}) = \max \{ r- h(0,0), r - h(0, \Delta + c/b)\}$ and $v_1(0) = \max \{ r- h(0,0), r - h(0, c/b)\}$, and the third inequality follows from (\ref{eq:concavelinear}), and the last equality follows from Taylor's series expansion of $\frac{h(0,c/b)-h(0,1)}{1-c/b}\Delta(\hat{x})$ around $\hat{x} = 0$ .

Thus, there exists $l_1<\infty$ such that $v_1(\hat{x})-v(0)\leq l_1 \hat{x}+o(\hat{x})$ for $\hat{x}\downarrow 0$. A similar argument holds for the case of $b>0$.
\end{enumerate}
\end{enumerate}
%Similar process also applies to $v_2=v_{x=1}$, $v_3=v_{y=0}$ and $v_4=v_{y=1}$ as we utilizing the concavity of $h$ to get linear upper bounds of $r-h$. 

%Combining the discussions, for $i=1,2,3,4$, there exists $l_i<\infty$ such that $v_i(\hat{x})-v(0)\leq l_i\hat{x}+o(\hat{x})$ for $\hat{x}\downarrow 0$.

\paragraph{The case of $v_5$:}

If $q=0$, then it is easy to see that there always exists an optimal solution to the optimization problem corresponding to computing $v(\hat{x})$ such that one of \eqref{sub1*}-\eqref{sub4*} is active. 
Therefore if $q = 0$, it is sufficient to verify (\ref{eq:toproveforthem1}) for $v_1$, $v_2$, $v_3$, and $v_4$ as $v_5(\hat{x}) \leq \max \left\{v_1(\hat{x}), v_2(\hat{x}), v_3(\hat{x}), v_4(\hat{x})\right\}$. 

%the discussion on $v_i$ for $i=1,2,3,4$ covers all possible cases of the optimal solution, and thus $v(\hat{x})-v(0) = \max_{i\in[4]}\{v_i(\hat{x})\}-v(0)\leq \max_{i} l_i\hat{x}+o(\hat{x})$ for $\hat{x}\downarrow 0$.

Therefore, we consider the case of $v_5$ for $q\neq 0$. Without loss of generality assume that $q>0$ or perform the transformation $x\leftarrow 1-x$. We in addition assume that $q=1$ or we can scale all parameters by $1/q$. The problem can now be rewritten as
%\begin{eqnarray*}
%    v(\hat{x}) = &\max & r - h(x,y)\\
%&\st & (x+b+\hat{q}\hat{x})(y+a)\geq c+ab +(\hat{q}a-\hat{a})\hat{x},\\
%&& \ x,y\in[0,1].
%\end{eqnarray*}
\begin{eqnarray*}
 v_5(\hat{x}) := &\max & r - h(x,y)\\
&\st & (x+b+\hat{q}\hat{x})(y+a)= c+ab +(\hat{q}a-\hat{a})\hat{x},\ (x,y)\in[0,1]^2.
\end{eqnarray*}
We denote its feasible region by $S_5(\hat{x})$.

The feasible region is the boundary of a hyperbola intersected with the $[0, 1]^2$ box. 
If both the connected components of the hyperbola intersect the $[0, 1]^2$ box, or $c+ab+(\hat{q}a-\hat{a})\hat{x}\leq 0$, then it is easy to see that there exists an optimal solution of the optimization problem corresponding to computation of $v(\hat{x})$ where at least one of \eqref{sub1*}-\eqref{sub4*} is active, \textit{i.e.}, $v_5(\hat{x}) \leq \max \left\{v_1(\hat{x}), v_2(\hat{x}), v_3(\hat{x}), v_4(\hat{x})\right\}$. So we can disregard this case as well and assume that only one of the connected components of the hyperbola is feasible, as well as $c+ab+(\hat{q}a-\hat{a})\hat{x}\geq 0$. 

Note again that if $S(0) = \emptyset$, then there exists $\hat{x}_0 > 0$ such that $S(\hat{x}) = \emptyset$ for all $0 \leq \hat{x} < \hat{x}_0$ and thus $v_5(\hat{x}) \leq v(\hat{x}) \leq 0$. 
%\footnote{\red{JP: Are week talking $S_5$ here?}}
%\footnote{\orange{XY: I think it should be $S$ instead of $S_5$ here}}
Therefore we may assume that $S(0) \neq \emptyset.$ 
Let $h^M := \max\{\max_{(x,y)\in [0,1]^2} h(x,y), 0\}$ and $h^m = \min\{\min_{(x,y)\in [0,1]^2} h(x,y), 0\}$.

%A similar argument shows that $S(0)$

%Again, for $\hat{x}_0>0$ sufficiently small and any $0<\hat{x}<\hat{x}_0$, we consider the maximizing problem $v$ to be in the following cases 

%1.) $S(\hat{x})\backslash S(0)\neq \emptyset$;

%2.) $S(\hat{x})\backslash S(0) = \emptyset$, including the case $S(\hat{x})=\emptyset$.

%In the case of 2.) we again have $v(\hat{x})-v(0)\leq 0 = 0\cdot \hat{x}$.

%In the case of 1.) if the optimal solution satisfies $x\in \{0,1\}$ or $y\in\{0,1\}$, we have $v(\hat{x})-v(0)\leq \max_{i\in[4]} l_i\hat{x}+o(\hat{x})$ for $\hat{x}\downarrow 0$. 

%We now consider the case where the optimal solution has $x,y\in (0,1)$.
%In this case, only one part of the hyperbola along with its convex hull can be feasible or it is impossible that $x,y\in (0,1)$. The optimal solution in this case must be within the curve of the hyperbola. Thus, in this case we denote
%\begin{eqnarray*}
% v_5(\hat{x}) := &\max & r - h(x,y)\\
%&\st & (x+b+\hat{q}\hat{x})(y+a)= c+ab +(\hat{q}a-\hat{a})\hat{x},\\
%&& \ x,y\in[0,1],
%\end{eqnarray*}
%and its feasible region as $S_5(\hat{x})$.

\begin{enumerate}[label={(\roman*)}]
\item
We first consider the case when $c+ab = 0$. Feasibility of $S(0)$ requires $-b\in[0,1]$ or $-a\in[0,1]$. In addition, as only one part of the hyperbola is feasible for $S(\hat{x})$ for $\hat{x} >0$ (and sufficiently small), we obtain that either $-a\not\in (0,1)$ or $-b\not\in(0,1)$, and in addition $\hat{q}a-\hat{a}\geq 0$.
\label{casei}

\begin{enumerate}[label={(\alph*)}]
\item \label{caseia}
If $-a \in[0,1]$ and $-b\notin[0,1]$ (see Fig~\ref{fig:exist_1a}),
% \begin{figure}
%     \centering
%     \includegraphics[width=0.75\textwidth]{exist_1a.eps}
%     \caption{Case i. a}
%     \label{fig:exist_1a}
% \end{figure}
%
%----------------------------------
\begin{figure}
     \centering
     \begin{subfigure}[b]{0.43\textwidth}
    \centering
    \includegraphics[width=\textwidth]{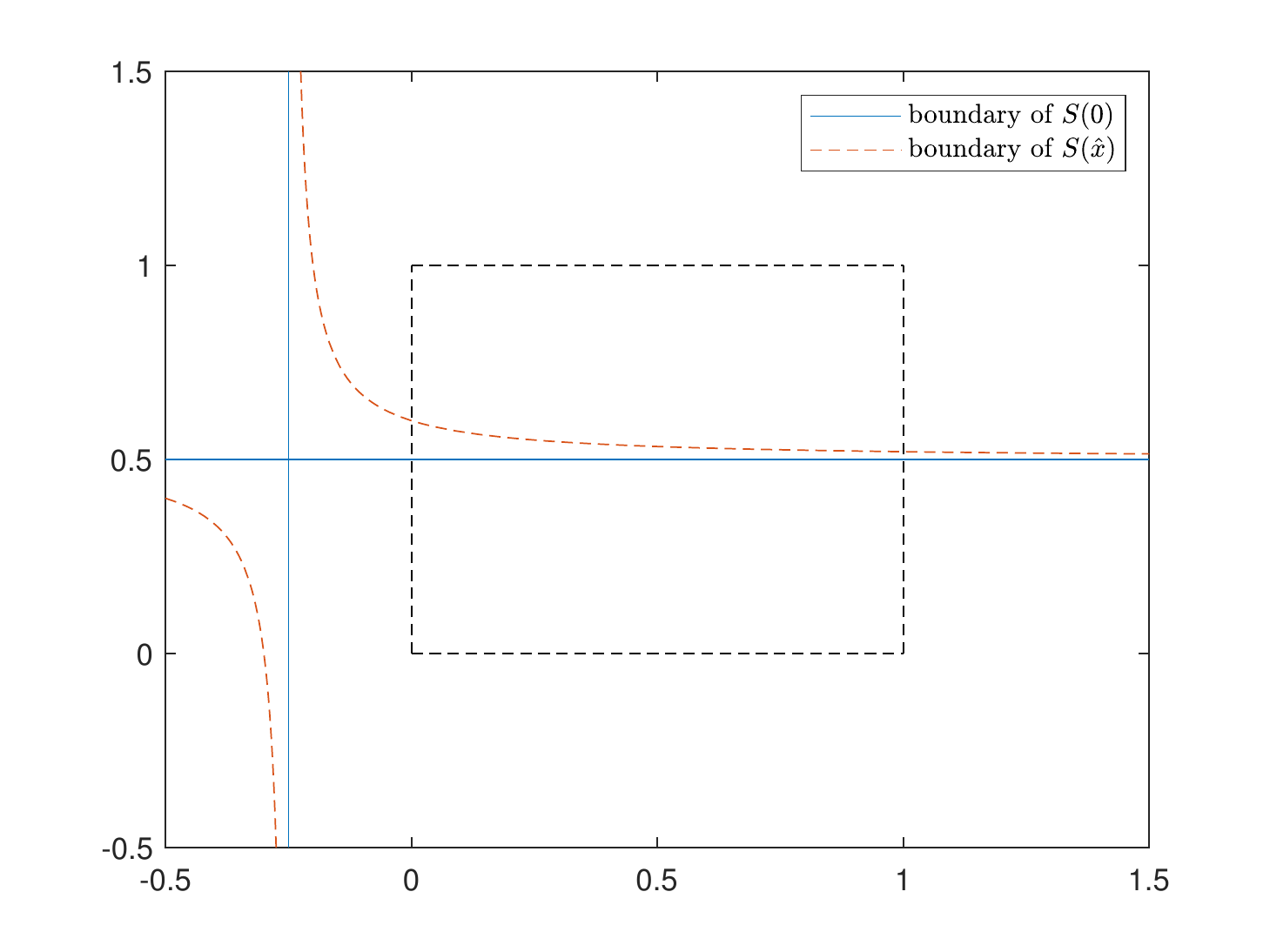}
    \caption{Case \ref{casei}\ref{caseia}}
    \label{fig:exist_1a}
     \end{subfigure}
     \begin{subfigure}[b]{0.43\textwidth}
    \centering
    \includegraphics[width=\textwidth]{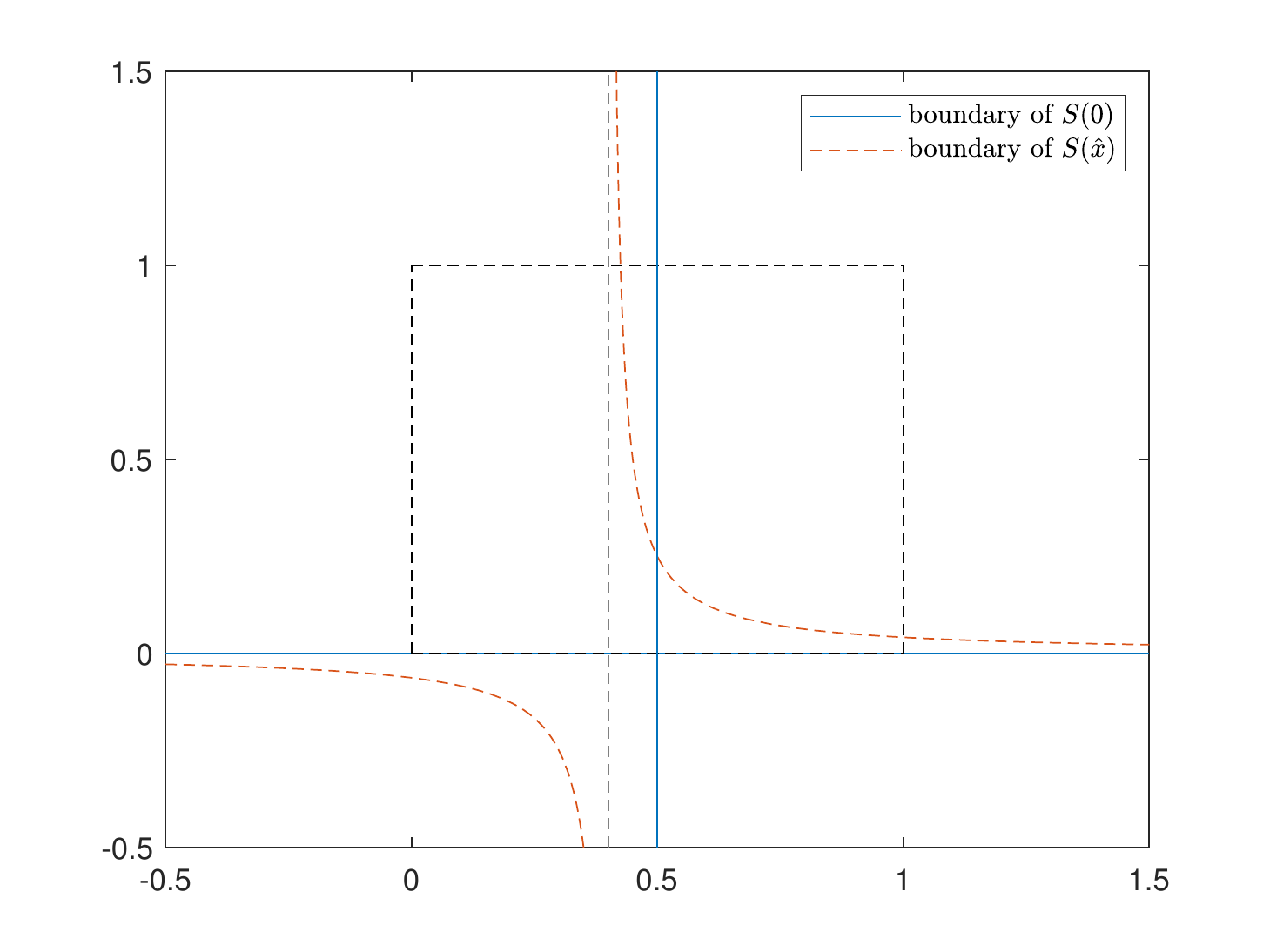}
    \caption{Case \ref{casei}\ref{caseid1}1}
    \label{fig:exist_1d1}
     \end{subfigure}
     \\
     \begin{subfigure}[b]{0.43\textwidth}
    \centering
    \includegraphics[width=\textwidth]{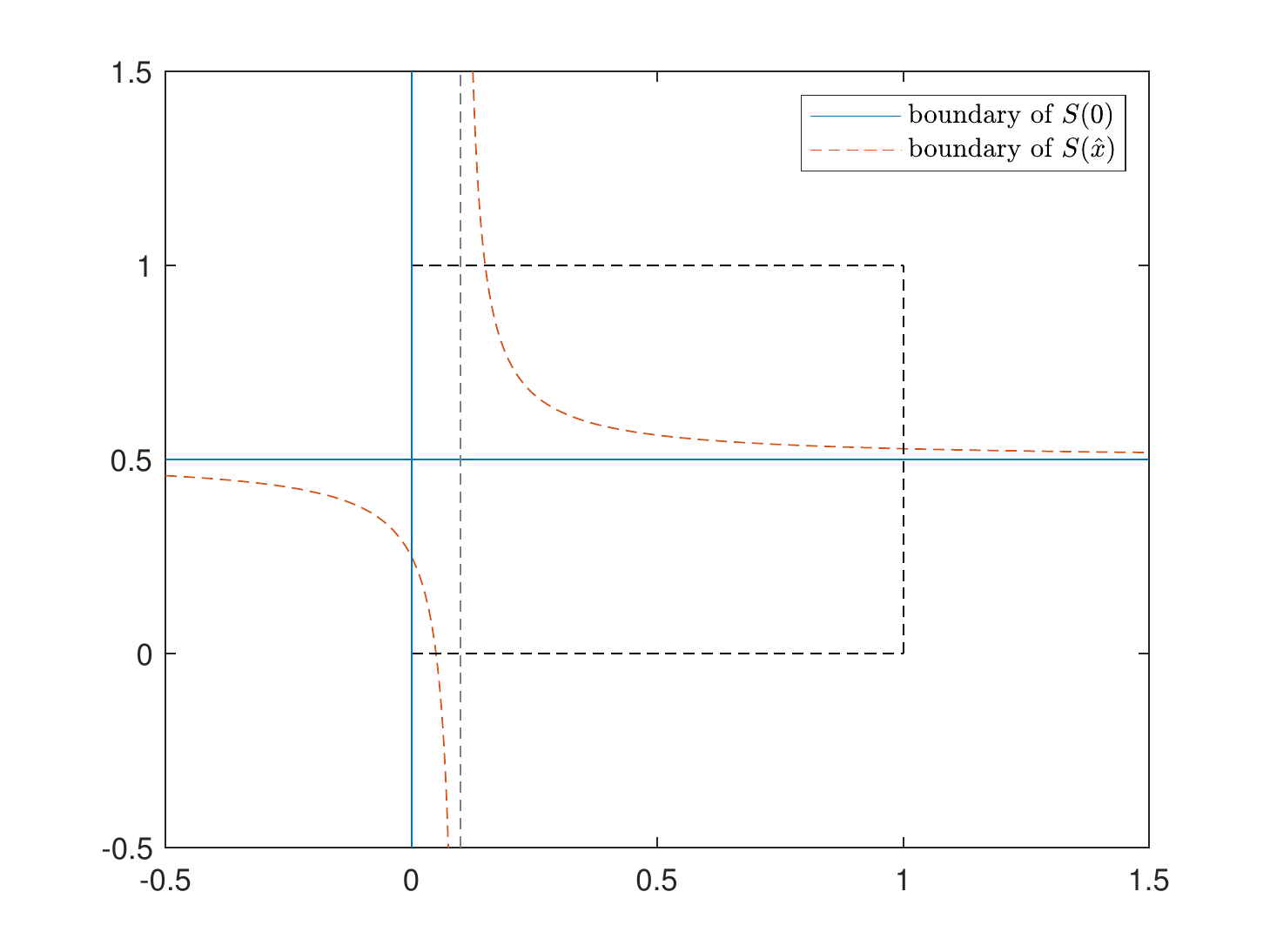}
    \caption{Case \ref{casei}\ref{caseie1}1}
    \label{fig:exist_1e1}
     \end{subfigure}
      \begin{subfigure}[b]{0.43\textwidth}
    \centering
    \includegraphics[width=\textwidth]{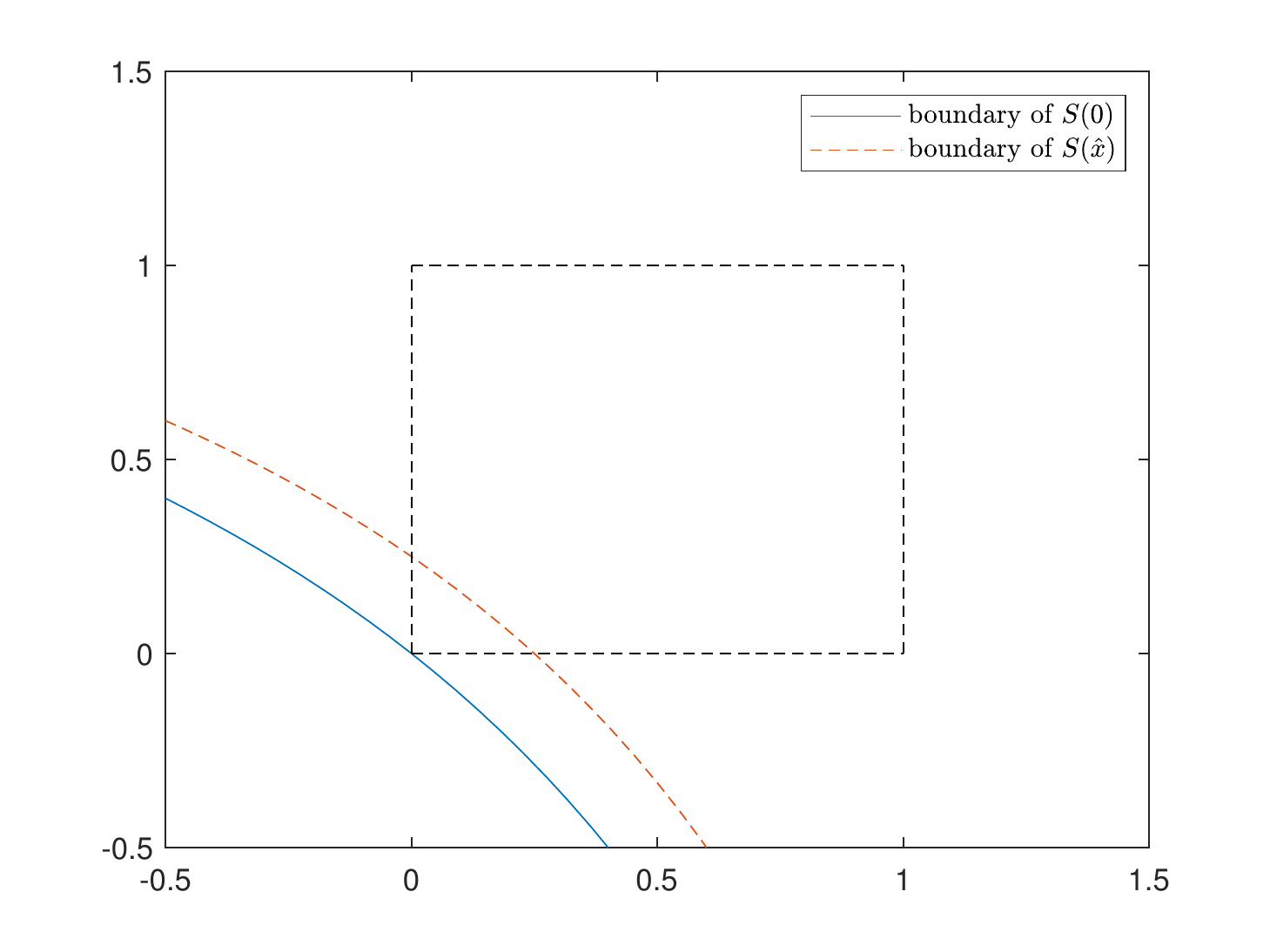}
    \caption{Case \ref{caseii}\ref{caseiia}1}
    \label{fig:exist_2a1}
     \end{subfigure}    
     \\
     \begin{subfigure}[b]{0.43\textwidth}
    \centering
    \includegraphics[width=\textwidth]{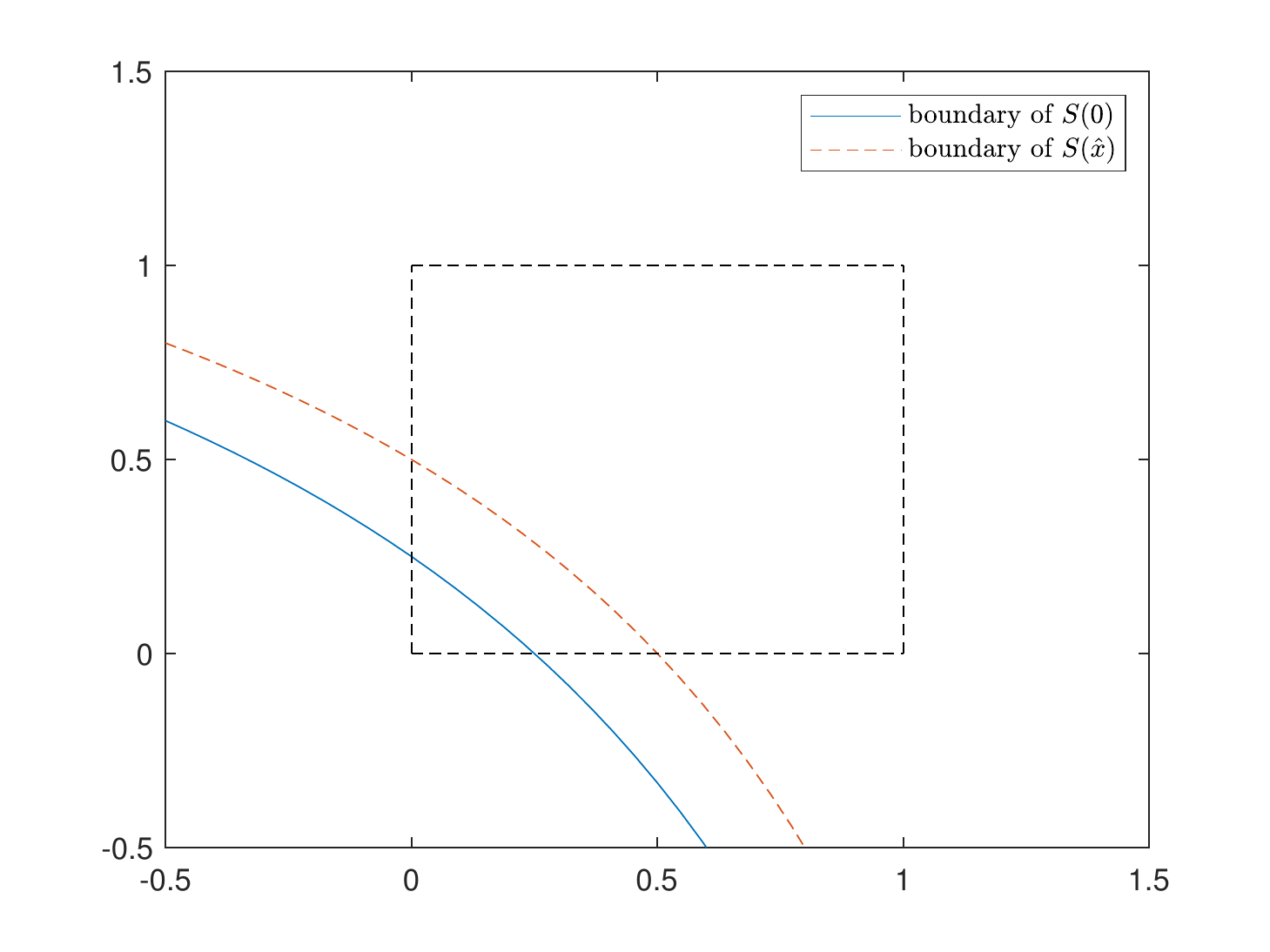}
    \caption{Case \ref{caseii}\ref{caseiia}2}
    \label{fig:exist_2a2}
     \end{subfigure}
      \begin{subfigure}[b]{0.43\textwidth}
    \centering
    \includegraphics[width=\textwidth]{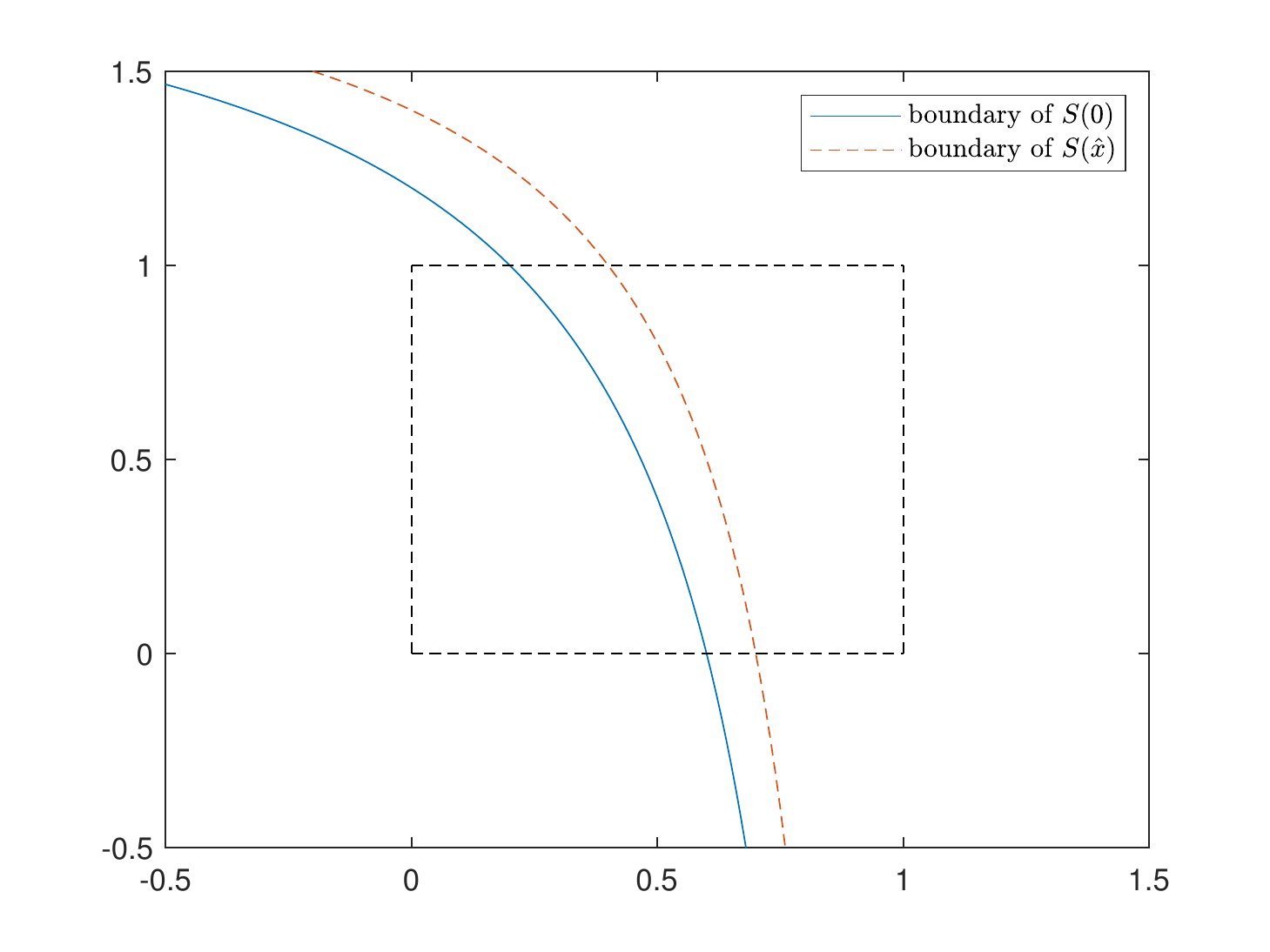}
    \caption{Case \ref{caseii}\ref{caseiia}3}
    \label{fig:exist_2a3}
     \end{subfigure}    
     \\
          \begin{subfigure}[b]{0.43\textwidth}
    \centering
    \includegraphics[width=\textwidth]{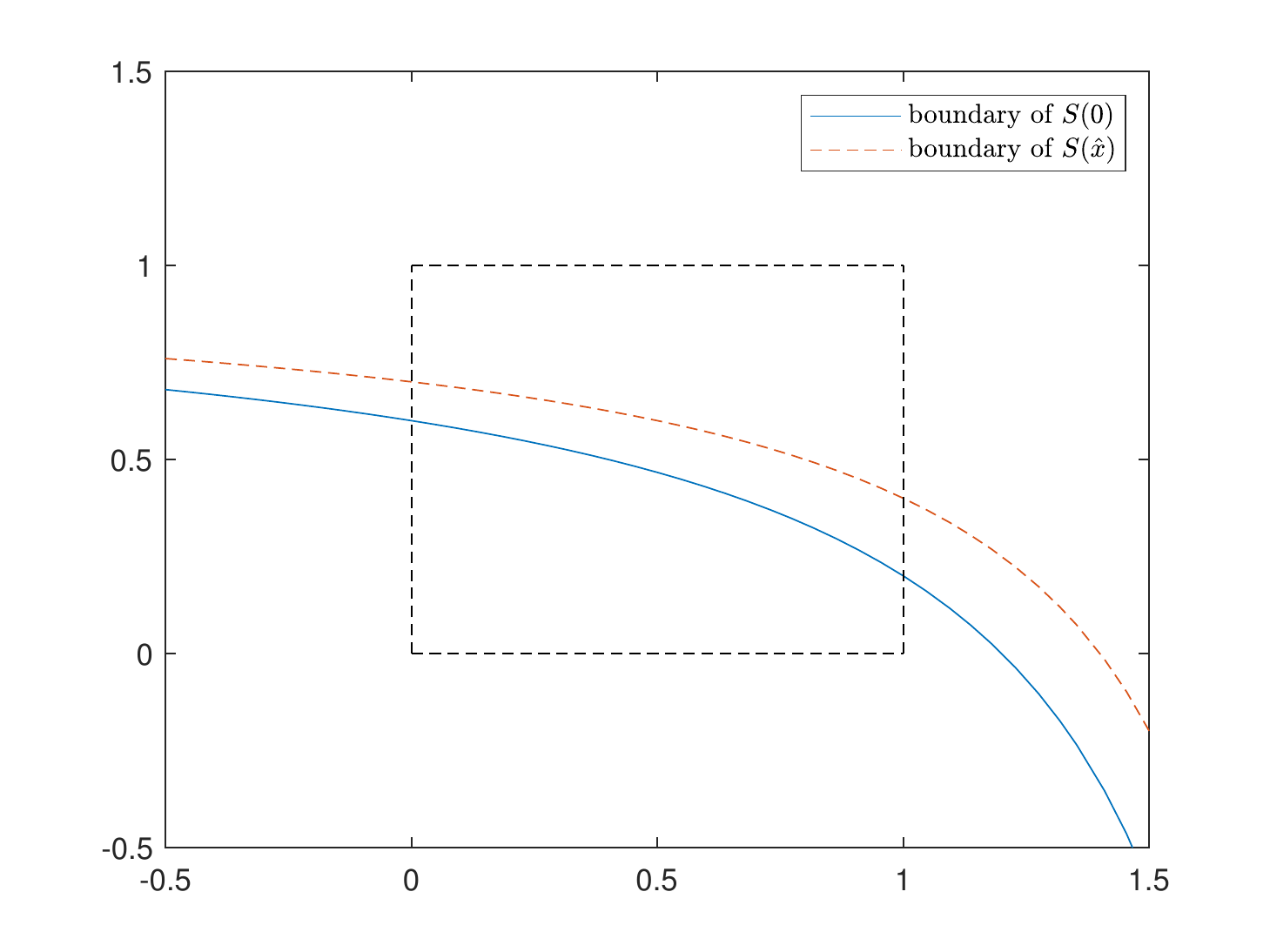}
    \caption{Case \ref{caseii}\ref{caseiia}4}
    \label{fig:exist_2a4}
     \end{subfigure}
      \begin{subfigure}[b]{0.43\textwidth}
    \centering
    \includegraphics[width=\textwidth]{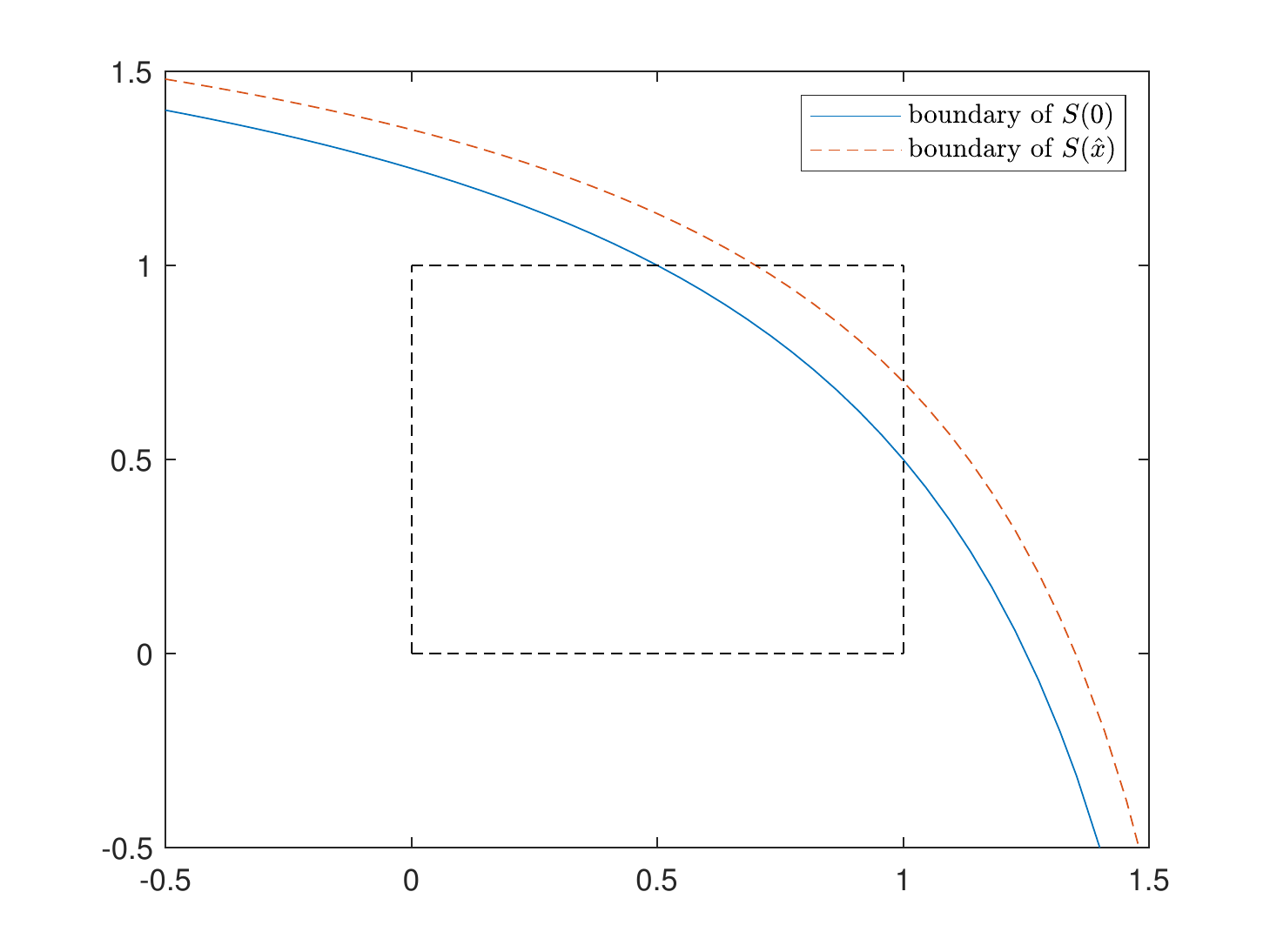}
    \caption{Case \ref{caseii}\ref{caseiia}5}
    \label{fig:exist_2a5}
     \end{subfigure}    
        % \caption{Three simple graphs}
        % \label{fig:three graphs}
\end{figure}
%------------------------------------
as clearly $(x,-a)\in S(0)$, we have 
\begin{eqnarray*}
    && v_5(\hat{x})-v(0)\\
    &\leq& \max\left\{\begin{array}{ll} \max & r - h(x, y)- v(0) \\
% \st & xy + ax+by+\hat{a}\hat{x} + \hat{q} \hat{x}y = c\\
\st& (x,y)\in S(0) \\ \textrm{ } \end{array},\right.\left.\begin{array}{ll} \max & (r  - h(x, y)) - (r - h(x,-a) ) \\
\st & xy + ax+by+\hat{a}\hat{x} + \hat{q} \hat{x}y = c\\
& (x,y)\in[0,1]^2\backslash S(0)\end{array}
        \right\}\\
&\leq& \max\left\{0,\ \begin{array}{ll}\max & h(x, -a) - h(x, y)\\
\st & y = \frac{(\hat{q}a-\hat{a})\hat{x}}{x+b+\hat{q}\hat{x}}-a,\\
        & (x,y)\in[0,1]^2\backslash S(0)\end{array}
        \right\}.
\end{eqnarray*}

%\footnote{\red{JP: The last line of (i) says that this quantity is nonnegative... Why consider two cases then?}}
%\footnote{\orange{XY: Propose Fix}}
Thus, depending on the sign of {
%\sout{$\hat{q}a-\hat{a}$}
$b$ (i.e. sign of $x+b+\hat{q}\hat{x}$ for sufficiently small $\hat{x}$)}, 
for $(x,y)\in[0,1]^2\backslash S(0)$ we have, using concavity of $h$, either
$$
h(x, y) \geq \frac{\Delta(x, \hat{x})}{1+a} h(x, 1) + (1-\frac{\Delta(x, \hat{x})}{1+a})h(x, -a)\geq \frac{\Delta(x, \hat{x})}{1+a} (h^m-h^M) + h(x, -a)
$$
or
$$
h(x, y) \geq \frac{-\Delta(x, \hat{x})}{a} h(x, 0) + (1-\frac{-\Delta(x, \hat{x})}{a})h(x, -a)\geq \frac{-\Delta(x, \hat{x})}{a} (h^m-h^M) + h(x, -a)
$$
where
$
\Delta(x,\hat{x}) := \frac{(\hat{q}a-\hat{a})\hat{x}}{x+b+\hat{q}\hat{x}}
$
is greater than $0$ in the first case and is less than $0$ in the second. 
%Specially, one should notice that for $a=0$ and $a = -1$ it will always be the first case and the second case, respectively, due to the fact that $S_5(\hat{x})\backslash S(0)\neq \emptyset$.
%Whichever the cases, from the continuity of $h$ we can get $\xi$ independent of $x$ and $\hat{x}$ such that 
From the continuity of $h$ we can get $\xi$ independent of $x$ and $\hat{x}$ such that 
$
h(x, y) \geq h(x, -a) + \xi\Delta(x, \hat{x}) .
$

Therefore, we conclude that 
\begin{eqnarray*}
v_5(\hat{x})-v(0) &\leq&  \max\left\{0, \max_{x\in [0,1]}-\xi\Delta(x, \hat{x})\right\}
=\max\left\{0, \max_{x\in\{0,1\}}-\xi\Delta(x, \hat{x})\right\}\\
&=& \max\left\{0, \max_{x\in\{0,1\}}\left.\frac{-\xi \partial \Delta(x, \hat{x})}{\partial \hat{x}}\right|_{\hat{x} =0}\hat{x}+o(\hat{x})\right\},\ \ \ \ (\hat{x}\downarrow 0)
\end{eqnarray*}
where the second equation comes from the monotonicity of $\Delta(x, \hat{x})$ on $x\in[0,1]$ for sufficiently small $\hat{x}$, due to the fact that $-b\notin[0,1]$. Since $x + b \neq 0$, we have that $\left|\max_{x\in\{0,1\}}\left.\frac{-\xi \partial \Delta(x, \hat{x})}{\partial \hat{x}}\right|_{\hat{x} =0}\right| < \infty$.
Thus, there exists $l<\infty$ such that $v_5(\hat{x})-v(0)\leq l \hat{x}+o(\hat{x})$ for $\hat{x}\downarrow 0$.

\item 
If $-b \in[0,1]$ and $-a\notin[0,1]$, a similar analysis can be conducted to obtain $l < \infty$ such that $v_5(\hat{x})-v(0)\leq l \hat{x}+o(\hat{x})$ for $\hat{x}\downarrow 0$.
%gives $l<\infty$.

\item
If $-a= -b=0$ or $-a=-b = 1$, then $S(0) = [0,1]^2$ so that $v_5(\hat{x})-v(0)\leq 0$.

\item \textit{(subcase 1)} If $-a = 0$ with $-b\in (0,1]$ (see Fig~\ref{fig:exist_1d1}), 
% \begin{figure}
%     \centering
%     \includegraphics[width=0.75\textwidth]{exist_1d1.eps}
%     \caption{Case i. d1}
%     \label{fig:exist_1d1}
% \end{figure}
we have $S(0)\supseteq [-b,1]\times [0, 1]$ and $S_5(\hat{x})\subset [-b-\hat{q}\hat{x},1]\times [0,1]$ (since $(\hat{q}a-\hat{a})\hat{x} \geq 0$). Therefore, for $\hat{q}\leq 0$, $v_5(\hat{x})-v(0)\leq 0$ and for $\hat{q}>0$,
\begin{align*}
&v_5(\hat{x})-v(0) \\
\leq &\max\left\{
\begin{array}{ll} 
\max & r - h(x, y)- v(0) \\
% \st & xy + ax+by+\hat{a}\hat{x} + \hat{q} \hat{x}y = c\\
\st& (x,y)\in S(0) \\
\textrm{ } \\
\end{array}\ , \ 
\begin{array}{ll} 
\max & (r  - h(x, y)) - (r - h(-b,y) ) \\
\st & xy + ax+by+\hat{a}\hat{x} + \hat{q} \hat{x}y = c\\
& (x,y)\in[0,1]^2\backslash S(0)
\end{array}
\right\}\\
\leq &\max\left\{0\ , \ 
\begin{array}{ll} 
\max & h(-b, y) - h(x, y) \\
\st & (x,y)\in[-b-\hat{q}\hat{x},-b]\times [0,1]
\end{array}
\right\}.
\end{align*}
Note that for $x\in [-b-\hat{q}\hat{x},-b]$, using concavity of $h$, we have that
\begin{align*}
h(x,y) &\geq \frac{x}{-b} h(-b, y) + \frac{-b-x}{-b} h(0,y) \\
&\geq \frac{-b-x}{-b} (h^m-h^M) + h(-b,y) \geq \frac{\hat{q}\hat{x}}{-b} (h^m-h^M) + h(-b,y).
\end{align*}
Thus
$
v_5(\hat{x})-v(0) \leq \max \left\{0, \frac{\hat{q}(h^m-h^M)}{b}\hat{x}\right\}
$
% \begin{align*}
% v_5(\hat{x})-v(0) & \leq \max \left\{0, \frac{\hat{q}(h^m-h^M)}{b}\hat{x}\right\}
% \end{align*}
and we obtain $l<\infty$ such that $v_5(\hat{x}) -v(0)\leq l\hat{x}$.
\label{caseid1}

\textit{(subcase 2)} If $-a = 1$ with $-b\in [0,1)$, this is the same as \ref{casei}\ref{caseid1}1 as we might perform $x\leftarrow 1-x$ together with $y\leftarrow 1-y$.

\item 
\textit{(subcase 1)}
If $-b = 0$ with $-a\in (0,1)$ (see Fig~\ref{fig:exist_1e1}), 
% \begin{figure}
%     \centering
%     \includegraphics[width=0.75\textwidth]{exist_1e1.eps}
%     \caption{Case i. e1}
%     \label{fig:exist_1e1}
% \end{figure}
we have $S(0)\supseteq ([0,1]\times [-a,1]) \cup (\{0\}\times [0,1])$ and $S_5(\hat{x})\subset ([0,1]\times [-a,1]) \cup ([0, -\hat{q}\hat{x}]\times [0,1])$.  For $\hat{q}\geq 0$ we have $S_5(\hat{x})\subset S(0)$ so $v_5(\hat{x})-v(0)\leq 0$. For $\hat{q}<0$, we have
\begin{align*}
 &   v_5(\hat{x})-v(0)\\ &\leq \max\left\{
\begin{array}{ll} 
\max & r - h(x, y)- v(0) \\
\st& (x,y)\in S(0) \\
\textrm{ } \\
\end{array}
\ , \
% \right.\\ 
% &\left.
\begin{array}{ll} 
\max & (r  - h(x, y)) - (r - h(0,y) ) \\
\st & xy + ax+by+\hat{a}\hat{x} + \hat{q} \hat{x}y = c\\
& (x,y)\in[0,1]^2\backslash S(0)\end{array}\right\}\\
&\leq \max\left
\{0 \ , \
\begin{array}{ll} 
\max & h(0, y) - h(x, y) \\
\st & (x,y)\in[0,-\hat{q}\hat{x}]\times [0,1]
\end{array}
\right\}.
\end{align*}
For $(x,y)\in[0,-\hat{q}\hat{x}]\times [0,1]$, using concavity of $h$, we write
\begin{eqnarray*}
h(x,y)\geq& (1-x)h(0,y)+ xh(1,y) &\geq h(0,y) - x(h(0,y) - h(1,y)) \\ \geq& h(0,y) - x(h^M - h^m)
&\geq h(0,y) + \hat{q}\hat{x}(h^M - h^m).
\end{eqnarray*}
Thus
$
v_5(\hat{x})-v(0) \leq \max\left\{0, \hat{q} (h^m-h^M)\hat{x}\right\}
$
and we get $l<\infty$ such that $v_5(\hat{x}) -v(0)\leq l\hat{x}$.
\label{caseie1}

\textit{(subcase 2)} 
If $-b=1$ with $-a\in (0,1)$, the argument is the same as for \ref{casei}\ref{caseie1}1 after performing $x\leftarrow 1-x$ and $y\leftarrow 1-y$.
\end{enumerate}
\item 
We next consider the case when $c+ab\neq 0$. 
%Due to the fact that only one part of the hyperbola (and its convex hull) can be feasible, we have $c+ab>0$. 
As discussed above, we assume $c + ab + (\hat{q}a - \hat{a})\hat{x} \geq 0$ for all $\hat{x} >0$ and sufficiently small, and thus $c + ab >0$. In addition, if $(x,y)\in S(\hat{x})$ for $\hat{x} >0$ and sufficiently small, we have $x>-b-\hat{q}\hat{x}, y>-a$ or $x<-b-\hat{q}\hat{x}, y<-a$ but not both. 
\label{caseii}

\begin{enumerate}[label={(\alph*)}]
\item In the case $x<-b-\hat{q}\hat{x}$, $y<-a$ for $(x,y) \in S(\hat{x})$, we denote $S'(0)=S(0)\cap \{(x,y)\,|\,y<-a, x<-b\}\subseteq S(0)$. 
Since $(0, 0)\in S'(0)$, we may assume $(1, 1)\notin S'(0)$, or $S'(0)\supseteq[0,1]^2$. 
Then clearly $S(\hat{x}) \subseteq S(0)$ and therefore $v_5(\hat{x}) \leq v(0) \leq 0.$
\label{caseiia}

%contrary to $S(\hat{x})\backslash S(0)\neq \emptyset$.

\textit{(Subcase 1)} If $S'(0) = \{(0,0)\}$ (see Fig~\ref{fig:exist_2a1}), 
% \begin{figure}
%     \centering
%     \includegraphics[width=0.75\textwidth]{exist_2a1.eps}
%     \caption{Case ii. a1}
%     \label{fig:exist_2a1}
% \end{figure}
we have $c = 0$ and $a, b<0$. Thus, for $\hat{x}$, we obtain a curve between
$
(0, -\frac{\hat{a}\hat{x}}{b+\hat{q}\hat{x}})$ and $(-\frac{\hat{a}\hat{x}}{a}, 0).
$
% while obviously
% $$
% \left(0, -\frac{\hat{a}\hat{x}}{b+\hat{q}\hat{x}}\right), \left(-\frac{\hat{a}\hat{x}}{a}, 0\right)\to (0,0)\quad, 0<\hat{x}\downarrow 0.
% $$
Thus, for any $(x,y)$ within the curve, from concavity of $h$, it is clear that
\begin{align*}
h(x,y)\geq\ & \min\left\{h(0,0), h\left(0, -\frac{\hat{a}\hat{x}}{b+\hat{q}\hat{x}}\right), h\left(-\frac{\hat{a}\hat{x}}{a}, 0\right), h\left(-\frac{\hat{a}\hat{x}}{a}, -\frac{\hat{a}\hat{x}}{b+\hat{q}\hat{x}}\right)\right\}\\
\geq\ &  h(0, 0) + \min\left\{0,-\frac{\hat{a}\hat{x}}{b+\hat{q}\hat{x}}(h(0,1)-h(0,0)), -\frac{\hat{a}\hat{x}}{a}(h(1,0)-h(0,0)),\right.\\ 
&\left. -\frac{\hat{a}\hat{x}}{b+\hat{q}\hat{x}}(h(0,1)-h(0,0))-\frac{\hat{a}\hat{x}}{a}(h(1,0)-h(0,0))\right\}\\
\geq\ & h(0,0)+\max\left\{0, -\frac{\hat{a}\hat{x}}{b+\hat{q}\hat{x}}, -\frac{\hat{a}\hat{x}}{a},  -\frac{\hat{a}\hat{x}}{b+\hat{q}\hat{x}}-\frac{\hat{a}\hat{x}}{a}\right\}(h^m-h^M)\\
=\ & h(0,0) + \xi\hat{x} + o(\hat{x}), \ \ \ \ (\hat{x}\downarrow 0)
% &= &h(0,0) + \xi\hat{x}, \ \ \ \ 
\end{align*}
where the second inequality uses the concavity of $h$ and the fact that for sufficiently small $\hat{x}$, $-\frac{\hat{a}\hat{x}}{b+\hat{q}\hat{x}} + (-\frac{\hat{a}\hat{x}}{a}) \leq 1$, and $\xi$ is a constant obtained by taking the derivative of the ``max" term times $h^m - h^M$ (since $a,b< 0$, the term is differentiable.) 
Thus, by setting $l = \xi$ we have $v_5(\hat{x})-v(0)\leq l\hat{x} + o(\hat{x})$ for all sufficiently small $\hat{x}$.

\textit{(Subcase 2)} If $S'(0)\backslash \{(0,0)\}\neq \emptyset$, but $(0,1),(1,0)\notin S'(0)$ (see Fig~\ref{fig:exist_2a2}). 
% \begin{figure}
%     \centering
%     \includegraphics[width=0.75\textwidth]{exist_2a2.eps}
%     \caption{Case ii. a2}
%     \label{fig:exist_2a2}
% \end{figure}
Then the curve for $\hat{x}=0$ is between $(0,\frac{c}{b})$ and $(\frac{c}{a}, 0)$. We find $\tilde{x}, \tilde{y}$ such that $(\tilde{x}, \frac{c}{2b}), (\frac{c}{2a}, \tilde{y})$ are within the curve. 
From the convex nature of one part of the hyperbola, it can be verified that
$\tilde{y}>\sfrac{c}{2b}$, $\tilde{x}>\sfrac{c}{2a}.$ 
Consider $x\in [0, \tilde{x}]$ 
% and $y\in[0, \tilde{y}]$,  
and denote $y'(x) = (c-ax)/(x+b)$, noting that $(x,y'(x))\in S(0)$ and 
\begin{eqnarray*}
    &&w_x(x, \hat{x}) := 
\max_y \{\ (r-h(x,y)) - v(0) \ | \ (x,y)\in S_5(\hat{x})\ \},
% \begin{array}{ll}\max_y & (r-h(x,y)) - v(0)\\
% \st&  (x,y)\in S_5(\hat{x}),
% \end{array}
\\
&\leq&\max\left\{
\begin{array}{ll}
\max_y & (r-h(x,y)) - v(0))\\
\st&  (x,y)\in S(0) \\
\textrm{ } \\
\end{array}
\ , \ 
% \right.\\
% && \left.    
\begin{array}{ll}
\max_y &  (r- h(x, y)) - (r-h(x, y'))\\
\st&  xy + ax+by+\hat{a}\hat{x} + \hat{q} \hat{x}y = c,\\
& (x,y)\in[0,1]^2\backslash S(0)
\end{array}\right\}\\
&= & \left\{
\begin{array}{lc}
\max\{0, h(x, y'(x)) - h(x, y'(x) +\Delta(x, \hat{x}))\}\  & \Delta(x, \hat{x}) \geq 0\\
0 & \Delta(x, \hat{x})<0 
\end{array}
\right.
\end{eqnarray*}
where
$$
\Delta(x, \hat{x}) := \frac{c-ax-\hat{a}\hat{x}}{x+b+\hat{q}\hat{x}}-\frac{c-ax}{x+b}=\frac{ax\hat{q} - \hat{a}x - \hat{a}b - c\hat{q}}{(x+b+\hat{q}\hat{x})(x+b)}\hat{x},
$$
while noting that if $\Delta(x, \hat{x}) <0$, then $(x, y'(x) + \Delta(x, \hat{x}))\in S(0)$. 

Note that $[0,\tilde{x}]\subseteq \textup{Proj}_x S'(0)$ and thus $x + b < 0$ for $x \in [0, \tilde{x}]$ and therefore for $0<\hat{x}<\hat{x}_0$ sufficiently small, we obtain 
$
\frac{\hat{q}\hat{x}}{x+b}\geq -\frac{1}{2}$ or  $2\frac{x+b+\hat{q}\hat{x}}{x+b}\geq 1
$
for any $x\in [0,\tilde{x}]$. Thus for $\Delta(x, \hat{x}) \geq 0 $,
$$
\Delta(x,\hat{x})\leq 2\frac{ax\hat{q} - \hat{a}x - \hat{a}b - c\hat{q}}{(x+b)^2}\hat{x}\leq \max_{x\in[0,\tilde{x}]}\left\{2\frac{ax\hat{q} - \hat{a}x - \hat{a}b - c\hat{q}}{(x+b)^2}\right\} \hat{x}:=l'_{5,x}\hat{x}
$$
while noting that continuity gives $l'_{5,x}<\infty$ independent of $x$, $\hat{x}$.
% Next observe that $\Delta(x, \hat{x})$ 
% is monotone on $x\in[0,\tilde{x}]$ for sufficiently small $\hat{x}$. %Thus, for $\Delta(x, \hat{x}) \geq 0$ for 

Now for the case of $\Delta(x, \hat{x}) \geq 0 $ we have that
\begin{eqnarray*}
&&h(x,y'(x)) - h(x,y'(x)+\Delta(x, \hat{x})) \\
&\leq  & \frac{\Delta(x, \hat{x})}{1-y'(x)}(h(x, y'(x))-h(x, 1))
\leq   \frac{\Delta(x, \hat{x})}{1-y'(x)}(h^M-h^m)\\
&\leq & (h^M-h^m)\max_{x\in[0,\tilde{x}]}\left\{\frac{1}{1-y'(x)}\right\} \max_{x\in[0,\tilde{x}]}\Delta(x, \hat{x})\\
&\leq& (h^M-h^m)\max_{x\in[0,\tilde{x}]}\left\{\frac{1}{1-y'(x)}\right\} l'_{5,x}\hat{x},
\end{eqnarray*}
Therefore, there exists $l_{5,x}<\infty$ independent of $x, \hat{x}$ such that $w_x(x, \hat{x}) \leq l_{5,x}\hat{x}+o(\hat{x})$.

A similar analysis of
$
w_y(y, \hat{x}) :=  \max_x \{r-h(x,y) - v(0)|x,y\in S(0)
\}
$
provides
$
w_y(y,\hat{x})\leq l_{5,y}\hat{x}+o(\hat{x})$ for  $\hat{x}\downarrow 0
$, 
where $l_{5,y}$ is a constant independent of $y\in[0,\tilde{y}]$ and $\hat{x}$.

Finally, we combine the results for $w_x$ and $w_y$.
Since they cover the whole curve with overlapping, we have $v_5(\hat{x})-v(0)\leq \max\{l_{5,x}, l_{5,y}\}\hat{x}+o(\hat{x})$ for $\hat{x}\downarrow 0$.

\textit{(Subcase 3)} If $(0,1)\in S'(0)$ but $(1,0)\notin S'(0)$ (see Fig~\ref{fig:exist_2a3}), 
% \begin{figure}
%     \centering
%     \includegraphics[width=0.75\textwidth]{exist_2a3.eps}
%     \caption{Case ii. a3}
%     \label{fig:exist_2a3}
% \end{figure}
we apply a similar analysis for $w_y(y, \hat{x})$ with $y \in [0,1]$ and obtain a constant $l_{5,y}$ independent of $y$ and $\hat{x}$. 
We thus get $v_5(\hat{x})-v(0)\leq l_{5,y}\hat{x}+o(\hat{x})$ for $\hat{x}\downarrow 0$.

\textit{(Subcase 4)} 
If $(1,0)\in S'(0)$ but $(0,1)\notin S'(0)$ (see Fig~\ref{fig:exist_2a4}), 
% \begin{figure}
%     \centering
%     \includegraphics[width=0.75\textwidth]{exist_2a4.eps}
%     \caption{Case ii. a4}
%     \label{fig:exist_2a4}
% \end{figure} 
we apply a similar analysis for $w_x(x, \hat{x})$ with $x \in [0,1]$ and similarly obtain a constant $l_{5,x}$ independent of $x$ and $\hat{x}$.

\textit{(Subcase 5)} 
If $(1,0),(0,1)\in S'(0)$ (see Fig~\ref{fig:exist_2a5}). 
% \begin{figure}
%     \centering
%     \includegraphics[width=0.75\textwidth]{exist_2a5.eps}
%     \caption{Case ii. a5}
%     \label{fig:exist_2a5}
% \end{figure} 
Then the curve for $\hat{x}=0$ is between $(1, \frac{c-a}{1+b})$ and $(\frac{c-b}{1+a},1)$. 
Similar to \ref{caseii}\ref{caseiia}2,
%\footnote{\red{JP: What is case ii.ab?}}
%\footnote{\orange{XY: fixed}}
we find $\tilde{x},\tilde{y}$ such that $(\tilde{x}, \frac{c-a+b+1}{2+2b})$, $(\frac{c-b+a+1}{2+2a}, \tilde{y})$ are within the curve. From the convex nature of one part of the hyperbola, we have
$
\tilde{y}>\frac{c-b+a+1}{2+2a}$ and $ \tilde{x}>\frac{c-a+b+1}{2+2b}.
$
Similar to \ref{caseii}\ref{caseiia}2, we consider $w_x(x, \hat{x})$ for $x\in [\frac{c-a+b+1}{2+2b}, 1]$ and $w_y(y, \hat{x})$ for $y\in [\frac{c-b+a+1}{2+2a}, 1]$. 
We obtain $l_{5,x}$ and $l_{5,y}$. 
Therefore, we write $v_5(\hat{x})-v(0)\leq \max\{l_{5,x},l_{5,y}\}\hat{x}+o(\hat{x})$ for $\hat{x}\downarrow 0$.

\item 
If for $\hat{x}\in (0,\hat{x}_0)$, and for $(x,y) \in S(\hat{x})$  $x>-b-\hat{q}\hat{x}, y>-a$, we denote $S'(0)=S(0)\cap \{y>-a, x>-b\}\subseteq S(0)$. 
If $(1,1)\in S'(0)$ while $(0,0)\notin S'(0)$, the proof is the same as \ref{caseii}\ref{caseiia} as we can perform $x\leftarrow 1-x$ with $y\leftarrow 1-y$.
%\footnote{\red{JP: What is case ii.1)?}}
%\footnote{\orange{XY: I believe it should be ii. a) now, fixed}}
% we apply the same procedure as ii. a) to get $l_{5,x}$ and/or $l_{5,y}$.
\end{enumerate}

Combining the discussions for \ref{casei} and \ref{caseii} shows that there exists $l_5<\infty$ such that $v_5(\hat{x})-v(0)\leq l_5\hat{x}+o(\hat{x})$ for $\hat{x}\downarrow 0$. \qed
\end{enumerate}
\end{proof}

\section{Proof of Proposition~\ref{lm:seqind}}
\label{section:seqind}
%\footnote{\red{JP: I changed the title; used to be "Valid inequality for minimal cover"}}
%In this section, we consider separable bilinear sets, as defined in Definition~\ref{def:separable}. 
%For this set, we will consider producing lifted inequalities by either fixing pairs of matched variables, or keeping them in the seed inequality. 
%We will then lifted fixed pairs of variables together, 
%We first prove in Lemma~\ref{lm:seqind} that when a subadditive upper approximation of the lifting function of the seed inequality can be obtained, then lifting can be performed in a way that is sequence independent in that the functional form given to a pair of lifted variables is not influenced by the functional form given to other pairs of lifted variables in the sequence.   
%Later, in the first subsection, we prove that the minimal cover must exist, or the problem is either infeasible or polyhedral. In the second subsection, we provide a tight restricted valid inequality for the minimal cover set. We then analyze the strength and significance of the provided valid inequality. 

\begin{proof}[Proposition~\ref{lm:seqind}]
Consider any feasible solution $(x',y')\in Q$. 
Then
\begin{align*}
   & \sum_{i\in J_0\cup J_1} \gamma_i (x'_i, y'_i)
    \geq\ \sum_{i\in J_0} \psi(a_ix'_iy'_i)+\sum_{i\in J_1} \psi(a_ix'_iy'_i-a_i)\\
    \geq\ & \psi\left(\sum_{i\in J_0}a_ix'_iy'_i + \sum_{i\in J_1}(a_ix'_iy'_i-a_i)\right)
    \geq\  \phi\left(\sum_{i\in J_0}a_ix'_iy'_i + \sum_{i\in J_1}a_ix'_iy'_i-\sum_{i\in J_1}a_i\right)\\
    =\ & \max_{(x_I,y_I)\in[0,1]^{2|I|}}\left\{r-h(x_I,y_I)\ \,\Bigm|\,\sum_{i\in I} a_ix_iy_i\geq d-\sum_{i\in J_0\cup J_1}a_ix'_iy'_i \right\}
    \geq\ r-h(x'_I, y'_I),
\end{align*}
%\footnote{\red{JP: In the max above, we should have $[0,1]^{2|I|}$, right?}}
%\footnote{\orange{XY: I think it should be either $(x_I, y_I)\in[0,1]^{2|I|}$ or $(x,y)\in[0,1]^{2n}$.}}
where the first inequality holds because of assumptions~\ref{item:seqind:3} and \ref{item:seqind:4}, 
the second inequality holds because assumption~\ref{item:seqind:2} requires $\psi(\cdot)$ to be subadditive over its range, 
the third inequality holds because assumption~\ref{item:seqind:1} requires $\psi(\cdot)$ to be an upper bound on $\phi(\cdot)$, the equality holds from the definition of $\phi(\cdot)$, and the last inequality is satisfied because $(x_I',y_I')$ is a feasible solution to the preceding optimization problem. 
\qed
\end{proof}

%The statement of Lemma~\ref{lm:seqind} does not specify the type of functional forms $\gamma_i(x_i',y_i')$ to use in ensuring that assumption 3 is satisfied. 
%It is however clear from the definition that choosing $\gamma_i(x,y)$ to be the concave envelope of $\psi(a_ixy)$ over $[0,1]^2$ when $i \in J_0$.

\section{Proof of Theorem~\ref{thm:minimal}}
\label{section:minimal}
%\footnote{\red{JP: I changed the title; used to be "Existence of minimal cover yielding partitions"}}
%\footnote{\red{JP: I suggest removing this transition sentence: ``We prove that minimal covers exist for most separable bilinear sets."}}
%\ThmMinimal*
\begin{proof}[Theorem~\ref{thm:minimal}]
Suppose $Q\neq\emptyset$. 
It follows from \cite{dey2019new} that the extreme points of $Q$ are such that $(x^*_j,y^*_j)\in\{ 0,1\}^2$ for all $j\in [n]\backslash\{i\}$ for some $i \in [n]$.

Assume first that $Q$ has an extreme point $(x^*,y^*)$ where   $x^*_iy^*_i\notin \{0,1\}$ for some $i \in [n]$ with $a_i>0$. 
Define the partition $\Lambda$ with
$I=\{i\}$, $J_0 = \{j|x^*_jy^*_j = 0\}$, and $J_1 = \{j|x^*_jy^*_j = 1\}$. 
Since $\sum_{j=1}^n a_jx^*_jy^*_j = d$, we have
$a_i>a_ix^*_iy^*_i = d-\sum_{i\in J_1}a_j = d^\Lambda>0$.
Since $d^\Lambda>0$, we conclude that $\Lambda$ is a minimal cover yielding partition.

Assume second that all extreme points $(x^*,y^*)$ are such that $(x^*_i, y^*_i)\in\{0,1\}^2$ for all $i \in [n]$ with $a_i>0$. 
Denote $I_+ = \{i \in [n] \,|\, a_i>0\}$ and, for $K_1, K_2\subseteq I_+$, define
$$
Q_{K_1, K_2}:=\conv
\left\{\ (x,y)\in[0,1]^{2n}
\ \left|\ \begin{array}{l}
\ \sum_{i=1}^n a_ix_iy_i\geq d\\
\ x_i = 0,\ i\in K_1\\
\ y_i = 0,\ i\in K_2\\
\ x_i = 1,\ i\in I_+\backslash K_1\\
\ y_i = 1,\ i\in I_+\backslash K_2
\end{array} \ \right.\right\}.
$$
It is clear that $\conv(Q) = \conv(\bigcup_{K_1, K_2\subseteq I_+} Q_{K_1, K_2})$.
Because, for any $K_1, K_2 \subseteq I_+$, $Q_{K_1, K_2}$ is a polytope~\cite[Proposition 17]{richard2010liftingframework}, we conclude that $\conv(Q)$ is a polytope.
%\footnote{\red{JP: For the second case, would it be easier to say in that case there is only a finite number of extreme points?}}
%\footnote{\orange{XY: We still need to discuss the negative $a_i$-s and corresponding $(x,y)$ pairs, so it will not be much easier.}}
\qed
\end{proof}

\section{Proof of Theorem~\ref{thm:valid}}
\label{section:valid}
%\footnote{\red{JP: I changed the title and promoted to its own section; used to be "A tight valid inequality for minimal covering sets"}}
%\footnote{\red{JP: I suggest removing this transition sentence: ``We now provide a tight valid inequality for the minimal covering set."}}
%\ThmValid*
\begin{proof}[Theorem~\ref{thm:valid}]
For $i \in [n]$, define
\begin{eqnarray*}
Q_i = \left\{ \ (x,y) \in [0,1]^{2n} \ \Biggm| \
\begin{array}{l}
(x_j,y_j)=(1,1), \,\forall j \in [n]\backslash i \\
\sqrt{a_i} \sqrt{x_iy_i} \ge \sqrt{d_i} 
\end{array} \
\right\}.
\end{eqnarray*}
First observe that, because $a_i$ for $i \in [n]$ form a minimal cover, we have that $a_i>d_i:=d-\sum_{j \neq i} a_j$ for each $i$.
This implies that sets $Q_i$ are nonempty. 
We next argue that $\conv(Q)=\conv(\bar{Q})$ where $\bar{Q}:=\bigcup_{i=1}^n Q_i$.
To this end, consider any extreme point $(x,y)$ of $Q$.
Then, \cite{dey2019new} shows that there exists a partition $(I_0,I_1,\{i\})$ of $[n]$ such that $x_jy_j=0$ for $j \in I_0$, $x_jy_j=1$ for $j \in I_1$ and $x_iy_i \in [0,1]$.
Because $a_i$ for $i \in [n]$ form a minimal cover, it must be that $|I_0|=0$ as otherwise $\sum_{j=1}^n a_jx_jy_j \le \sum_{j \in I_1} a_j < d$.
We conclude that $(x,y) \in Q_i$.
Since $Q$ is compact, it follows that $\conv(Q) \subseteq \conv(\bar{Q})$. 
Further, since $Q_i \subseteq \bar{Q} \subseteq Q$, it is clear that $\conv(\bar{Q}) \subseteq \conv(Q)$.

% It follows from the minimality assumption that all feasible solutions to $Q$ satisfy $x_iy_i>0$ for $i \in [n]$.
%It follows from the minimality assumption that all feasible solutions to $Q$ satisfy for at most one $i'\in [n]$, $x_{i'}y_{i'}=0$, and in this case we have $x_jy_j=1$ for all the remaining $j\in [n]\backslash{i'}$. 
%In addition from 
%Thus, the extreme points are $\conv(Q)$ are of the form:
%\begin{itemize}
%\item $x_i = y_i = 1$ for $i \in [n] \setminus \{i^0\}$,
%\item $a_{i^0}x_{i^0} y_{i^0} \geq d - \sum_{i \neq i^0} a_i = d_{i^0}$ (We note that $a_{i^0} \geq d_{i^0}$ by being a cover).
%\end{itemize} 
 
We now use disjunctive programming to obtain an extended formulation of $\conv(\bar{Q})$. 
This formulation introduces convex multipliers $\lambda_i$ and copies $(x^i,y^i)$ of variables $(x,y)$ for each disjunct $Q_i$.  
Because disjunct $Q_i$ yields constraints $y^i_j=x^i_j=\lambda_i$ for $j \neq i$, variables $x^i_j$ and $y^i_j$ for $j \neq i$ can be eliminated from the formulation in favor of $\lambda_i$. 
Renaming variables $x^i_i$ as $\hat{x}_i$, we obtain
\begin{eqnarray*}
\begin{array}{rcll}
x_j &=& \hat{x}_j + \sum_{i \neq j} \lambda_i  &\ \forall j \in [n] \\
y_j &=& \hat{y}_j + \sum_{i \neq j} \lambda_i &\ \forall j \in [n] \\
\sqrt{a_i}\sqrt{\hat{x}_i \hat{y}_i } &\geq& \sqrt{d_i} \lambda_i &\ \forall i \in [n]\\
\lambda_i &\geq& \hat{x}_i,\hat{y}_i \geq 0 &\ \forall i \in [n]\\
% 0 \leq \hat{y}_i &\leq &\lambda_i &\ \forall i \in [n]\\
%\hat{x}_i &\geq &0 &\ \forall i \in [n]\\
%\hat{y}_i &\geq &0 &\ \forall i \in [n]\\
\sum_{i = 1}^n \lambda_i &=& 1&\\
%---------------------------------------XYGU!!!!!
% \lambda_i &\geq & 0 &\  \forall i \in [n]\\
\end{array}
\end{eqnarray*}
because the constraint functions of each $Q_i$ are positively homogeneous. 

Using the fact that $\sum_{i \neq j} \lambda_i  = 1 - \lambda_ j$, we obtain
$\hat{x}_j = x_j -(1-\lambda_j)$ and
$\hat{y}_j = y_j -(1-\lambda_j)$.
Eliminating these variables from the formulation, we obtain
\begin{eqnarray}\label{eq:nomoreFM}
\begin{array}{rcll}
\sqrt{a_i}\sqrt{(x_i - (1 - \lambda_i))\cdot (y_i - (1 - \lambda_i)) } &\geq& \sqrt{d_i} \lambda_i &\ \forall i \in [n]\\
1-\lambda_i \le x_i,y_i &\leq &1  &\ \forall i \in [n]\\
% 1 - \lambda_i \le y_i &\leq &1 &\ \forall i \in [n]\\
%x_i &\geq &1 - \lambda_i &\ \forall i \in [n]\\
%y_i &\geq & 1 - \lambda_i &\ \forall i \in [n]\\
\sum_{i = 1}^n \lambda_i &=& 1.&\\
% \lambda_i &\geq & 0 &\  \forall i \in [n].
\end{array}
\end{eqnarray}
Because projecting variables $\lambda_i$ from the above formulation seems difficult, we relax the above set by using, for each $i \in [n]$ the following inequality 
\begin{eqnarray}
\sqrt{a_i} \left( \sqrt{x_i y_i}  - (1 - \lambda_i) \right) \geq \sqrt{a_i} \sqrt{(x_i - (1 - \lambda_i))\cdot (y_i - (1 - \lambda_i)) }, \label{eq:ag}
\end{eqnarray}
%\footnote{\red{JP: I added parentheses in the lhs of previous equation}}
which holds as
$\left(\sqrt{x_iy_i}-(1-\lambda_i)\right)^2 \ge  x_iy_i - (x_i+y_i)(1-\lambda_i) + (1-\lambda_i)^2 =  (x_i-(1-\lambda_i))(y_i-(1-\lambda_i))$
where the first inequality is obtained by expanding the square and using the arithmetic-geometry mean inequality $-2\sqrt{x_iy_i} \ge -(x_i+y_i)$.
%(that follows from the arithmetic-geometric inequality)
%\begin{eqnarray}
%\left[\Leftrightarrow -2\sqrt{x_iy_i}(1 - \lambda_i) \geq -(x_i + y_i)( 1- \lambda_i).\right] \nonumber
%\end{eqnarray}
Substituting (\ref{eq:ag}) in (\ref{eq:nomoreFM}), we obtain:
\begin{eqnarray}
\begin{array}{rcll}
\lambda_i &\geq & \frac{\sqrt{a_i}}{\sqrt{a_i} - \sqrt{d_i}}\left( 1 - \sqrt{x_iy_i}\right)&\ \forall i \in [n]\\
\lambda_i &\geq& 1 - x_i  &\ \forall i \in [n]\\
\lambda_i &\geq& 1 - y_i  &\ \forall i \in [n]\\
x_i,y_i &\leq &1  &\ \forall i \in [n]\\
%  &\leq &1 &\ \forall i \in [n]\\
\sum_{i = 1}^n \lambda_i &=& 1.&\\
% \lambda_i &\geq & 0 &\  \forall i \in [n].
\end{array}
\end{eqnarray}
% Inequalities $\lambda_i\ge 0$ are redundant since they are implied by $\lambda_i \ge 1-x_i$ and $x_i \le 1$. 
Using Fourier-Motzkin to project variables $\lambda_i$,   we obtain $(x, y) \in [0, 1]^{2n}$ together with
\begin{eqnarray*}
\sum_{i = 1}^n \textup{max}\left\{ \frac{\sqrt{a_i}}{\sqrt{a_i} - \sqrt{d_i}}\left( 1 - \sqrt{x_iy_i}\right), 1 - x_i, 1 - y_i \right\} \leq 1,
\end{eqnarray*}
which is a convex inequality. 
Retaining only the first term in the maximum for each pair $(x_i,y_i)$ and multiplying through by $-1$ yields the weaker convex inequality (\ref{eq:bilincoverineq}).
%begin{eqnarray*}
%$\sum_{i = 1}^n \frac{\sqrt{a_i}}{\sqrt{a_i} - \sqrt{d_i}}\left( \sqrt{x_iy_i}-1\right) \geq -1$.
% \red{GU: REMOVE this part?} which can be rewritten as (\ref{COVER_INEQUALITY}) after multiplying the numerators and denominators of the coefficient of each variable pair $(x_i,y_i)$ with $\sfrac{1}{\sqrt{a_i}}$ and by substituting $d_i$ with $a_i-\Delta$.
% %Leftrightarrow \sum_{i = 1}^n \frac{1}{ 1 - \sqrt{\frac{d_i}{a_i}}}\left( 1 - \sqrt{x_iy_i}\right) \leq 1 \\
%Leftrightarrow -\sum_{i = 1}^n \frac{1}{ 1 - \sqrt{1 - \frac{\Delta}{a_i}}}\left( \sqrt{x_iy_i} - 1\right) \leq 1.
%end{eqnarray*}
\qed
\end{proof}

\section{Proof of Theorem~\ref{thm:strcovercon} }
\label{section:strcovercon}
%\footnote{\red{JP: I changed the title of this section and promoted to its own section; title used to be "Strength of (\ref{eq:bilincoverineq})"}}

In this section, we provide a proof of Theorem~\ref{thm:strcovercon}.
We say that $G\in\Real^n$ is a \textit{set of the covering type} if whenever $\hat{x}\in G$, then $\tilde{x}\in G$ for all $\tilde{x} \in \Real^n$ such that $\tilde{x}\geq \hat{x}$. 
Due to lack of space we skip the proof of the next proposition; see~\cite{bodur2018aggregation} for a similar result.
\begin{proposition}
\label{prop:eqposcone}
Let $B=[0,1]^n$ and let $G$ and $H$ be sets of the covering type, such that $\conv(G\cap B)\subseteq H$. 
If there exists $\theta\geq 1$, such that for any $c \geq 0$, 
$ z^l\leq z^*\leq \theta z^l$,
where $z^*:=\min\{c\tr x|x\in G\cap B\}$ and $z^l:=\min\{c\tr x|x\in H\cap B\}$, then
$
(\theta H) \cap B \subseteq \conv(G\cap B).$
\end{proposition}

%The following theorem estimates the difference between the two feasible regions, and establishes the strength of the proposed seed inequality.
%\ThmStrength*
Following Proposition~\ref{prop:eqposcone}, 
%in the appendix, 
Theorem~\ref{thm:strcovercon} will be proven if, for all $(p, q) \in \Real^{2n}_+$,
\begin{align*}
    z^*&:=\min\left\{\ \sum_{i=1}^n (p_ix_i+q_iy_i) \ \Bigm| \ \sum_{i=1}^n a_ix_iy_i\geq d,\  (x,y) \in[0,1]^{2n} \ \right\}\\
    z^l&:=\min\left\{\ \sum_{i=1}^n (p_ix_i+q_iy_i) \ \Bigm| \ \sum_{i=1}^n \frac{\sqrt{a_i}}{\sqrt{a_i}-\sqrt{d_i}}(\sqrt{x_iy_i}-1)\geq -1,\  (x,y)\in[0,1]^{2n} \ \right\}
\end{align*}
satisfy $z^l\leq z^*\leq 4z^l$.
To this end, we prove first four ancillary results in Lemmas~\ref{lem:thetacompute}-\ref{lem:rangealpha}. 

%We proceed accordingly in this section and pose
\begin{assumption}\label{assp:1}
$p_i \geq q_i$, $\forall i \in [n]$.
\end{assumption}
Assumption~\ref{assp:1} is without loss of generality as it can always be achieved by renaming variables $x_i$ as $y_i$, if necessary.
%We present next some preliminary results. 

\begin{lemma}\label{lem:thetacompute}
For $\alpha \in [0, 1]$
%, $p_i \ge q_i \ge 0$
and $i \in [n]$, define
\begin{eqnarray*}
\theta_i(\alpha) = &\textup{min} \left\{ p_i x_i + q_i y_i \Bigm|  
\begin{array}{l}
\sqrt{x_i y_i } = \alpha\\
(x_i, y_i) \in [0,1]^2 \\
\end{array}
\right\}.
\end{eqnarray*}
Then, $\theta_i(\alpha)=0$ when $p_i=0$. 
Further, when $p_i>0$, 
\begin{eqnarray*}
\theta_i(\alpha) = \left\{\begin{array}{lcl} 
%0\ &\textup{when}& p_i=0\\
2\sqrt{p_iq_i}\cdot\alpha\ &\textup{when}& \alpha \leq \sqrt{\frac{q_i}{p_i}}\\  
p_i\cdot\alpha^2 + q_i\ &\textup{when}& \alpha \geq \sqrt{\frac{q_i}{p_i}}.
\end{array}\right.
\end{eqnarray*}
\end{lemma}
\begin{proof} 
When $p_i=0$, it follows from Assumption~\ref{assp:1} that $q_i=0$. 
The result holds trivially. 
For $p_i>0$,  setting $x_i = \alpha^2/y_i$, we write $\theta_i(\alpha)=\min \{ p_i \sfrac{\alpha^2}{y_i} +q_iy_i \,|\, \alpha^2 \le y_i \le 1 \}$, a problem with linear constraints and a convex objective over $\Real_+$.
When $q_i=0$, $y^*_i=1$ is optimal and the result follows as $x^*_i=\alpha^2$.
When $q_i>0$, the problem has $y^*_i =  \sqrt{\sfrac{p_i}{q_i}}\alpha \ge \alpha \ge \alpha^2$ as unique stationary point over $\Real_+$.
We conclude that $\bar{y}_i = \min\{y^*_i,1\}$ is optimal for the constrained problem. 
%This implies that $\bar{x}_i = \sfrac{\alpha^2}{\bar{y}_i}=\max\{\sfrac{\alpha^2}{y_i^*,\alpha^2}\}$.
%We conclude that $\theta_i(\alpha) = \max\{ \alpha \sqrt{p_iq_i}, p_i\alpha^2\} + \min\{\alpha \sqrt{p_iq_i},q_i \}$.
%Setting $x_i = \alpha^2/y_i$ and ignoring bounds on $x_i$ and $y_i$ we obtain that the KKT point is at $y_i = \sqrt{\frac{p_i}{q_i}}\alpha$ which is the optimal solution if $\alpha \geq \sqrt{\frac{q_i}{p_i}}$. Otherwise, we set $y_i = 1$ and $x_i = \alpha^2$.
\qed
\end{proof}

\begin{lemma}\label{lem:zoptstruc}
Let $\alpha^*_i := \sqrt{\frac{d_i}{a_i}}$. 
%Note that if $(x,y)$ is a feasible point for the bilinear set, then $\sqrt{x_iy_i} \geq \alpha^*_i \ \forall i\in [n]$. 
Then,
$z^{*} = \textup{min}_{i \in [n]} \left\{ \sum_{j \in [n]\setminus \{i\}}(p_j + q_j)  + \theta_i(\alpha^*_i) \right\}$.
\end{lemma}
\begin{proof} 
Since an optimal solution to the problem defining $z^*$ can always be chosen among the extreme points of $Q$ and since the proof of Theorem~\ref{thm:valid} in Section~\ref{section:valid} establishes that extreme points of $Q$ belong to $\bigcup_{i=1}^n Q_i$, we write that $z^*=\min_{i \in [n]} \min\{ p\tr x+q\tr y \,|\, (x,y) \in Q_i\}$. 
Points of $Q_i$ satisfy $x_j=y_j=1$ for $j \neq i$ and $a_ix_iy_i \ge d_i$. 
Since $p_i \ge q_i \ge 0$, it suffices to consider solutions that satisfy $\sqrt{x_iy_i}= \sqrt{\sfrac{d_i}{a_i}}=\alpha_i^*$ in the above problem, yielding the result. 
\qed
%Directly from extreme point structure of the bilinear set \cite{dey2019new}.
\end{proof}

Rearranging the variables if necessary, assume from now on that
$
z^* = \sum_{i \in [n-1]}(p_i + q_i) + \theta_n(\alpha^*_n).
$
As a consequence of this assumption and Lemma~\ref{lem:zoptstruc}, we obtain that 
\begin{eqnarray}\label{eq:uselaterpairwise}
\theta_j(\alpha^*_j) + p_n + q_n \geq \theta_n(\alpha^*_n) + p_j + q_j,  \quad \forall j \in [n].
\end{eqnarray}

\begin{lemma}\label{lem:termwiselower}
 Let $\tau_i(\alpha) = (p_i + q_i)\cdot\alpha^2$. 
 Then  $\tau_i(\alpha) \leq \theta_i(\alpha)$ for $\alpha \in [0, 1]$.
\end{lemma}
\begin{proof}
When $p_i=0$, the result is clear. Assume therefore that $p_i>0$.
When $\alpha \ge \sqrt{\sfrac{q_i}{p_i}}$, we write that $\theta_i(\alpha)=p_i \alpha^2 + q_i \ge p_i \alpha^2 + q_i \alpha^2 = \tau_i(\alpha)$, where the inequality holds because $\alpha \in [0,1]$.
When $\alpha \le \sqrt{\sfrac{q_i}{p_i}}$ (or equivalently $\sqrt{q_i} \ge \sqrt{p_i} \alpha)$, we write that
$\theta_i(\alpha)= 2\sqrt{q_i}\sqrt{p_i} \alpha \ge 2 p_i \alpha^2 \ge (p_i+q_i) \alpha^2$, where the last inequality holds because $p_i\ge q_i \ge 0$.
%Easy to verify using Lemma~\ref{lem:thetacompute}.
\end{proof}

\begin{lemma}\label{lem:rangealpha}
Assume that $(x,y)\in[0,1]^{2n}$ satisfies (\ref{eq:bilincoverineq}), \textit{i.e.,} $\sum_{i = 1}^n \frac{\sqrt{a_i}}{\sqrt{a_i} - \sqrt{d_i} } (\sqrt{x_iy_i} - 1 )\geq -1 $.
Define $\alpha_i = \sqrt{x_iy_i}$ for $i \in [n]$. 
Then 
% \begin{enumerate}[label={(\roman*)},align=left]
(i) $\alpha^*_i \leq \alpha_i$ for all $i\in [n]$, (ii) $\alpha_i < \frac{1}{2}$ for at most one $i \in [n]$.
% \item $\alpha^*_i \leq \alpha_i$ for all $i\in [n]$,
% \label{rangealpha:1}
% \item $\alpha_i \leq \frac{1}{2}$ for at most one $i$.
% \item $\alpha_i < \frac{1}{2}$ for at most one $i \in [n]$.
% \label{rangealpha:2}
% \end{enumerate}
\end{lemma}
\begin{proof} 
Statement~(i) trivially holds, as any $\sqrt{x_iy_i}<\alpha_i^*=\sqrt{\sfrac{d_i}{a_i}}$ invalidates (\ref{eq:bilincoverineq}), even if we set $x_j=y_j=1$ for $j\in[n]\backslash\{i\}$.
For (ii), assume by contradiction there exists distinct indices $i_1$ and $i_2$ in $[n]$ such that $\alpha_{i_1} \le \alpha_{i_2}<\frac{1}{2}$. 
Then
$
\sum_{i=1}^n \frac{\sqrt{a_i}}{\sqrt{a_i} - \sqrt{d_i} } (\sqrt{x_iy_i} - 1 ) 
< \sum_{i \in \{i_1,i_2\}} \frac{\sqrt{a_i}}{\sqrt{a_i} - \sqrt{d_i} } \left(-\frac{1}{2} \right)
\leq \sum_{i \in \{i_1,i_2\}} \left(-\frac{1}{2} \right) = -1,
$
which violates (\ref{eq:bilincoverineq}).
% From the validity of (\ref{eq:bilincoverineq}), if $(x,y)$ satisfy (\ref{eq:bilincoverineq}), then there exists $\lambda$ such that:
% \begin{eqnarray*}
% \alpha_j &=& \alpha^*_j \lambda_j + \sum_{i \in [n]\setminus\{j\}}\lambda_i \quad \forall j \in [n] \\
% % \sum_{i \in [n]} \lambda_i &=& 1\\
% % \lambda& \in& \mathbb{R}^n_{+}.
% \lambda &\geq& 0, \ \sum_{i=1}^n \lambda_i = 1.
% \end{eqnarray*}
% The first part follows trivially.  For the second part, it is sufficient to prove that if $\alpha_j \leq \frac{1}{2}$ for some $j\in [n]$, then $\alpha_i \geq \frac{1}{2}$ for all $i \in [n]\setminus \{j\}$.
% Assume therefore that $\alpha_j\le \frac{1}{2}$. 
% Then $\frac{1}{2} \geq \alpha_j = \alpha^*_j \lambda_j + \sum_{i \in [n]\setminus\{j\}}\lambda_i = \alpha^*_j \lambda_j + (1-\lambda_j)$.
% This implies that $ (1 - \alpha^*_j)\lambda_j \ge \frac{1}{2}$. 
% Further, as $\alpha_j^* \in (0,1)$, we conclude that $\lambda_j \ge \frac{1}{2}$.
% %\begin{eqnarray*}
% %\Rightarrow \frac{1}{2} &\geq& \alpha^*_j \lambda_j + (1 - \lambda_j)\\
% %\Rightarrow -\frac{1}{2}&\geq& ( - 1 + \alpha^*_j)\lambda_j\\
% %\Rightarrow \frac{1}{2}&\leq& (1 - \alpha^*_j)\lambda_j\\
% %\Rightarrow \frac{1}{2}&\leq& \lambda_j\\
% %\end{eqnarray*}
% Now since $\alpha_i \geq \lambda_j$ for all $i \in [n] \setminus \{j\}$, we have that $\alpha_i \geq \frac{1}{2}$.
\qed
\end{proof}

We are now ready to give a proof of Theorem~\ref{thm:strcovercon} that inequality (\ref{eq:bilincoverineq}) yields strong bounds for optimization problems over $Q$. 
\begin{proof}[Theorem~\ref{thm:strcovercon}]
Let $(\tilde{x},\tilde{y})$ be an optimal solution for the relaxation defining $z^l$ and let $\tilde{\alpha}_i = \sqrt{\tilde{x}_i\tilde{y}_i}$.
From Lemma~\ref{lem:rangealpha}, it is sufficient to consider the following three cases.
%\footnote{\red{JP: I think it is sufficient to consider two cases. The first is $\tilde{\alpha}_j \le \frac{1}{2}$ for some $j < n$. The second is 
%$\tilde{\alpha}_j \ge \frac{1}{2}$ for all $j < n$. Right?}}
%\footnote{\orange{XY: I think it might be needed to discuss $\tilde{\alpha}_n$}}

First assume that $\tilde{\alpha}_j \le \frac{1}{2}$ for some $j < n$.  
Lemma~\ref{lem:rangealpha} implies $\tilde{\alpha}_i \ge \frac{1}{2}$ for $i \neq j$.
Then 
\begin{eqnarray*}
4z^l &\ = \ & 4\sum_{i = 1}^n \theta_i(\tilde{\alpha}_i)  
% &\geq& 4\left(\sum_{i \in [n]\setminus \{j\}} \tau(\tilde{\alpha}_i) + \theta_j(\tilde{\alpha}_j)  \right) \quad (\textup{lemma}~\ref{lem:termwiselower})\\
\ \geq \ 4\left(\sum_{i \in [n]\setminus \{j\}} (p_i +q_i)\tilde{\alpha}_i^2 + \theta_j(\tilde{\alpha}_j)  \right) \\
%\quad (\textup{Definition of}~\tau)\\
&\geq& 4\left(\sum_{i \in [n]\setminus \{j\}} (p_i +q_i)\frac{1}{4} + \theta_j(\tilde{\alpha}_j)  \right) 
 \ = \ \sum_{i \in [n]\setminus \{j\}} (p_i +q_i) + 4\theta_j(\tilde{\alpha}_j) \\
 &\geq &\sum_{i \in [n]\setminus \{j,n\}} (p_i +q_i) + \theta_j(\alpha^*_j) + (p_n +q_n)\\
% & \geq& \sum_{i \in [n]\setminus \{j,n\}} (p_i +q_i) + \theta(\alpha^*_j) + (p_n +q_n)\quad (\alpha^*_j \leq \tilde{\alpha}_j, \theta~\textup{is monotonically increasing})\\
% &\geq& \sum_{i \in [n]\setminus \{j,n\}} (p_i +q_i) + p_j +q_j + \theta(\alpha^*_n)\quad (\ref{eq:uselaterpairwise}) \\
&\geq& \sum_{i \in [n]\setminus \{j,n\}} (p_i +q_i) + p_j +q_j + \theta_n(\alpha^*_n)\ = \ z^*,
\end{eqnarray*}
where the first inequality holds because of Lemma~\ref{lem:termwiselower}, 
the second inequality holds because $\alpha_i \geq \frac{1}{2}$ for $i \neq j$, 
the third inequality is because $\alpha^*_j \leq \tilde{\alpha}_j$ from Lemma~\ref{lem:rangealpha} and because $\theta_j$  is monotonically increasing, 
and the fourth inequality holds because of (\ref{eq:uselaterpairwise}).

Second assume that $\tilde{\alpha}_n \leq \frac{1}{2}$.
Lemma~\ref{lem:rangealpha} implies that $\tilde{\alpha}_i \ge \frac{1}{2}$ for $i < n$.
Similarly,
% \begin{eqnarray*}
% 4z^l &=& 4\sum_{i = 1}^n \theta_i(\tilde{\alpha}_i)  
% % &\geq& 4\left(\sum_{i \in [n -1]} \tau(\tilde{\alpha}_i) + \theta(\tilde{\alpha}_n)\right) \quad (\textup{lemma}~\ref{lem:termwiselower})\\
% \geq 4\left(\sum_{i = 1}^{n-1} (p_i +q_i)\tilde{\alpha}_i^2 + \theta_n(\tilde{\alpha}_n) \right)\\
% %\quad (\textup{Definition of}~\tau)\\
% &\geq& 4\left(\sum_{i =1}^{n-1} (p_i +q_i)\frac{1}{4} + \theta_n(\tilde{\alpha}_n)  \right) 
% %\quad (\alpha_i \geq \frac{1}{2} \forall i \neq n)\\
%  = \sum_{i=1}^{n-1} (p_i +q_i) + 4\theta_n(\tilde{\alpha}_n) \\
%  &\geq& \sum_{i=1}^{n-1} (p_i +q_i) + \theta_n(\alpha^*_n) 
% %\quad (\alpha^*_n \leq \tilde{\alpha}_n, \theta~\textup{monotonically increasing})\\
%  = z^*.
% \end{eqnarray*}
$
4z^l = 4\sum_{i = 1}^n \theta_i(\tilde{\alpha}_i)  
\geq 4(\sum_{i = 1}^{n-1} (p_i +q_i)\tilde{\alpha}_i^2 + \theta_n(\tilde{\alpha}_n) )
\geq 4(\sum_{i =1}^{n-1} (p_i +q_i)\frac{1}{4} + \theta_n(\tilde{\alpha}_n)  ) 
 = \sum_{i=1}^{n-1} (p_i +q_i) + 4\theta_n(\tilde{\alpha}_n)
 \geq \sum_{i=1}^{n-1} (p_i +q_i) + \theta_n(\alpha^*_n) 
 = z^*.
$

Finally assume that $\tilde{\alpha}_i \geq \frac{1}{2}$ for all $i$, and we use the same proof as just given.
%\footnote{\red{JP: Should we add a statement about the appendix, since the statement of the result is about sets, and not optimal values?}}
\qed
\end{proof}

%\subsection{TODO: REASON for the Valid Inequality}
%\input{Sec_Lifting_NEW.tex}
\section{Proof of Theorem~\ref{thm:upperbound}}
\label{section:upperbound}
%\footnote{\red{JP: the title used to be ``Lifted minimal covering inequalities (\ref{eq:bilincoverineq})"}}
%In this section, we consider bipartite bilinear covering sets. 
%We fix variable pairs so as to obtain a restriction of the set that is a minimal covering set. 
%Although the lifting function of the minimal covering set seems hard to compute, 
%\footnote{\red{JP: IS THIS TRUE? DO WE HAVE A NP-COMPLETE PROOF FOR THIS?}} 
%we construct a subadditive upper bound.
%This function can be used to lift the seed inequality in closed-form
%in a way that does not depend on the sequence in which variable pairs are reintroduced. 
%\subsection{Upper bound on the lifting function}
In this section, we provide a proof of Theorem~\ref{thm:upperbound}, which gives 
a subadditive over-approximation to the lifting function of the minimal covering inequality. We first pose
\begin{assumption}\label{assp:2}
$0<\Delta \le a_1 \le a_2 \le \ldots \le a_n$.
\end{assumption}

Assumption~\ref{assp:2} can always be achieved by reordering the variables since the notion of minimal cover requires that $a_i \ge \Delta$ for $i \in [n]$; see discussion following Notation~\ref{not:1}.

We next present ancillary results in Lemmas~\ref{lm:decreasing}-\ref{lm:subadditive} and Proposition~\ref{prop:minimal-slope} that are used in the derivation of the approximation of the lifting function. 
The proof of Lemma~\ref{lm:decreasing} is straightforward and can be obtained by investigating signs of derivatives. 
%\red{The following lemma is straightforward to verify.}
\begin{lemma}\label{lm:decreasing}
For $u\geq \max\{\alpha, \beta\}$ where $\alpha, \beta >0$, the function $f(u):= \frac{\sqrt{u}-\sqrt{u-\alpha}}{\sqrt{u}-\sqrt{u-\beta}}$ is decreasing when $\alpha >\beta$ and increasing when $\alpha < \beta$. 
\end{lemma}
% \begin{proof}
% First note that $f(u)=\frac{1}{\beta}(\sqrt{u}-\sqrt{u-\alpha})(\sqrt{u}+\sqrt{u-\beta})$. 
% It follows that $f(u) \ge 0$ since we assumed that $u \ge \max\{\alpha, \beta\}$. 
% Next we compute that
% \begin{align*}
% %\frac{\Diff f}{\Diff u} 
% &f'(u)\\
% =\ & \frac{1}{2\beta}\left[\left(\frac{1}{\sqrt{u}}-\frac{1}{\sqrt{u-\alpha}}\right)(\sqrt{u}+\sqrt{u-\beta})+(\sqrt{u}-\sqrt{u-\alpha})\left(\frac{1}{\sqrt{u}}+\frac{1}{\sqrt{u-\beta}}\right)\right]\\
% =\ &\frac{1}{2\beta}(\sqrt{u}-\sqrt{u-\alpha})(\sqrt{u}+\sqrt{u-\beta})\left(\frac{1}{\sqrt{u}\sqrt{u-\beta}}-\frac{1}{\sqrt{u}\sqrt{u-\alpha}}\right)\\
% =\ &\frac{f(u)}{2\sqrt{u}}\left(\frac{1}{\sqrt{u-\beta}}-\frac{1}{\sqrt{u-\alpha}}\right).
% \end{align*}
% Thus, for $u \ge \max\{\alpha, \beta\}$, $f'(u)\geq 0$ if $\beta >\alpha$ and $f'(u) \leq 0$ if $\beta < \alpha$.
% \qed
% \end{proof}

Lemma~\ref{lm:fractional-concave} establishes that the lifting function $\phi(\delta)$ exhibits local convexity.
\begin{lemma}
	\label{lm:fractional-concave}
	Any point $\delta$ of the lifting function $\phi(\delta)$ corresponding to an optimal solution $(x,y)$ with at least one index $i$ such that 
	%pair of $(x_i, y_i)$ such that 
	$x_iy_i\in(0,1)$,
	% variable
% 	$x_iy_i$ not at boundary, 
%	\footnote{\red{JP: Should we just say "one pair $(x_i,y_i)$ such that $x_iy_i \in (0,1)$"?}}
%	\footnote{orange{XY: changed}}
	is locally convex, \textit{i.e.}, there exists 
	%\orange{\sout{$\xi$ and  a neighborhood $(\delta-r, \delta+r)$ with }}
	$r>0$ {and $\xi$} such that
	 $\phi(\delta+\eta)\geq \phi(\delta)+\xi\eta$ for all $\eta \in [-r,r]$.
\end{lemma}
\begin{proof}
Let $\dot{\delta}$ be a point for which an optimal solution ($\dot{x},\dot{y}$) to the problem defining $\phi(\dot{\delta})$ is such that $\dot{x}_i\dot{y}_i\in(0,1)$. 
Define $r=\min\{\dot{x}_i\dot{y}_i,1-\dot{x}_i\dot{y}_i\}/2>0$.
Consider $\eta\in[-r,r]$ and construct $(x,y)$ so that $x_j=\dot{x}_j$, $y_j=\dot{y}_j$ for any $j\neq i$ and $x_iy_i = \dot{x}_i\dot{y}_i - \eta$. 
From the feasibility of $(\dot{x},\dot{y})$ for $\dot{\delta}$, we conclude that $(x,y)$ is a feasible solution to the optimization problem defining $\phi(\dot{\delta}+\eta)$. 
Therefore,
%\footnote{\red{JP: What about doing it this way instead of the blue?}}
%\footnote{\orange{XY: agree}}
\begin{eqnarray*}
\phi(\dot{\delta}+\eta) - \phi(\dot{\delta}) &\ge&
\frac{\sqrt{a_i}}{\sqrt{a_i}-\sqrt{d_i}} \left(\sqrt{\dot{x}_i\dot{y}_i}-\sqrt{\dot{x}_i\dot{y}_i-\eta}\right) \\
&=& \frac{a_i+\sqrt{a_id_i}}{a_i-d_i} \left(\sqrt{\dot{x}_i\dot{y}_i}-\sqrt{\dot{x}_i\dot{y}_i-\eta}\right) \ge \frac{a_i+\sqrt{a_id_i}}{2\Delta\sqrt{\dot{x}_i\dot{y}_i}}\eta,
\end{eqnarray*}
where the last inequality holds {because $a_i-d_i=\Delta$} and because the concavity of the square root function over $\Real_+$ implies that $\sqrt{\dot{x}_i\dot{y}_i-\eta} \le \sqrt{\dot{x}_i\dot{y}_i}-\frac{1}{2\sqrt{\dot{x}_i\dot{y}_i}}\eta$.
\qed
\end{proof}

To obtain the tightest linear over-approximation of $\phi(\delta)$ for $\delta\in(0,\infty)$,
%\footnote{\red{JP: Is the notation $\Real_{++}$ defined?}}
%\footnote{\orange{XY: proposed deletion and change}}
we next narrow down the set of points $\delta$ where function $\sfrac{\phi(\delta)}{\delta}$ can achieve a local maximum. 
\begin{proposition}
	\label{prop:minimal-slope}
% \footnote{\red{JP: I REPHRASED THE STATEMENT OF THIS PROPOSITION. THE OLD VERSION WAS} For any non-zero local maximum of the function $\sfrac{\phi(\delta)}{\delta}$ with $\delta>0$, all of its optimal solution $(x,y)$ must have 
% 	% all variables
% 	all $x_iy_i$-s are
% 	at boundary, unless it is within a local plateau.}
Assume that $\dot{\delta}>0$ is a local maximizer of the function $\sfrac{\phi(\delta)}{\delta}$ and that $(\dot{x},\dot{y})$ is an optimal solution to the problem defining $\phi(\dot{\delta})$. 
Then either
\begin{enumerate}[label={(\roman*)},align=left]
\item all $(\dot{x}_i,\dot{y}_i)$ pairs belong to $\{0,1\}^2$, or \label{minimal-slope:1}
\item there exists $r>0$ such that $\sfrac{\phi(\delta+\eta)}{(\delta+\eta)}=\sfrac{\phi(\delta)}{\delta}$ for all $\eta \in (-r,r)$.
\label{minimal-slope:2}
\end{enumerate}
\end{proposition}
\begin{proof}
%Suppose a point $\delta^0>0$ corresponding to a local maximum with $x^0_iy^0_i\in(0,1)$ to $\phi(\delta^0)$, we are going to show that $\delta^0$ must be within a local plateau.
Assume that \ref{minimal-slope:1} does not hold, \textit{i.e.}, there exists $i \in [n]$ for which $\dot{x}_i\dot{y}_i \in (0,1)$. 
We show that \ref{minimal-slope:2} holds.  
From Lemma~\ref{lm:fractional-concave}, there exists $\xi$ and $r>0$ such that $\phi(\dot{\delta}+\eta)\geq \phi(\dot{\delta})+\xi \eta$ for $\eta\in(-r, r)$. Without loss of generality, we assume $r<\dot{\delta}$.
We consider two cases. 
Assume first that $\xi\geq \sfrac{\phi(\dot{\delta})}{\dot{\delta}}$.
For any $\eta\in(0,r)$ we have 
	$\phi(\dot{\delta}+\eta)\geq\phi(\dot{\delta})+\xi\eta \geq \frac{\phi(\dot{\delta})}{\dot{\delta}}(\dot{\delta}+\eta)$
	or equivalently $\phi(\dot{\delta}+\eta)/(\dot{\delta}+\eta)\geq \phi(\dot{\delta})/\dot{\delta}$.
Assume second that $\xi\leq \sfrac{\phi(\dot{\delta})}{\dot{\delta}}$. For any $\eta\in(-r,0)$ we have 
	$\phi(\dot{\delta}+\eta)\geq\phi(\dot{\delta})+\xi\eta \geq \frac{\phi(\dot{\delta})}{\dot{\delta}}(\dot{\delta}+\eta)$
	or equivalently $\phi(\dot{\delta}+\eta)/(\dot{\delta}+\eta)\geq \phi(\dot{\delta})/\dot{\delta}$.
From analyzing these cases, we see that $\dot{\delta}$ can be a local maximum only if $\eta = \phi(\dot{\delta})/\dot{\delta}$ and all points in $(\delta-r, \delta+r)$ are also local maxima. 
%, \textit{i.e.}, $\dot{\delta}$ is within a local plateau.
\qed
\end{proof}

We now derive a linear over-approximation to the function $\phi(\delta)$ for $\delta \ge 0$.

%\footnote{JP: I introduced a latex def for i0 so that it can be rechanged easily if preferred.}
\begin{lemma}\label{lm:boundpos}
%\footnote{I move to remove what follows, especially if Notation~1 is used: Let $\Delta %:=  \sum_{i = 1}^n a_i - d$.
%Assume that variables are sorted so that $\Delta\leq a_1 \leq a_2 \leq a_3 \leq \dots %\leq a_n$. 
%Let $a_{\i0} = \min \{a_i\ | \ a_i>\Delta\}$ if it exists.} 
Define $l_+:=\frac{\sqrt{a_{\i0}}+\sqrt{d_{\i0}}}{\Delta \sqrt{d_{\i0}}}$
if $I^> \neq \emptyset$
%\footnote{Replaced this: if $a_{\i0}$ exists} 
and $l_+:=\frac{1}{\Delta}$ otherwise.
Then $\phi(\delta) \le l_+ \delta$ for $\delta \ge 0$.
%For any $\delta \geq 0$, we have that 
%\begin{itemize}
%\item if $a_{i_0}$ is the smallest among $\{a_i\}$ with $a_{i_0}>\Delta$, then
%$$\phi(\delta)\leq  \frac{1}{\sqrt{d_{i_0}}(\sqrt{a_{i_0}} - \sqrt{d_{i_0}}) }\delta=\frac{\sqrt{a_{i_0}}+\sqrt{d_{i_0}}}{\Delta\sqrt{d_{i_0}}}\delta. $$
%\item if such $a_{i_0}$ does not exist, i.e. $a_i=\Delta$ for all $i\in [n]$, then
%$$\phi(\delta)\leq \frac{1}{\Delta}\delta. $$
%\end{itemize}
%Further, coefficient $l_+$  is best possible.
\end{lemma}
\begin{proof}
The result holds trivially for $\delta = 0$ since $\phi(\delta)=0$. 
Our main tool to prove this result is Proposition~\ref{prop:minimal-slope} which will allow us to verify the value of $\phi(\delta)/\delta$ only for a finite set of values of $\delta$. 
However, since Proposition~\ref{prop:minimal-slope} holds only for $\delta>0$, we first prove the result in an interval that has $0$ as an end point.

%as the minimal covering inequality defines a nonempty face of the corresponding set.
%\footnote{\red{JP: Face would typically say intersection with a linear face ... Not sure it is the case here...}}
	
%We first observe that \orange{\sout{function
%$g(a):= 
%\frac{\sqrt{a}+\sqrt{a-\Delta}}{\sqrt{a-\Delta}} 
%= 1+\sqrt{1+\frac{\Delta}{a-\Delta}}
%$
%is a decreasing function of $a$ when $a>\Delta$. }}
%Definition~\ref{def:minimal} implies that $\Delta>0$,
%% 	\footnote{\red{JP: probably need also $\Delta>0$}}
%% 	\footnote{\color{blue}{$\Delta>0$ is guaranteed from minimal cover}}
%\orange{and, in the second case, implies $n\geq 2$.}
%% there exists at least two $a_i$ so that $\sum a_i>d>0$.
%\footnote{\red{JP: sum over what?}}
%\footnote{\red{JP: not sure I see the point of this paragraph. Remove?}}

As mentioned above, the first part of the proof investigates the function $\phi$ in a neighborhood of the point $\delta = 0$.
There are two cases to consider. 

For the first case, assume that $I^> \neq \emptyset$.
Consider $\delta \in [0, \min\{a_{\i0}-\Delta, \Delta\}/2]$.
Because the problem defining $\phi$ consists of maximizing a convex function, optimal solutions can be found at extreme points of the feasible region. 
It follows that there exists an optimal solution that is such that $x_i^*y_i^* \in \{0,1\}$ for all $i\in[n]\backslash\{j\}$ for some $j \in [n]$. 

Further, at most one index $k \in [n]\backslash\{j\}$ can be such that $x_k^*y_k^* = 0$ as otherwise $\sum_{i=1}^n a_i x_i^*y_i^* \leq \sum_{i = 1}^n a_i - 2\Delta = d-\Delta < d-\delta$ which would made this solution infeasible for the problem defining $\phi$. 

Also, if there exists $k$ with $x_k^*y_k^* = 0$, then $a_k=\Delta$. If not, $a_k\geq a_{\i0}$ and thus $\sum a_ix_i^*y_i^*\leq d +\Delta - a_{\i0} < d-\delta$, infeasible.
%\footnote{\red{JP: not sure I see why $a_k=\Delta$}}
%\footnote{\orange{XY: proposed change}}
{Thus, $-\Delta-\delta = \sum_{i = 1}^n a_ix_i^*y_i^* - \sum_{i = 1}^n a_i = a_jx_j^*y_j^* - a_j - a_k$, \textit{i.e.}, $a_jx_j^*y_j^* = a_j-\delta$, and as} $\frac{\sqrt{a_k}}{\sqrt{a_k}-\sqrt{d_k}}= \frac{\sqrt{\Delta}}{\sqrt{\Delta}  - 0} = 1$, we obtain
$$
\phi(\delta) 
= \frac{\sqrt{a_j}}{\sqrt{a_j}-\sqrt{d_j}}(1-\sqrt{x_j^*y_j^*}) 
= \frac{\sqrt{a_j}-\sqrt{a_j-\delta}}{\sqrt{a_j}-\sqrt{a_j-\Delta}} 
\leq \frac{\sqrt{a_n}-\sqrt{a_n-\delta}}{\sqrt{a_n}-\sqrt{a_n-\Delta}}:= \eta(\delta),
$$
% \orange{where the first inequality follows from Lemma~\ref{lm:decreasing} and the second inequality is $\sqrt{a_n-\delta}\geq \sqrt{a_n} - (\sqrt{a_n}-\sqrt{a_n-\Delta})\delta/\Delta$ from the concavity of $\sqrt{a_n-\delta}$ on $[0,\Delta]$.}
where the last step follows from Lemma~\ref{lm:decreasing}. 
% \orange{$\sqrt{a_n-\delta}$ is concave for $\delta\in[0,\Delta]$, and thus by taking a linear inequality tight at $0$ and $\Delta$, we obtain that $\sqrt{a_n-\delta}\geq \sqrt{a_n} - (\sqrt{a_n}-\sqrt{a_n-\Delta})\delta/\Delta$, i.e. $\phi(\delta) \leq \delta/\Delta$.}
% The function $\phi(\delta)$ is convex in $\delta$. 
{Note that $\eta$ is well-defined and convex on $[0,\Delta]$.}
% Therefore, by taking a linear inequality tight at $\delta = 0$ and $\delta = \Delta$, we obtain that $\phi(\delta) \leq \delta/\Delta$.
%Therefore, by taking a linear inequality tight at {$0$ and $\Delta$}, we obtain that
Therefore, it is easy to verify that $\eta(\delta)\leq \delta/\Delta$ for $\delta\in[0,\Delta]$, and thus $\phi(\delta) \leq \delta/\Delta \leq l_+ \delta$ for $\delta \in [0, \min\{a_{\i0}-\Delta, \Delta\}/2]$.
%\footnote{\red{JP: The display is given under the assumption that $\delta \le \Delta/2$. Why can we use it at $\delta=\Delta$?}}
    
	If there is no $k$ with $x^*_ky^*_k = 0$, we can verify $a_j>\Delta$ and $a_jx_j^*y_j^* = a_j-\Delta-\delta$. Thus,
	\begin{eqnarray*}
	\phi(\delta) &=& \frac{\sqrt{a_j}}{\sqrt{a_j}-\sqrt{d_j}}(1-\sqrt{x_j^*y_j^*})-1 \\
	&=& \frac{\sqrt{a_j}-\sqrt{a_j-\Delta-\delta}}{\sqrt{a_j}-\sqrt{a_j-\Delta}}-1
	\leq \frac{\sqrt{a_{\i0}}-\sqrt{a_{\i0}-\Delta - \delta}}{\sqrt{a_{\i0}}-\sqrt{a_{\i0}-\Delta}}-1:=\xi(\delta),
	\end{eqnarray*}
	where the last step follows Lemma~\ref{lm:decreasing}. 
	The function $\xi(\delta)$ is again convex. Therefore, it is easy to verify that $\xi(\delta)\leq \frac{\sqrt{a_{\i0}}+\sqrt{d_{\i0}}}{\Delta\sqrt{d_{\i0}}}\delta$ for $\delta\in[0, a_{\i0}-\Delta]$. Thus we obtain that $\phi(\delta) \leq l_+ \delta$ for $\delta \in [0, \min\{a_{\i0}-\Delta, \Delta\}/2]$.
%	$$
%	\phi(\delta)\leq \frac{\sqrt{a_{\i0}}+\sqrt{d_{\i0}}}{\Delta\sqrt{d_{\i0}}}\delta.
%	$$
	
%	Combining the discussions, for the first case where there exists such $a_{\i0}$, the inequality holds for $[0, \min\{a_{\i0}-\Delta, \Delta\}/2]$.
	
	For the second case $I^>=\emptyset$, \textit{i.e.},  $a_1=\ldots=a_n=\Delta$. Note that in this case $n \geq2$. 
	Consider $\delta\in[0, \Delta/2]$. 
	Similar to above, there exists an optimal solution that is such that $x_i^*y_i^* \in \{0,1\}$ for all $i\in[n]\backslash\{j\}$ for some $j\in [n]$. 
	In addition, there exists exactly one index $k \in [n]\backslash \{j\}$ with $x_k^*y_k^* = 0$, or otherwise we obtain $a_j>\Delta$ a contradiction to $I^>=\emptyset$. 
	As $\frac{\sqrt{a_k}}{\sqrt{a_k}-\sqrt{d_k}}=1$ and $a_jx_j^*y_j^*=\Delta x_j^*y_j^*=\Delta-\delta$, we obtain
	$$
	\phi(\delta) = \frac{\sqrt{a_j}}{\sqrt{a_j}-\sqrt{d_j}}(1-\sqrt{x_j^*y_j^*}) = \frac{\sqrt{\Delta}-\sqrt{\Delta-\delta}}{\sqrt{\Delta}} \leq \frac{\delta}{\Delta} \leq l_+ \delta, \ \delta \in [0, \Delta/2].
	$$
%	\textit{i.e.}, for the second case the inequality holds for $[0, \Delta/2]$.

% \green{
%     We continue to rewrite the inequality for $\delta\geq \min\{a_{i_0}-\Delta, \Delta\}/2$ in the first case as
%     $$
%     \frac{\phi(\delta)}{\delta} \leq \frac{\sqrt{a_{\i0}}+\sqrt{d_{\i0}}}{\Delta \sqrt{d_{\i0}}}
%     $$
%     and for $\delta\geq \Delta/2$ in the second case as
%     $
%      \phi(\delta)/\delta \leq 1/\Delta.
%     $
% }
% \footnote{\green{JP: I have no idea what the above paragraph is meant to say}}
% \footnote{\orange{XY: Propose deletion and following changes}}

The second part of the proof investigates the function $\phi$ away from the origin. 
As we are attempting to show that $\phi(\delta)/\delta$ bounded from above by $l_+$, it is sufficient to consider all local maximas of $\phi(\delta)/\delta$.
It follows from Proposition~\ref{prop:minimal-slope} that it is sufficient to verify the condition at 
%that all the local maximas of $\phi(\delta)/\delta$ 
values of $\delta$ such that $x_iy_i \in \{0,1\}$ for $i \in [n]$. 
(This is because, at other local maximas, the function $\phi(\delta)/\delta$ is locally constant and so it is sufficient to check at the end points of these ``constant intervals" where $x_iy_i \in \{0,1\}$.)
%\footnote{\red{JP: Why is that? Can we add a bit more of a discussion?}}
Any such local maximum $\delta$ is therefore such that there exists $S \subseteq [n]$ with $x_iy_i=0$ for $i \in S$ and $x_iy_i=1$ for $i \notin S$.  
We denote it as $\delta^S$. 
%For any of such local maximum, we denote it by $\delta^S$, corresponding to the $S\subset\{1,...,n\}$, where the optimal $(x,y)$ satisfies that for any $i\in S$, $x_iy_i = 0$ and for any $i\notin S$, $x_iy_i = 1$.
It is easily verified that $\delta^S = \sum_{i\in S}a_i - \Delta$. Let $S = \{i_1, i_2,\ldots,i_k\}$ such that $a_{i_1}\leq \ldots \leq a_{i_k}$, and we have
$$
\phi(\delta^S) = -1+\sum_{i\in S} \frac{a_i+\sqrt{a_id_i}}{\Delta} = \sum^{k-1}_{j=1}\frac{a_{i_j}+\sqrt{a_{i_j}d_{i_j}}}{\Delta} + \frac{d_{i_k}+\sqrt{a_{i_k}d_{i_k}}}{\Delta}.
$$
% and that
% $$\phi(\delta^S) 
% = -1 + \sum_{i \in S} \frac{\sqrt{a_i}}{\sqrt{a_i}-\sqrt{d_i}} 
% = -1+\sum_{i\in S} \frac{a_i+\sqrt{a_id_i}}{\Delta}.$$
% Let $S = \{i_1, i_2,\ldots,i_k\}$ be so that $a_{i_1}\leq \ldots \leq a_{i_k}$. 
% Since $d_{i_k}=a_{i_k} - \Delta$, we write
% %and we have $\delta^S = \sum^{k-1}_{j=1}a_{i_j} + d_{i_k},$
% %	along with
% \begin{align*}
% \phi(\delta^S) 
% %&= \sum^{k-1}_{j=1}\frac{a_{i_j}+\sqrt{a_{i_j}d_{i_j}}}{\Delta}+\frac{a_{i_k}+\sqrt{a_{i_k}d_{i_k}}}{\Delta}-1 \\
% &= \sum^{k-1}_{j=1}\frac{a_{i_j}+\sqrt{a_{i_j}d_{i_j}}}{\Delta} + \frac{d_{i_k}+\sqrt{a_{i_k}d_{i_k}}}{\Delta}.
% \end{align*}
Consider two cases.    
%Now consider whether we have $a_{i_k}>\Delta$. 
On the one hand, if $a_{i_k}> \Delta$, then $I^> \neq \emptyset$ and $a_{\i0}\leq a_{i_k}$. Thus,
	\begin{align*}
	\phi(\delta^S) &=   \sum_{j=1}^{k-1} \frac{\sqrt{a_{i_j}}+\sqrt{d_{i_j}}}{\Delta\sqrt{a_{i_j}}} a_{i_j} +\frac{\sqrt{d_{i_k}}+\sqrt{a_{i_k}}}{\Delta\sqrt{d_{i_k}}}d_{i_k}\\
	&\leq  \sum_{j=1}^{k-1} \frac{\sqrt{a_{i_0}}+\sqrt{d_{i_0}}}{\Delta\sqrt{d_{i_0}}} a_{i_j} +\frac{\sqrt{d_{i_0}}+\sqrt{a_{i_0}}}{\Delta\sqrt{d_{i_0}}}d_{i_k}
	 = \frac{\sqrt{a_{i_0}}+\sqrt{d_{i_0}}}{\Delta \sqrt{d_{i_0}}}\delta^S = l_+\delta^S,
	\end{align*}
where {the inequality follows from the fact that
$
\frac{\sqrt{a}+\sqrt{a-\Delta}}{\sqrt{a-\Delta}} 
= 1+\sqrt{1+\frac{\Delta}{a-\Delta}}
$
is decreasing on $a$ for $a>\Delta$,} and the second last equality holds because $\delta^S = \sum^{k-1}_{j=1}a_{i_j} + d_{i_k}$.
%\footnote{\blue{JP: Should the comment about $g(a)$ be moved here?}}
%\footnote{orange{XY: changed}}
	
On the other hand, if $a_{i_k}=\Delta$, then we have $d_{i_k}=d_{i_j}=0$ for any $i_j\in S$. 
Thus
\begin{align*}
\phi(\delta^S) &= \sum^{k-1}_{j=1}\frac{a_{i_j}+\sqrt{a_{i_j}d_{i_j}}}{\Delta} + \frac{d_{i_k}+\sqrt{a_{i_k}d_{i_k}}}{\Delta}
	=  \frac{1}{\Delta}\delta^S \leq l_+\delta^S.\tag*{\qed}
\end{align*}
% \green{	
% We hence verify for the second case. 
% If on the other hand $a_{\i0}$ exists, we clearly have
% 	$$
% 	\frac{\sqrt{a_{\i0}}+\sqrt{d_{\i0}}}{\Delta \sqrt{d_{\i0}}}\delta^S\geq  \frac{1}{\Delta}\delta^S.
% 	$$
% }
% \orange{We hence verify for $I^>=\emptyset$. If $I^>\neq \emptyset$, we clearly have
% 	$$
% 	\phi(\delta^S) = \frac{1}{\Delta}\delta^S\leq \frac{\sqrt{a_{\i0}}+\sqrt{d_{\i0}}}{\Delta \sqrt{d_{\i0}}}\delta^S.
% 	$$
% }
% \footnote{\green{Not sure what the above means...}}
% \footnote{\orange{XY: Proposed deletion and Change}}
% The last part of the proof establishes the optimality of the linear coefficient. 
% When $I^> \neq \emptyset$, we use $x'_{\i0}=y'_{\i0}=0$ and $x'_{j}=y'_{j} = 1$ for $j\neq \i0$, with $\delta' = a_{\i0}-\Delta = d_{\i0}$.
% Thus 
% $$\phi(\delta')
% \geq \frac{\sqrt{a_{\i0}}}{\sqrt{a_{\i0}}-\sqrt{d_{\i0}}} - 1
% = \frac{\sqrt{a_{\i0}}+\sqrt{d_{\i0}}}{\Delta \sqrt{d_{\i0}}}\delta'.$$
% When $I^>=\emptyset$, we use $x'_{1}=x'_{2}=y'_{1}=y'_{2}=0$ and $x'_{j}=y'_{j} = 1$ for $j\neq 1,2$, with $\delta' = \Delta$. 
% Thus
% $\phi(\delta')\geq 1 = \delta'/\Delta.$
% This proves the optimality of the linear coefficient.
% \qed
\end{proof}

%We denote the above result for $\delta \geq 0$ as $\phi(\delta)\leq l_+ \delta$. 
Next, we derive an over-approximation of $\phi(\delta)$ when $\delta \leq 0$. 
\begin{lemma}\label{lm:boundneg} 
Define $l_-:=\frac{1}{\Delta}$. 
For $\delta \leq 0$, we have
% {\red{
\begin{equation*}
\phi(\delta) = \left\{\begin{aligned} 
&-\infty & \delta &< -\Delta \\
&\frac{ \sqrt{a_n - \Delta} - \sqrt{a_n - \Delta - \delta} }{\sqrt{a_n} - \sqrt{a_n  - \Delta}}  &  - \Delta\leq \delta &\leq 0.
\end{aligned}\right .
\end{equation*}
% }}
% \begin{eqnarray*}
% \phi(\delta) = \left\{\begin{array}{cl} \frac{ \sqrt{a_n - \Delta} - \sqrt{a_n - \Delta - \delta} }{\sqrt{a_n} - \sqrt{a_n  - \Delta}}\  &  - \Delta\leq \delta \leq0\\
% -\infty\ & \delta < -\Delta
% \end{array}\right .
% \end{eqnarray*}
Further, $\phi(\delta) \le l_- \delta$ for $\delta \in [-\Delta,0]$.
%\footnote{\red{JP: Should we make the assumption that $a_i$ are sorted here too more explicit ... or should we say earlier that it applies through the rest of the section?}}
%\footnote{\orange{XY: I would suggest put it earlier for the whole section.}}
% \begin{eqnarray*}
% \phi(\delta) \leq \left\{\begin{aligned} &  \frac{1}{\Delta}\delta &  - \Delta\leq \delta \leq0\\
% & -\infty & \delta < -\Delta
% \end{aligned}\right .
% \end{eqnarray*}
%\begin{equation*}
%\phi(\delta) \leq \left\{\begin{aligned} &\frac{1}{\Delta} \delta &  - \Delta\leq \delta &\leq0\\
%&-\infty & \delta &< -\Delta.
%\end{aligned}\right .
%\end{equation*}
% \begin{eqnarray*}
% \phi(\delta) \leq \left\{\begin{array}{cl} \frac{1}{\Delta} \delta &  - \Delta\leq \delta \leq0\\
% -\infty & \delta < -\Delta
% \end{array}\right .
% \end{eqnarray*}

\end{lemma}
\begin{proof}
When $\delta < -\Delta$, $\phi(\delta) = -\infty$ as the right-hand-side of the problem defining $\phi$ is larger than $\sum_{i=1}^n a_i$. 
Consider therefore the case when $0\geq \delta \geq - \Delta$. 
There exists an optimal solution $(x^*,y^*)$ of the problem defining $\phi(\delta)$ that is such that $x^*_i = y^*_i = 1$ for all $i \in [n]\setminus \{j\}$ for some $j \in [n]$. 
Further, $a_jx^*_jy_j^* = a_j -\Delta-\delta$. 
We obtain
\begin{eqnarray*}\phi(\delta)
&=&  \textup{max}_j \left[\frac{\sqrt{a_j } - \sqrt{a_j - \Delta - \delta}}{\sqrt{a_j} - \sqrt{a_j - \Delta}}\right]-1 
= \frac{\sqrt{a_n } - \sqrt{a_n - \Delta - \delta}}{\sqrt{a_n} - \sqrt{a_n - \Delta}}-1,
\end{eqnarray*}
where the last step follows from Lemma~\ref{lm:decreasing}.
Finally, observe that $\phi(\delta)$ is convex in $\delta$.  
Therefore, by taking a linear inequality tight at $\delta = 0$ and $\delta = -\Delta$,  we obtain that $\phi(\delta) \leq \delta/\Delta = l_- \delta$ since $\phi(0)=0$ and $\phi(-\Delta)=-1$. 
\qed
\end{proof}

By combining Lemmas~\ref{lm:boundpos} and \ref{lm:boundneg}, we obtain the following over-approximation of $\phi$:
\begin{eqnarray*}
\phi(\delta) \leq \tilde{\psi}(\delta):=\left\{ 
\begin{array}{llrclcl}  
-\infty &\ & &&\delta&\leq& -\Delta \\
l_- \delta &\ & -\Delta &\leq &\delta &\leq& 0 \\
l_+ \delta &\ & 0 &\leq &\delta.&& 
\end{array}\right.
\end{eqnarray*}
% \footnote{\blue{JP: Should add here a discussion of why $l_+ \ge l_-$}}
%\orange{While note that $l_+\geq l_-$,} 
Note that the function $\tilde{\psi}$ is not subadditive.
Lemma~\ref{lm:subadditive} describes a subadditive function that upper bounds $\tilde{\psi}$, thus giving a subadditive upper bound of $\phi$.

\begin{lemma}\label{lm:subadditive}
{It holds that $l_+ \ge l_->0$.} 
Further, the function
\begin{equation*}
\psi(\delta) = \left\{ 
\begin{aligned}  
&l_+(\delta+\Delta) - l_-\Delta & \delta&\leq -\Delta \\
&l_- \delta & -\Delta \leq \delta &\leq 0 \\
&l_+ \delta &  0\leq \delta&
\end{aligned}
\right.
\end{equation*}
is subadditive. 
%In addition, any subadditive function $\dot{\psi}\geq \tilde{\psi}$ with $\dot{\psi}(\delta) = \tilde{\psi}(\delta)$ for $\delta>0$ must satisfy $\dot{\psi}\geq \psi$.
\end{lemma}
\begin{proof}
%We first prove the second part. For any subadditive function $\hat{\psi}\geq \tilde{\psi}$, it is clear that for any $\delta\leq -\Delta$, 
%$\dot{\psi}(\delta)\geq \dot{\psi}(-\Delta) - \dot{\psi}(-\Delta-\delta)\geq -l_-\Delta + l_+(\delta + \Delta) = \psi(\delta)$, which proves the minimality of $\psi$.
% We next show that $\psi$ is indeed subadditive, \textit{i.e.}, $\psi(u) + \psi(v) \geq \psi(u+ v)$, $\forall u, v\in \Real$. 
Define $\mathring{\psi}(\delta):=l_- \delta$ when $\delta \le 0$ and $\mathring{\psi}(\delta):=l_+ \delta$ when $\delta \ge 0$.
%\begin{equation*}
%\mathring{\psi}(\delta) := \left\{ 
%\begin{aligned} 
%&l_- \delta & \delta &\leq 0 \\
%&l_+ \delta & \delta &\geq 0, 
%\end{aligned}
%\right.
%\end{equation*}
Function $\mathring{\psi}(\delta)$ satisfies $\mathring{\psi}\geq \psi$ and is subadditive since it is straightforward to verify that $l_+ \ge l_- >0$. 
Thus, for $u,v$, such that $u,v, u + v \in [-\Delta, +\infty)$, we already have that $\psi(u)+\psi(v)\geq \psi(u+v)$. 
It remains to consider the cases where at least one of $u$, $v$ or $u+v$ belongs to $(-\infty,-\Delta]$.
We do so by considering the possible values of $u+v$ and by assuming without loss of generality that $u\geq v$.
We use the fact that for $\delta\leq 0$, $\psi(\delta)=\min\{l_+(\delta+\Delta)-l_-\Delta, l_-\delta\}\geq l_+\delta$.
There are three cases to consider. 
First assume that  $u + v \geq 0$. 
In this case,
$\psi(u) + \psi(v) - \psi(u + v)  \geq  l_+ u + l_+v - l_+(u+v) = 0$.
Second assume that $-\Delta \leq u + v \leq 0$. 
In this case, $v \leq - \Delta$ and $u\geq 0$ so that
$\psi(u) + \psi(v) - \psi(u + v)  
=  l_+ u + l_+(v+\Delta)-l_- \Delta - l_-(u+v) 
= (l_+-l_-)(u+v+\Delta)  \geq  0.$
Third assume that $u + v \leq -\Delta$. 
There are two subcases. 
If $v \leq -\Delta$, we have
$
\psi(u) + (\psi(v) - \psi(u + v))  \geq   l_+ u + (l_+(v+\Delta) - l_+(u+v+\Delta))=0.$
If $v \geq -\Delta$, then $0\geq u\geq v\geq -\Delta$. 
Therefore
$\psi(u) + \psi(v) - \psi(u + v)  \geq l_-u + l_- v -  l_-(u+v) = 0.$
\qed
\end{proof}

\begin{proof}[Theorem~\ref{thm:upperbound}]
Combining Lemmas~\ref{lm:boundpos}, ~\ref{lm:boundneg}, and~\ref{lm:subadditive} yields Theorem~\ref{thm:upperbound}.
\end{proof}

\section{Proof of Theorem~\ref{thm:lifted}}
\label{section:lifted}
%\footnote{\red{JP: changed the title; used to be "Lifting from subadditive upper bound"}}

%\footnote{JP: I removed this: ``In this section, independently of sequence we lift (\ref{eq:bilincoverineq}) following the subadditive upper bound (\ref{eq:upperbound})."}
%\ThmLifted*

\begin{proof}[Theorem~\ref{thm:lifted}]
%\footnote{\red{JP: I changed this a bit. Looking back, we argue a bit more than what the theorem states, since we prove "tightness" as opposed to just validity}}
Following Theorems~\ref{thm:valid} and \ref{thm:upperbound}, it is sufficient to show that $\gamma_i(x,y)\geq \psi(a_ixy)$ for $i\in J_0$ and $\gamma_i(x,y)\geq \psi(a_i(xy-1))$ for $i\in J_1$, where $\psi$ is the subadditive over-approximation of $\phi$ derived in Theorem~\ref{thm:upperbound}. 
%For simplicity, we consider $\sum_{i\in I}a_ix_iy_i \pm axy\geq d^\Lambda$ where $a>0$ and $xy$ is fixed at $0$ or $1$. 
We discuss the possible cases.
\begin{enumerate}[label={(\roman*)}]
\item 
%$\sum a_ix_iy_i + axy\geq d$ while $x,y$ lifted from $0$. 
Assume $i \in J_0^+$. 
We must find $\gamma_i(x,y)\geq \psi(a_ixy) = l_+a_ixy$ for $(x,y) \in [0,1]^2$ where the equality holds as $a_i>0$. 
As $\min\{x, y\}\geq xy$ is the best concave upper bound for $(x, y)\in [0,1]^2$, we choose $\gamma_i(x,y) = l_+a_i\min\{x,y\}$.

\item 
Assume $i \in J_1^-$.
%$\sum a_ix_iy_i - a(xy-1)\geq d$ while $x,y$ lifted from $1$. 
We must find $\gamma_i(x,y) \geq \psi(a_i(xy-1)) = l_+(a_ixy-a_i) = l_+a_i(xy-1)$  for $(x,y) \in [0,1]^2$ where the equality holds since $a_i<0$. 
As $\max\{x+y-1, 0\}\leq xy$ is the best convex lower bound for $(x, y) \in [0,1]^2$, we choose
$\gamma_i(x,y) = l_+a_i (\max\{x+y-1, 0\}-1)=-l_+a_i\min\{2-x-y,1\}$.

\item 
\label{lift3}
Assume $i \in J_0^-$.
%$\sum a_ix_iy_i - axy\geq d$ while $x,y$ lifted from $0$. 
We must find $\gamma_i(x,y)\geq \psi(a_ixy)$  for $(x,y) \in [0,1]^2$. 
As $a_i<0$,  $\psi(a_ixy) = \min\{l_-a_ixy, l_+a_ixy + l_+\Delta-1\}\leq 
\min\{l_-a_i(x+y-1), l_+a_i(x+y-1)+l_+\Delta-1, 0\} := \gamma _i(x,y)$.
%since $\max\{x+y-1,0\}$ is the concave envelope of $xy$ over $[0,1]^2$,
% Take $(x,y)$ as $(0,0)$, $(1,0)$, $(0,1)$, $(1,1)$, and in case that $a>\Delta$, also $(1,\Delta/a)$, $(\Delta/a, 1)$. We can verify that the above concave upper bound is optimal, so we pick
% $$
% \gamma (x,y) = \min\{l_-a(1-x-y), l_+a(1-x-y)+l_+\Delta-1, 0\}.
% $$

% \item $\sum a_ix_iy_i - axy\geq d$ while $x,y$ lifted from $0$. In this case, we would like to find $\gamma(x,y)\geq \psi(-axy)$. Note that
% $\psi(-axy) = \min\{-l_-axy, -l_+axy + l_+\Delta-1\}$. Thus
% \begin{eqnarray*}
% \psi(-axy)&\leq&0 \\
% \psi(-axy)&\leq&-l_-axy\leq l_-a(1-x-y)\\
% \psi(-axy)&\leq&-l_+axy + l_+\Delta-1\leq l_+a(1-x-y)+l_+\Delta-1
% \end{eqnarray*}
% or
% $$
% \min\{l_-a(1-x-y), l_+a(1-x-y)+l_+\Delta-1, 0\}\geq \psi(-axy).
% $$
% Take $(x,y)$ as $(0,0)$, $(1,0)$, $(0,1)$, $(1,1)$, and in case that $a>\Delta$, also $(1,\Delta/a)$, $(\Delta/a, 1)$. We can verify that the above concave upper bound is optimal, so we pick
% $$
% \gamma (x,y) = \min\{l_-a(1-x-y), l_+a(1-x-y)+l_+\Delta-1, 0\}.
% $$

\item 
Assume $i \in J_1^+$.
%$\sum a_ix_iy_i + a(xy-1)\geq d$ while $x,y$ lifted from $1$. 
In this case, we must find $\gamma_i(x,y)\geq \psi(a_ixy-a_i)$ 
for $(x,y) \in [0,1]^2$. 
Since $a_i>0$,  $\psi(a_ixy-a_i) = \min\{l_-a_i(xy-1), l_+a_i(xy-1)+l_+\Delta-1\}.$ 
Similar to \ref{lift3}, we have $
\psi(a_ixy-a_i)\leq l_-a_i(\min\{x, y\}-1) =: \tilde{h}(x,y)$, and $
\psi(a_ixy-a_i)\leq l_+a_i(\min\{x, y\}-1)+l_+\Delta-1=: \tilde{g}(x,y)$.
%This is already optimal for $a\leq\Delta$ or for $l_- = l_+$ (i.e. $I^>=\emptyset$), and a good approximation for $a\in(\Delta, a_{\i0})$ in case that $I^>\neq\emptyset$. In these cases, we choose
Thus, $\gamma(x,y) = \min\{\tilde{h}(x,y), \tilde{g}(x,y)\}$ is a concave upper bound of $\psi$. 

Next we improve this upper bound when $a_i\geq a_{\i0} > \Delta$. 
As $g$ and $h$ (defined in Theorem~\ref{thm:lifted}) are concave, it remains to show the following:
\begin{claim} 
For $a_i \geq a_{\i0} > \Delta$, $\min\{g(x,y),h(x,y)\} \geq \psi(a_ixy -a_i)$ for $(x,y) \in [0, 1]^2$.% and $h(x,y) \geq \psi(axy -a)$ for $x,y \in [0, 1]$
\end{claim}
Observe that
\begin{eqnarray*}
\psi(a_ixy - a_i) = \left\{ 
\begin{array}{lll}  
l_+a_i((\sqrt{xy})^2-1) + l_+\Delta-1 \ &\textrm{if}& 0 \leq \sqrt{xy} \leq \sqrt{1 - \frac{\Delta}{a_i}}\\
l_-a_i ((\sqrt{xy})^2 - 1) &\textrm{if}& \sqrt{1 - \frac{\Delta}{a_i}} \leq \sqrt{xy} \leq 1.
\end{array} \right.
\end{eqnarray*}
% Now thinking of the two functions in (\ref{eq:t}) as convex functions of ``$\sqrt{xy}$". So we can upper bound them, by linearizing them (\textit{i.e.}, compute the function value at the end points and construct a linear tight at these points). We consider the `first part' of the function and arrive at
% \begin{eqnarray}\label{eq:defng}
% g(x,y) =  l_+\sqrt{a-\Delta}\sqrt{a}\sqrt{xy} - l_+(a-\Delta)-1
% \end{eqnarray}
Consider first the function $g_i(x,y)= \sqrt{a_i - \Delta}\sqrt{a_i}l_+\sqrt{xy} -l_+(a_i-\Delta) -1$:
\begin{itemize}
\item $\sqrt{xy} \in [0, \sqrt{1 - \frac{\Delta}{a_i}}]$:
$g_i(x,y) \geq -1+ (a_i(\sqrt{xy})^2 - a_i + \Delta) l_+  = \psi(a_ixy -a_i).$
\item $\sqrt{xy} \in [\sqrt{1 - \frac{\Delta}{a_i}}, 1]$: 
% we want to prove that:$
% l_+\sqrt{a-\Delta}\sqrt{a}\sqrt{xy} - l_+(a-\Delta)-1  \geq l_-a (xy - 1), $
% in the interval $xy \in \left[1 - \Delta/a, 1\right]$ which
we simply prove
$
\hat{g}_i(t) :=   l_+\sqrt{a_i-\Delta}\sqrt{a_i}\sqrt{t} - l_+(a_i-\Delta)-1  \geq l_-a_i (t - 1)  =:f_i(t), 
$
for $t \in [1 - \Delta/a_i, 1]$.  
To this end, we verify:  
(i) $\hat{g}_i(1- \frac{\Delta}{a_i}) = f_i(1 - \frac{\Delta}{a_i})$ and 
(ii) $\hat{g}_i(1) \geq f_i(1)$. 
This is sufficient since $\hat{g}_i$ is a concave function and $f_i$ is a linear function. The proof of (i) is straightforward. 
To prove (ii) observe that $\hat{g}_i(1) = l_+\sqrt{a_i-\Delta}\sqrt{a_i}- l_+(a_i-\Delta)-1  \geq f_i(1) = l_-a_i (1 - 1)  = 0$ is equivalent to  verifying $l_+\geq \frac{1}{\sqrt{a_i -\Delta} (\sqrt{a_i} - \sqrt{a_i -\Delta}) }$ or equivalently $\frac{\sqrt{a_{\i0}} + \sqrt{a_{\i0} - \Delta}}{\sqrt{a_{\i0} - \Delta} }  \geq \frac{\sqrt{a_i} + \sqrt{a_i -\Delta}}{\sqrt{a_i -\Delta}  }$  which holds since $a_i\geq a_{\i0}$.
% \begin{eqnarray*}
% &&l_+\sqrt{a-\Delta}\sqrt{a}- l_+(a-\Delta)-1  \geq l_-a (1 - 1)  = 0\\
% \Leftrightarrow && l_+\geq \frac{1}{\sqrt{a -\Delta} (\sqrt{a} - \sqrt{a -\Delta}) }\\
% \Leftrightarrow && \frac{\sqrt{a_{\i0}} - \sqrt{a_{\i0} - \Delta}}{\sqrt{a_{\i0} - \Delta} }  \geq \frac{\sqrt{a} + \sqrt{a -\Delta}}{\sqrt{a -\Delta}  },
% \end{eqnarray*}
% where the last statement is  true since $a\geq a_{i_0}$.
\end{itemize}

Consider second the function        
$h_i(x,y)= \frac{\sqrt{a_i}}{\sqrt{a_i} - \sqrt{d_i}}(\sqrt{xy} - 1)$:
\begin{itemize}
\item $\sqrt{xy} \in [\sqrt{1 - \frac{\Delta}{a_i}}, 1]$: 
by construction, $h_i(x,y) \geq l_-a_i(xy - 1)\geq \psi(a_ixy-a_i)$. 
\item $\sqrt{xy} \in [0, \sqrt{1 - \frac{\Delta}{a_i}}]$: 
first observe that $\hat{h}_i(t^*) = f_i(t^*)$ for $t^* =  \sqrt{1 - \frac{\Delta}{a_i}}$, where $f_i(t) := l_+a_i(t-1)+l_+\Delta-1 $. 
Since $h_i(x,y)$ is concave it is sufficient to verify that $\hat{h}_i(0) \geq f_i(0)$.
This condition holds as
$
\frac{\sqrt{a_i}}{\sqrt{a_i} - \sqrt{a_i - \Delta}} \leq 1 + (a_i - \Delta)l_+$ which is equivalent to  $l_+ \geq \frac{1}{\sqrt{a_i -\Delta} (\sqrt{a_i} - \sqrt{a_i -\Delta})}.
$
\qed
\end{itemize}
% So we prove the claim that $\psi \leq \min\{g,h\}$.

% 
% Finally, let us compare $g$ with the ``similar" function $\tilde{g}(x,y)$ which also is concave and upper bounds $\psi$. These two functions are not comparable. In particular:
% \begin{itemize}
% \item Consider a point of the form $(\hat{x}, \hat{x})$ where $1 >\hat{x} > 0$. In this case, 
% $$g(\hat{x}, \hat{x}) - \tilde{g}(\hat{x}, \hat{x}) = \sqrt{a} \sqrt{a - \Delta}\hat{x} - a\hat{x} < 0.$$
% \item Consider a point of the form $(\hat{x}, 1)$ where $1 > \hat{x} > 0$. In this case, 
% $$ g(\hat{x}, 1) - \tilde{g}(\hat{x}, 1) = \sqrt{a} \sqrt{a - \Delta}\sqrt{x} - a\hat{x} > 0,$$
% which holds for all $\hat{x} < \sqrt{\frac{a - \Delta}{a}}$.
% \end{itemize}
 
% Similarly the functions $h$ and $\tilde{h}(x,y)$ are not comparable. Therefore, for $a>a_{\i0}$, we choose
% \begin{align*}
% \gamma(x,y) = \min\{\tilde{g}(x,y), \tilde{h}(x,y), g(x,y), h(x,y)\}.\tag*{\qed}
% \end{align*}
\end{enumerate}

\end{proof}

\bibliographystyle{plain}      % basic style, author-year citations
% \bibliographystyle{spphys}       % APS-like style for physics
%\bibliography{CBEEGPaper1}   % name your BibTeX data base
% \bibliographystyle{IEEEtran}
\bibliography{mybibliography}

\end{document}